\documentclass[11pt]{article}
\usepackage{amssymb, amsthm ,amsmath, dsfont,graphics,graphicx,pspicture,graphpap,picins}
\usepackage{color}
\usepackage{enumerate}

\usepackage{pst-node,pst-text,pst-3d}

\def\JARSbluJARS{{    }}

\def\Lip{{C_{Lip}}}
\def\newn{D}
\def\Ejpm{E_{j}}
\def\island{\mathcal{I}}

\def\phiphiphi{{\phi}}
\def\psipsipsi{{\Psi}}
\def\cubeDist{{d_\textrm{Cube}}}

\def\lec{\lesssim}
\def\gec{\gtrsim}

\newcommand{\ps}[1]{\left( #1 \right)}
\newcommand{\ck}[1]{\left\{ #1 \right\}}
\newcommand{\av}[1]{\left| #1 \right|}
\def\ve{\epsilon}

\def\d{\partial}
\def\avint_#1{\mathchoice
{\mathop{\vrule width 6pt height 3 pt depth -2.5pt
\kern -8.8pt \intop}\nolimits_{#1}}%
{\mathop{\vrule width 5pt height 3 pt depth -2.6pt
\kern -6.5pt \intop}\nolimits_{#1}}%
{\mathop{\vrule width 5pt height 3 pt depth -2.6pt
\kern -6pt \intop}\nolimits_{#1}}%
{\mathop{\vrule width 5pt height 3 pt depth -2.6pt \kern -6pt
\intop}\nolimits_{#1}}}
\newcommand{\isif}[1]{\left\{\begin{array}{cc} #1
\end{array}\right.}

\newcommand {\bt}{{\tilde{\beta}} }
\newcommand {\cM}{\mathbb{X}}

\newcommand{\Claim}{{\bf Claim}}

\newcommand {\side}{\textrm{side}}
\newcommand {\bR}{\mathbb {R}}

\newcommand {\bN}{\mathbb{N}}
\newcommand {\cA}{\mathcal{A}}
\newcommand {\cB}{\mathcal{B}}
\newcommand {\cE}{\mathcal {E}}
\newcommand {\cC}{\mathcal {C}}

\newcommand {\cF}{\mathcal {F}}
\newcommand {\cS}{\mathcal {S}}

\newcommand {\cW}{\mathcal {W}}
\newcommand {\cH}{\mathcal {H}}
\newcommand {\cK}{\mathcal {K}}

\newcommand{\diam}{\textrm{diam}}
\newcommand{\dist}{\textrm{dist}}
\newcommand{\ball}{\textrm{Ball}}

\newcommand{\cQ}{\mathcal{Q}}

\newcommand{\eqn}[1]{\eqref{e:#1}}
\def\one{\mathds{1}}

\def\Corner{\textrm{Corner}}

\def\Span{\textrm{Span}}

\newcommand{\starit}[1]{\stackrel{\star}{#1}}

\def\hitsE{\cQ}
\def\lowc{\mathcal{REST\!ART}} 
\def\highc{\mathcal{STOP}} 
\newcommand{\low}[1]{\check{#1}}

\def\RESTART{\lowc} 
\def\STOP{\highc} 

\newcommand{\bcubes}{\cE_{3}}

\theoremstyle{plain}
\newtheorem{theorem}{Theorem}[section]
\newtheorem*{thma}{Theorem I}
\newtheorem*{thmc}{Theorem II}
\newtheorem*{thmd}{Theorem III}
\newtheorem{lemma}[theorem]{Lemma}
\newtheorem{proposition}[theorem]{Proposition}
\newtheorem{corollary}[theorem]{Corollary}
\newtheorem{remark}[theorem]{Remark}
\newtheorem{definition}[theorem]{Definition}
\numberwithin{equation}{section}

\begin{document}

\title{Hard Sard: Quantitative Implicit Function and  Extension Theorems for Lipschitz Maps}
\author{Jonas Azzam and Raanan Schul}
\date{}
\maketitle

\begin{abstract}
We prove a global implicit function theorem.  
In particular we show that any Lipschitz map $f:\bR^n\times \bR^m\to\bR^n$ (with $n$-dim. image) can be precomposed with a bi-Lipschitz map 
$\bar{g}:\bR^n\times \bR^m\to \bR^n\times \bR^m$ such that  
$f\circ \bar{g}$ will satisfy, 
when we restrict to a large portion of the domain $E\subset \bR^n\times \bR^m$,  
that $f\circ \bar{g}$ is bi-Lipschitz in the first coordinate, and constant in the second coordinate.
Geometrically speaking, the map $\bar{g}$ distorts $\bR^{n+m}$ in a controlled manner so that the fibers of $f$ are straightened out.  Furthermore, our results stay valid when the target space is replaced by {\bf any metric space}.
A main point is that our results are quantitative: 
the size of the set $E$ on which behavior is good is a significant part of the discussion. 
Our estimates are motivated by examples such as Kaufman's 1979 construction of a $C^1$  map from $[0,1]^3$ onto $[0,1]^2$ with rank $\leq 1$ everywhere. 

On route we prove an extension theorem which is of independent interest.  
We show that for any $D\geq n$, any Lipschitz function   $f:[0,1]^n\to \bR^D$ gives rise to a large (in an appropriate sense) subset $E\subset [0,1]^n$ such that $f|_E$ is bi-Lipschitz and may be extended to a bi-Lipschitz function defined on {\bf all} of $\bR^n$. This extends results of P. Jones and G. David, from 1988. As a simple corollary, we show that $n$-dimensional Ahlfors-David regular spaces lying in $\bR^{D}$ having big pieces of bi-Lipschitz images also have big pieces of big pieces of Lipschitz graphs in $\bR^{D}$. This was previously known only for $D\geq 2n+1$ by a result of G. David and S. Semmes. 



\end{abstract}

\newpage

\tableofcontents

\bigskip

\section{Introduction}

For quantities $A$ and $B$, we write $A\lec B$ if there is a constant $C$ (independent of the values $A$ and $B$) such that $A\leq CB$, and $A\sim B$ if
\[\frac{1}{C}B\leq A\leq CB.\]
We will also write $A\sim_{n} B$ or $A\lec_{n}$ if the implied constant $C$ depends on $n$.

Let $\cM$ be a metric space.
Define, for a set  $A\subset \cM$ and $0<t\leq \infty$,
\begin{equation}
\cH_{t}^{n}(A)=c_{n}\inf\ck{\sum_{j} (r_j)^n\ :\ r_j<t}\,,
\end{equation}
where the infimum is over all covers of $A$ of the form $\cup_j\ball(x_j,r_j)$ and $c_{n}$ is the $n$-dimensional volume of the $n$-dimensional sphere.
The $n$-Hausdorff content of $A$ is defined to be $\cH^n_\infty(A)$.
Also define 
$$\cH^n(A)=\lim\limits_{t\to 0}\cH^n_t(A)\,,$$
which is called the  (spherical) Hausdorff measure. 
We will also write $|A|$ to denote the Hausdorff measure when $n$ is clear.

A function $f:[0,1]^{n}\rightarrow\cM$ is $L$-Lipschitz if for all $x,y\in [0,1]^{n}$,
\[\dist(f(x),f(y))\leq L|x-y|.\]
If in addition to being $L$-Lipschitz, $f$ also satisfies 
\[\dist(f(x),f(y))|\geq \frac{1}{L}|x-y|\,,\]
then $f$ is called $L$-bi-Lipschitz.
We will also say that $f$ is $(l,L)$-bi-Lipschitz if 
\[l|x-y|\leq  \dist(f(x),f(y))\leq L|x-y|.\]

\subsection{Motivation}

A simple example of a Lipschitz map $f:\bR^3\to\bR$ is the map
$f(x,y,z)=x$.
Besides being Lipschitz, this map enjoys other nice properties.
In particular,
\begin{itemize}
\item[(i)]
the preimage of every point is a $2$-plane
\item[(ii)]
$f$ is bi-Lipschitz along any line of the form $\{(x,y_0,z_0): x\in \bR\}$.
\end{itemize}
The goal of this paper is to show that, in a some quantitative manner,
{\bf ALL} Lipschitz functions ($\bR^{n+m}\to \cM$, where $\cM$ is a metric space, and the image is  $n$-dimensional) enjoy properties akin to (i) and (ii)
above.  
This is of course not true in the most naive interpretation. 
If however we allow precomposing with a bi-Lipschitz map 
($\bR^{n+m}\to \bR^{n+m}$)
then we may get analogues of (i) and (ii) above when we restrict to a large subset.
We make this precise in
Theorem~{{I}}.
One may view these results as a global and quantitative version of the implicit function theorem.
Even for the case $\cM=\bR^n$ our results are new.

Similar  qualitative ideas existed for some time.  
Recall two well known variants  of Sard's Theorem.
\begin{theorem}[\cite{Federer} Ch. 3]
\label{t:sard-1}
If $f:\bR^n\to \bR^D$ is Lipschitz, then there is a Borel set $B\subset\bR^n$ such that 
$f|_B$ is univalent, and $f(B)=f(\bR^n)$ up to $\cH^n$-measure zero.  
Furthermore, one may write $B=\cup E_i$ such that $f|_{E_i}$ is $2^i$- bi-Lipschitz.
\end{theorem}
\begin{theorem}[\cite{Federer} Ch. 3]
\label{t:sard-2}If $f:\bR^{n+m}\to \bR^n$ is Lipschitz, then for $\cH^n$-almost every $y \in f(\bR^{n+m})$, the set $f^{-1}(y)$ is countably $m$-rectifiable.
\end{theorem}

The case where $\bR^D$ is replaced by a general metric space was first investigated by Kirchheim \cite{Kirchheim-sards}.  See also \cite{magnani2010-area-formula, reichel2009-coarea, AK2000-rectifiable, Karmanova2008}.

Theorem \ref{t:sard-1}
was made quantitative in
\cite{David88, Jones-lip-bilip}, and has since been generalized and modified \cite{DS-quantitative-rectifiability, Schul-lip-bilip, Meyerson-lip-bilip}. 
The version in \cite{Schul-lip-bilip} reads as follows.

\begin{theorem}\label{t:lip-bilip}
Let $0<\kappa<1$ and $n\geq 1$ be  given.
There are universal constants  $M=M(\kappa,n)$, and $c_1=c_1(n)$  such that the following statements hold.
Let $\cM$ be any metric space and 
let $f:[0,1]^n\to \cM$ be a 1-Lipschitz function.
Then there are sets 
$E_1,...,E_M\subset [0,1]^n$ so that for $1\leq i\leq M$,  
$x,y\in E_i$ we have
\begin{equation*}
\kappa|x-y| \leq \dist(f(x),f(y))\leq |x-y|\,,
\end{equation*}
and 
\begin{gather}
\cH^n_\infty(f([0,1]^n\setminus(E_1\cup...\cup E_M)))\leq c_1 \kappa\,.
\label{e:c_1}
\end{gather}
($\cH^n_\infty$ is the $n$-dimensional Hausdorff content).
\end{theorem}

Below, we prove Theorem~{{I}}, 
a quantitative, global, implicit function theorem.  
In particular we show that for any Lipschitz map $f:\bR^n\times \bR^m\to\cM$, where $\cH^n(\cM)<\infty$,  there corresponds a bi-Lipschitz homeomorphism 
$g:\bR^n\times \bR^m\to \bR^n\times \bR^m$ such that
$F:=f\circ g^{-1}$
satisfies, 
when we restrict to 
a large portion of the domain
$E\subset \bR^n\times \bR^m$,  
that $F$ is bi-Lipschitz in the first coordinate, and constant in the second coordinate.

On route we prove  a second result, Theorem~{{II}}. This is an extension theorem which, loosely speaking, says that given a Lipschitz function from one Euclidean space to another, one can decompose most of the domain into a finite (controlled) number of sets $E_i$, such that $f|_{E_i}$  can be extended to a bi-Lipschitz function defined on the {\bf whole} original cube. 
See Theorem~{{II}} for a precise statement.

\subsection{Statements of main results}


\begin{figure}[h]
\begin{picture}(100,200)(0,0)
\put(67,36){$g(E)$}
\put(215,65){$E$}
\put(130,50){\vector(-1,0){30}}
\put(115,60){$g$}
\put(240,50){\vector(1,0){30}}
\put(250,60){$f$}
\put(300,80){$f([0,1]^{2})$}
		\put(50,100){\rnode{tail}}
		\put(300,100){\rnode{head}}
		\put(300,100){\vector(2,-1){3}}
		\nccurve[linecolor=black,linewidth=0.5pt,angleA=150,angleB=45]{<-}{head}{tail}
		\put(140,125){ $F$}

\scalebox{0.36}{\includegraphics{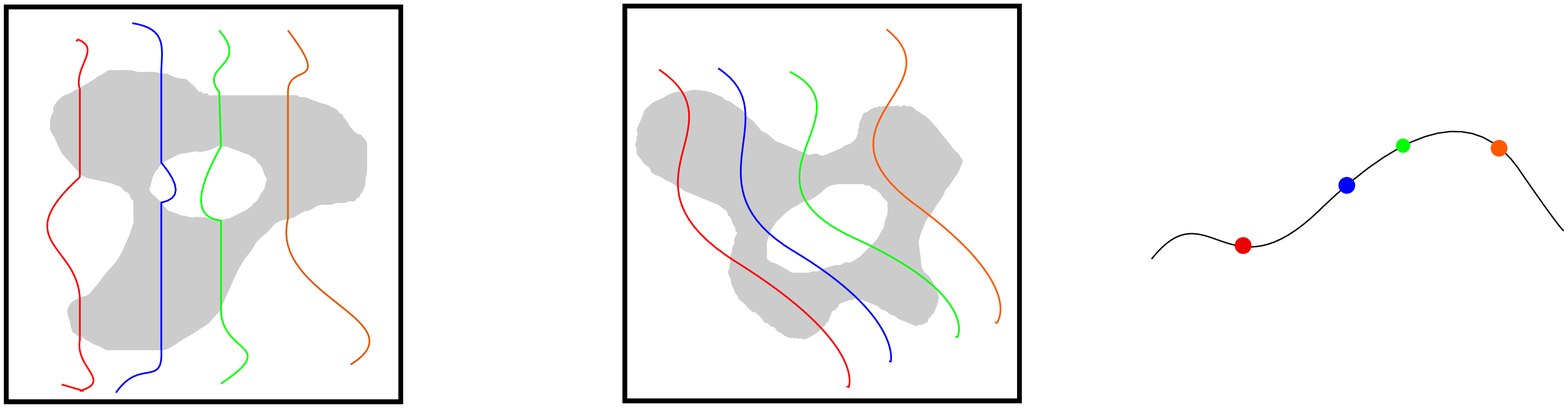}}
\end{picture}
\caption{In the center square above, we have four rectifiable fibers that are preimages of four points in the image of $f$. Theorem~{{I}} tells us that there is a large region in the domain, $E$ (denoted by the shaded area), so that the portions of these fibers that intersect $E$ are sent to subsets of straight lines under $g$.}
\label{f:figure-1}
\end{figure}
%


Let $A\subset \bR^{n+m}$ be a subset of $\bR^{n+m}$, and  
$f:A\rightarrow \cM$ a Lipschitz function. 
We define the $(n,m)$-Hausdorff content of $(f,A)$ as
\begin{equation}
\cH^{n,m}_{\infty}(f,A):=\inf \sum_{j} \cH^{n}_{\infty}(f(Q_{j}))\side(Q_{j})^{m}\,. 
\label{e:f-content-intro}
\end{equation}
where the infimum is over all measure theoretic partitions of $A$ into disjoint open cubes $Q_{j}$.
We discuss the $(n,m)$-Hausdorff content in Section \ref{ss:Hausdorff-cont-note}.

We show the following, which is  illustrated by Figure \ref{f:figure-1}.
\begin{thma}[Quantitative Implicit Function Theorem]
Suppose 
$$f: \bR^{n+m}\rightarrow \cM$$ 
is a $1$-Lipschitz function  into a metric space, and  
\begin{equation}\label{e:artificial-upper-bound}
0<\cH^n(f([0,1]^{n+m}))\leq 1\,. 
\end{equation}
Suppose
$$0<\delta\leq \cH^{n,m}_\infty(f,[0,1]^{n+m})\,.$$
Then there 
are
constants
$\Lip>1$ 
and $\eta>0$, depending on $n,m$ and $\delta$,
a set $E\subset [0,1]^{n+m}$,
and a homeomorphism $g:\bR^{n+m}\rightarrow \bR^{n+m}$, such that 
if $F=f\circ g^{-1}$ then 
the following four properties hold.
\begin{enumerate}[(i)]
\item $\cH^{n+m}(E)\geq \eta$.
 \item
 $g$ is  $\Lip$-bi-Lipschitz. 
 
 \item
For $(x,y)\in \bR^{n}\times\bR^{m}$ if $(x,y) \in g(E)$, then 
$$
F^{-1}\big(F(x,y)\big) \cap g(E)\subseteq 
g(E)\cap \big(\{x\}\times\bR^{m}\big).$$ 
 \item
For all $y\in\bR^{m}$, $F|_{(\bR^{n}\times\{y\})\cap g(E)}$ is $\Lip$-bi-Lipschitz.
\end{enumerate}
\end{thma}
Note that $\Lip$ and $\eta$ do not depend on $f$ or the metric space $\cM$ besides the stated dependencies on $n,m$ and $\delta$. We reword (iii) and (iv):
\begin{itemize}
\item  inside $g(E)$, $F$ is independent of $y$, and
\item  for fixed $y$, the function $F$  is bi-Lipschitz in $x$.
\end{itemize}
The above  theorem is novel even in the case of $\cM=\bR^{n}$.

\begin{corollary}\label{c:lower-bd-content}
Assume $f$ is as in Theorem~{{I}}, 
There is a constant $\xi>0$ depending only on $n,m,\delta$
such that 
if
$g,E,F$  satisfy the conclusions of  Theorem~{{I}},
then 
$$\cH^{n,m}_\infty (f,E)> \xi\,.$$
\end{corollary}

{\JARSbluJARS    
\begin{remark}
\label{r:can't}
One might hope that there is a version of Theorem~{{I}} that resembles Theorem \ref{t:lip-bilip} in the sense that we may partition the domain of $f$ into subsets $E_{j}$ and bi-Lipschitz maps $g_{j}$ satisfying (ii)-(iv) {\textit and} their images 
exhaust all of  the image of $f$ except for a piece of arbitrarily small Hausdorff content, that is, \eqn{c_1} is satisfied. This is not possible. This will be evident after reading Section \ref{s:non-examples} below, where we demonstrate how to find, say, a Lipschitz map $f$ from $\bR^{4}$ onto a path connected purely 2-unrectifiable  set of positive and finite two dimensional measure. If a set $E$ of measure $\cH^{4}(E)=\eta$ and a bi-Lipschitz function $g$ satisfying (ii), (iii), and (iv) exist for $f$, then by Fubini's theorem, there is a $2$-plane $V$ in $\bR^{4}$ such that $\cH^{2} (V\cap g(E)) >0$. Since $f\circ g^{-1}$ is bi-Lipschitz on $V\cap g(E)$, $f\circ g^{-1}(V\cap g(E))$ is a rectifiable set with  $\cH^{2}(f\circ g^{-1}(V\cap g(E)))>0$. If we had a decomposition satisfying \eqn{c_1}, this would contradict the image of $f$ being purely unrectifiable.

The reason for this is that it is $\cH^{n,m}_{\infty}(f,[0,1]^{n+m})$, not the Hausdorff content of the image, that determines whether there are any nontrivial sets $E_{i}$ satisfying the conditions of the theorem. See Remark \ref{r:open} for further discussion.

\end{remark}
}


A corner stone in the proof of Theorem~{{I}} is  Theorem~{{II}} below. Loosely speaking, it says that a Lipschitz function
$f:\bR^{n}\rightarrow\bR^{\newn}$ whose image has large content, 
may be bi-Lipschitzly extended on a large subset of the domain. For the purpose of proving Theorem~{{I}}, however,
we will only use Theorem~{{II}} with the dimension $D=n$.

\begin{thmc}[Bi-Lipschitz Extension on Large Pieces]

Let $D\geq n$. Let $0<\kappa<1$ be given. 
There is a constant $M=M(\kappa,\newn)$ such that if
$f:\bR^{n}\rightarrow\bR^{\newn}$ is a $1$-Lipschitz function, then the following hold. 
\begin{itemize}
\item [(i)] 
There are sets $E_{1},...,E_{M}$ such that 
\begin{equation}
H_{\infty}^{n}(f([0,1]^{n}\backslash\bigcup E_{i}))\lec_{\newn} \kappa
\label{e:inverse-theorem-1}
\end{equation}
and for each $E_i$, $f|_{E_i}$ is $(l, 1)$-bi-Lipschitz with $l\sim\kappa$.
\item [(ii)] For any $\delta$ satisfying $0<\delta <\cH_{\infty}^{n}(f[0,1])^{n})$, 
there is an $\eta>0$ and 
a  set $E\subseteq [0,1]^{n}$ 
with $|E|\gec \eta$ upon which $f$ is $(l,1)$-bilipschitz with $l\sim\delta$. 
The constant $\eta\sim\frac\delta{M(\delta,\newn)}$ where $M(\delta,\newn)$ is from (i) above.
\item [(iii)] The sets $E_i$ of part (i) may be chosen such that if $E_{i}\neq\emptyset$, there is $F_{i}:\bR^{n}\rightarrow\bR^\newn$ which is $L$-bi-Lipschitz, $L\sim_{D}\frac{1}\kappa$, so that 
\[F_{i}|_{E_i}= f|_{E_{i}}.\]
\end{itemize}

\end{thmc}

Part (i) is a restatement of Theorem \ref{t:lip-bilip}, and part (ii) is a corollary of Part (i). 
Parts (i) and (ii) are in fact the main result of \cite{Jones-lip-bilip}.
Part (iii) of Theorem~{{II}}, however, is a new development even in the case $D=n$. 

\begin{remark}
David and Semmes \cite{DS00-regular-mappings} have results analogous to 
Theorem {{I}}  (for the case  $\cM=\bR^D$ only),  
which do not enjoy  a globally defined $g$.  In that paper, they were concerned with determining when a Lipschitz map $f:\bR^{n+m}\rightarrow\bR^{n}$ has a large subset of its domain upon which $(f(x),P_{V}(x))$ is bi-Lipschitz, where $V$ is some $m$-dimensional subspace and $P_{V}$ is its orthogonal projection. They investigate what happens in some special cases, i.e when some additional mild conditions are assumed on the function. 
\end{remark}

\begin{remark}
A key point in \cite{DS00-regular-mappings} is that a Lipschitz function $f:\bR^n\to\bR^D$ is, loosely speaking, usually affine.
In \cite{DS00-regular-mappings} (and going back to  \cite{DS}) this follows from a result by Dorronsoro \cite{Dorronsoro-potential-spaces}. In the setting of Theorem~{{I}}, we are concerned with Lipschitz functions that have a metric space target.  The notion of ``affine approximation" needs to be revisited as does the Carleson type estimate.
See Section \ref{ss:bt} below.
\end{remark}

\begin{remark}
Theorems~{{I}} and~{{II}} assume that the domain of $f$ is all of $\bR^{n+m}$ (for the former) or all of $\bR^n$ (for the latter).
This assumption is not really necessary as  the arguments are local in nature.  Furthermore, if one is only given a Lipschitz function $f$ with domain, say, $[0,1]^n$, then one may extend it so that it is constant on rays emanating from $(\frac12,..., \frac12)$  outside of  $[0,1]^n$ without increasing the Lipschitz constant.
\end{remark}

To prove Theorem~{{II}}, we will use Theorem~{{III}}, which we state below, coupled with a stopping-time construction. 
We say that a function $f$ from $E\subset \bR^{n}$ to $\bR^D$ is \textit{$(\rho,M,\kappa)$-Reifenberg flat} if the following hold. 
For every dyadic cube $Q$ intersecting $E$, there is an affine map $A_{Q}$ such that, denoting by $\sigma(A_Q)$ the $n$-th singular value of $A_Q=A'_Q+ A_Q(0)$, and by $|A'_Q|$ the operator norm of the linear transformation $A'_Q$,
\[
|f(x)-A_{Q}(x)|<\rho\diam Q,\;\;\; x\in 3Q\cap E, \;\;\; \]
\[
{\JARSbluJARS     \sigma(A_{Q})>\kappa, \;\;\; |A_{Q}'|\leq M,}
\]
and if $Q$ is a child of or is adjacent to $R$, 
\[
|A_{Q}'-A_{R}'|<\rho.
\]
See Section \ref{s:rf-functions} for a discussion of Reifenberg flat functions and the origin of this name.
{\JARSbluJARS     
\begin{thmd}
There is $C=C(D)>0$ such that the following holds. For all $M,\kappa>0$ there is a $\rho>0$ such that if $E\subseteq \bR^{n}$ is closed and $f:E\rightarrow\bR^{\newn}$ is a $(\rho,M,\kappa)$-Reifenberg flat function from a subset $E\subset\bR^{n}$ to $\bR^{D}$, then $f$ admits an $(\frac{\kappa}{C},CM)$-bi-Lipschitz extension to a function $f:\bR^{D}\rightarrow\bR^{D}$.
\end{thmd}}
For the statement above to make sense it is important that we think of $\bR^n$ as a subset of $\bR^D$.


\subsection{Other bi-Lipschitz extension Theorems}

Here we make a few comments about prior work related to Theorems~{{II}} and~{{III}}. 

A typical bi-Lipschitz extension theorem says that an $L$-bi-Lipschitz function $f:E\rightarrow \bR^{n}$, where $E\subseteq \bR^{m}$, $m\leq n$, may be extended to a $C(L)$-bi-Lipschitz function $f:\bR^{m}\rightarrow\bR^{m}$, where $C(L)$ depends only on $L$ and $E$, and not on $f$. The existence of a bi-Lipschitz extension theorem typically depends on the geometry of the initial domain with respect to the super-domain one wishes to extend to. It was shown independently by Tukia and Jerison and Kenig, for example, that bi-Lipschitz functions of the real line may be extended to bi-Lipschitz homeomorphisms of $\bR^{2}$ \cite{Tukia80,Jerison-Kenig-A°}. This was subsequently generalized by Macmanus to hold for arbitrary compact subsets of the circle \cite{MacManus-bilip-extensions}. It is, however, possible to map the two dimensional sphere into the Fox-Artin wild sphere in $\bR^{3}$ in a bi-Lipschitz manner, and such a mapping does not permit a homeomorphic extension to all of $\bR^{3}$ \cite{Fox-Artin-wild-sphere}. 

Further restrictions on the class of functions intended to be extended may eliminate such topological obstacles. {\JARSbluJARS     In \cite{TV84-extensions-of-maps-with-QC-extensions}, Tukia and V\"ais\"al\"a  show that any bi-Lipschitz function $f:X\subseteq \bR^{n} \to \bR^n$ (where $X$ is closed and $n\neq 4$) permits a bi-Lipschitz extension so long as it permits a quasisymmetric extension.} 
In \cite{Vaisala-extension-properties, TV84} the authors explored geometric conditions that would guarantee a set $E\subseteq\bR^{n}$ had the so-called bi-Lipschitz extension property (BLEP): There is $L_{0}>1$ and a homeomorphism $L_{1}:[1,\infty)\rightarrow [1,\infty)$ such that for whenever $f:E\rightarrow \bR^{n}$ is $L$-bi-Lipschitz and $L\in [1,L_{0})$, then there is a $L_{1}(L)$-bi-Lipschitz extension of $f$ to all of $\bR^{n}$. Loosely speaking, if $f$ is sufficiently close to being an isometry on $E$, and $E$ has the BLEP, then $f$ can be bi-Lipschitzly extended.\\

If one is not concerned with raising the dimension of the target space of the bi-Lipschitz function, then obtaining an extension becomes significantly easier:

\begin{theorem}[\cite{DS}, Proposition 17.1]
If $K\subset\bR^n$ is compact, and $f:K\to\bR^D$ is $L$-bi-Lipschitz, then 
$f$ has an extension to a $C(L,n,D)$-bi-Lipschitz map $\bR^n\to\bR^{\max\{D,2n+1\}}$.
\label{t:DS-extension}
\end{theorem}

Theorem~{{II}}, on the other hand, says that at the expense of sacrificing a large (fixed) portion  of the initial domain, one may obtain a bi-Lipschitz extension without needing to raise the dimension of the target space. Its proof uses the extension Theorem~{{III}}, which says that a function may be extended from any arbitrary compact set assuming that it is approximately affine on all cubes intersecting that set. This function-analytic description is summed up in the phrase \textit{Reifenberg flat function}, which we borrow from the world of Reifenberg flat sets.  David and Toro \cite {David-Toro-reifenberg-with-holes, David-Toro-snowballs, Toro95} have results similar in spirit about 
parametrizing Reifenberg flat sets (with holes), which also requires some extension results.  

There is also substantial work on the problem of extension in the class $C^k$ and in Sobolev spaces.  A partial  list of references  is \cite{Jones-QC-extensions, Jones-extension-theorems-for-bmo, HK92, Luke-Rogers-Thesis, F06-Cm-Whitney, F05-sharp-whitney, FK09-I, FK09-II}


\subsection{Another corollary.  BP(BPLG)}\label{s:bpbplg}
We point out that  Theorem~{{II}}
gives another corollary.
This is probably only of interest to a smaller set of people, namely those interested in uniformly rectifiable sets (c.f. \cite{DS,DS-analysis-of-and-on}).

Let $\Sigma$ be an $n$-Ahlfors-David regular set lying in $\bR^{D}$, meaning 
\[\cH^{n}(B(x,r)\cap \Sigma)\sim r^{n},\;\;\; x\in \Sigma, \;\;\; 0<r<\diam(\Sigma).\]
For a collection $\cF$ of subsets in $\bR^{D}$, we say $\Sigma$ contains big pieces of $\cF$ (denoted BP($\cF$)) if there is $\ve>0$ such that for all $x\in \Sigma$, $0<r<\diam(\Sigma)$ there is $A\in \cF$ such that 
\begin{equation}
\cH^{n}(\Sigma\cap B(x,r)\cap A)>\ve r^{n}.
\label{e:big-piece}
\end{equation}
Standard examples are surfaces $\Sigma$ that have big pieces of $L$-bi-Lipschitz images of subsets of $\bR^{n}$ (BPBI), isometric copies of $L$-Lipschitz graphs (BPLG), or surfaces that have big pieces of big pieces of $L$-Lipschitz graphs (BP(BPLG)).

\begin{corollary}
Suppose the set $\Sigma\subset \bR^D$ is $n$-Ahlfors-David-regular and has BPBI.
Then $\Sigma$ has BP(BPLG).
\label{c:bpbplg}
\end{corollary}
\begin{remark}
This was known if the codimension of $\Sigma$ was large enough \cite{DS}. 
Indeed, if $A$ is a bi-Lipschitz image satisfying \eqn{big-piece}, then there is a bi-Lipschitz map $f:E\subseteq\bR^{n}\rightarrow \Sigma$, which we may extend using Theorem \ref{t:DS-extension} to a map $f:\bR^{n}\rightarrow \bR^{\max\{D,2n+1\}}$
The corollary then follows in this case, since bi-Lipschitz images of $\bR^{n}$ contain BPLG (see \cite{David-wavelets}, p. 62).
Steve Hofmann had posed to us some time ago the question of whether or not the large codimension is needed.  The corollary above says it is not needed.

\end{remark}

\begin{proof}[Proof of Corollary \ref{c:bpbplg}]
Let $A=f(E)$ satisfy \eqn{big-piece} for some $x\in\Sigma$ and $r>0$ where $f:E\rightarrow A$ is bi-Lipschitz. Without loss of generality, we may assume $r=1$, and that $E\subseteq B(0,c)$, where $c$ is a constant depending on $\ve$ and $L$. Extend $f$ to a Lipschitz function $f:\bR^{n}\rightarrow\bR^{D}$, and by picking $\kappa$ small enough, we may find $E_{j}\subseteq B(0,c)$ such that $\cH^{n}(f(E_{j})\cap B(x,1))\gec_{\ve,\kappa}1$ and $f|_{E_{j}}$ has a bi-Lipschitz extension $F:\bR^{n}\rightarrow\bR^{D}$. As mentioned in the previous remark, $F(\bR^{n})$ has BPLG, and hence $\Sigma$ has BP(BPLG).
\end{proof}

\begin{remark}
It is true that BPBI$\not\Rightarrow$BPLG.  Indeed, an example based on `venetian blinds' was given by T. Hrycak.
\end{remark}

Much more is true about these classes of sets.  See \cite{DS-analysis-of-and-on}.

\subsection{Organization of paper}
In Section \ref{s:non-examples} we give some examples of Lipschitz functions where 
Theorem~{{I}} holds in a vacuous manner.  
In Section \ref{s:notation-prelim} we give some notation and discuss some preliminaries.
In Section \ref{s:rf-functions} we introduce Reifenberg flat functions, discuss their basic properties,  and prove Theorem~{{III}}, a bi-Lipschitz extension theorem for Reifenberg-flat functions.
In Section \ref{s:pf-of-thm-3},
we use  a stopping time construction on top of Theorem~{{III}} to prove 
Theorem~{{II}}. 
Finally, we use Theorem~{{II}} to prove Theorem~{{I}} 
in Section \ref{s:pf-of-thm-1}, which means, in particular,  that we study functions from Euclidean space into a metric space .
Late in Section \ref{s:pf-of-thm-1} we also verify 
Corollary \ref{c:lower-bd-content}.

\subsection{Acknowledgements}
The authors would like to thank Kevin Wildrick and Enrico Le Donne for useful discussions. The authors would also like to thank John Garnett, Terence Tao, and Peter Jones for comments on preliminary versions of this paper. 
Jonas Azzam would also like to thank Jacob Bedrossian and Robert Shukrallah for their helpful discussions. The authors would like to also thank the anonymous referee for his helpful comments and corrections.
Raanan Schul was supported by
a fellowship from the Alfred P. Sloan Foundation and by  NSF grants DMS 0965766 and DMS 1100008.

\section{Some non-examples}\label{s:non-examples}
In this section we give some examples of Lipschitz functions where  Theorem~{{I}}
holds in a vacuous manner i.e. the parameter $\delta=0$.  We feel this is important in order to understand the work we have done.

\subsection{Kaufman's example}\label{example:Kaufman}

Here we note a  basic example to demonstrate the need for a quantity such as \eqn{f-content-intro}.  In \cite{Kaufman-example}, Kaufman constructed a function $f\in C^{1}(\bR^{3})$
from $\bR^3$ \textit{onto} the unit square $[0,1]^2$ such that at every point, the rank of the derivative of this function is either $0$ or $1$. For such an  $f$, a set $E$ as in Theorem~{{I}} must have null measure. To see this, suppose $|E|>0$. Then there is $g:\bR^{3}\rightarrow\bR^{3}$ bi-Lipschitz so that $f\circ g^{-1}$ is bi-Lipschitz on $g(E)_{t}=(\bR^{2}\times \{t\})\cap g(E)$ for each $t\in \bR$. By Fubini's theorem, we may pick a $t$ so that the Jacobian of $f\circ g^{-1}$ is zero almost everywhere on $g(E)_{t}$. Extend $f\circ g^{-1}$ to a Lipschitz function $F$ on all of $\bR^{2}\times\{t\}$ and $J_{F}=J_{f\circ g^{-1}}$ a.e. on $g(E)_{t}$, but since $F$ is bi-Lipschitz on $E$, $J_{F}>0$ a.e. on $g(E)_{t}$, a contradiction.

Note, however, that for this function $f$, however, one gets   from Lemma \ref{l:content-vs-jacobian} below that 
\[f([0,1]^3)=[0,1]^2,\mbox{ and yet } \cH_{\infty}^{2,1}(f, [0,1]^3)=0\,.\]

More smoothness of $f$, however, would have prevented this from happening.  Recall Sard's Theorem.
\begin{theorem}[Sard's Theorem]
Suppose $f:\bR^{n+m}\to\bR^n$ is $C^k$ 
with $k\geq m+1$.
Then $$f\{x: rank(Df_x)<n\}$$ has $n$-dimensional measure 0.
\end{theorem}

We will not describe the Kaufman example here (even though the paper \cite{Kaufman-example} is only 2 pages long), but instead we will give a general scheme below for generating similar ``rank one" Lipschitz maps with large images, and will produce a simple variant of Kaufman's example as a consequence.  Kaufman's original example follows the same general idea, but is done with more care so that the resulting map is of class $C^1$.

\subsection{A general scheme for getting Lipschitz maps}

The following scheme can be used to get many metric spaces as Lipschitz images of $\bR^n$ for some $n$.  
Examples where this scheme applies include any compact metric space which is  a doubling and quasi-convex metric space (see e.g. \cite{Heinonen-lipschitz-analysis} for definitions).
We had learned this scheme through the private communication  \cite{G-via-dennis-and-semmes}.  
The main part, subsection \ref{ss:ext-maps-tree},  also  appeared   in  
the survey \cite{Heinonen-lipschitz-analysis}, 
and in \cite{LS97, Naor-Sheffield, HT-Peano-cubes},
for various cases. 

\subsubsection{Extending real maps}
If 
$f:A_1\to \bR$ 
is a 1-Lipschitz map between a metric space $A_1$ and the real line, and we wish to extend it to a map 
$\hat{f}:A_1\cup\{x\}\to \bR$,
then a standard idea is to consider $f_a(x)=f(a)+ \dist(x,a)$, and define
$\hat{f}(x)=\inf_{a\in A_1} f_a(x)$.
The map $\hat{f}$ is then 1-Lipschitz.

A similar idea works for trees as targets, as sketched below.

\subsubsection{Extending maps into trees}\label{ss:ext-maps-tree}
A similar idea can be used to extend maps that have the images in \textit{compact trees}. This is well known, but we sketch the extension here.
By a compact tree, we mean the metric space one gets by taking an abstract (rooted) tree and thinking of each edge as a segment of some length,  taking the path metric, and then its closure, and finally,  restricting to the case where we get a compact space.
For example, one could take a length of a segment associated to an edge, to be exponentially decaying with the number of vertices one needs to cross to get to the root.  
Denote  such a compact tree by $T$.
Suppose
$f:A_1\to T$ 
is a 1-Lipschitz map between the metric space $A_1$ and the compact tree $T$,
and that we wish to extend it to a map 
$\hat{f}:A_2\to T$,
where $A_2\supset A_1$.
Here we assume that $A_2$ is separable, and so we may extend from $A_1$ to $A_2$ one element at a time, and then to the rest of $A_2$ continuously.  Suppose we want to extend the domain of $f$ to include $x$.
Choose
$$\hat{f}(x)\in \bigcap\limits_{a\in A_1} \overline{\ball}(f(a),\dist(a,x))\,.$$

One needs to check that the intersection is non-empty, and then the resulting map is clearly  1-Lipschitz.
By the triangle inequality, 
$$\overline{\ball}(f(a_i),\dist(x,a_i))\cap  \overline{\ball}(f(a_j),\dist(x,a_j))\neq \emptyset$$ 
for any $i,j$.
The balls above are convex sets in the sense that for any two points in them,
the (unique) geodesic connecting them, is inside the ball.  {\JARSbluJARS     We'll prove that any finite intersection of balls in 
\[\cC=\{\overline{\ball}(f(a),\dist(a,x))\}_{a\in A_{1}}\] 
are nonempty by induction on the number, and then the result will follow by compactness. Suppose $B_{1},B_{2},B_{3}\in \cC$ have empty intersection. Since they pairwise intersect, there are points $a_{ij}\in B_{i}\cap B_{j}$ for $i\neq j\in \{1,2.3\}$. Let $\gamma_{j}\subseteq B_{j}$ be the geodesic connecting $a_{ij}$ and $a_{ik}$. Then we can combine these paths into a loop, and since they are contained in a tree, there is $j$, say $j=1$, so that $\gamma_{1}\subseteq \gamma_{2}\cup \gamma_{3}$. If $\gamma_{1}$ is contained in $\gamma_{2}$, we're done, as the endpoint of $\gamma_{1}$ that is in $\gamma_{3}$ is now also in $\gamma_{2}$. We're similarly done if $\gamma_{1}\subseteq \gamma_{3}$. If neither of these cases occur, then there is an extremal point $t$ for which $\gamma_{1}(t)\in \gamma_{2}$, but this point must also be in $\gamma_{3}$.}

{\JARSbluJARS     
For the induction step, suppose we have convex sets $B_{1},...,B_{n}$ that pairwise intersect. Let $B_{i}'=B_{i}\cap B_{n}$. Then $B_{i}'\cap B_{j}'=B_{i}\cap B_{j}\cap B_{n}$, which convex is nonempty by the previous discussion, hence by the induction hypothesis, 
\[\emptyset \neq \bigcap_{j=1}^{n-1} B_{j}'=\bigcap_{j=1}^{n}B_{j} .\]
}


\subsubsection{Mapping $\bR^n$ to a tree, with the set of leaves contained in the image}
Suppose one has a compact tree with $t+1$ branches emanating from every vertex, except the root, which has $t$ branches. Call it $T$.
Suppose further, that the branches of the tree which are $k$ generations away from the root have size 
$c^{k}$ for some $c<1$.
Consider a Cantor set $C$ in $\bR^n$ obtained as $C=\cap C_k$, 
where $C_k$ has $t^k$ components which are  grouped into $t^{k-1}$ collections, 
such that within each collection the components are at least $c^{k}$ apart, see Figure \ref{f:Garnett-example}.\\

\begin{figure}[t]
\begin{center}
		\scalebox{0.3}{\includegraphics{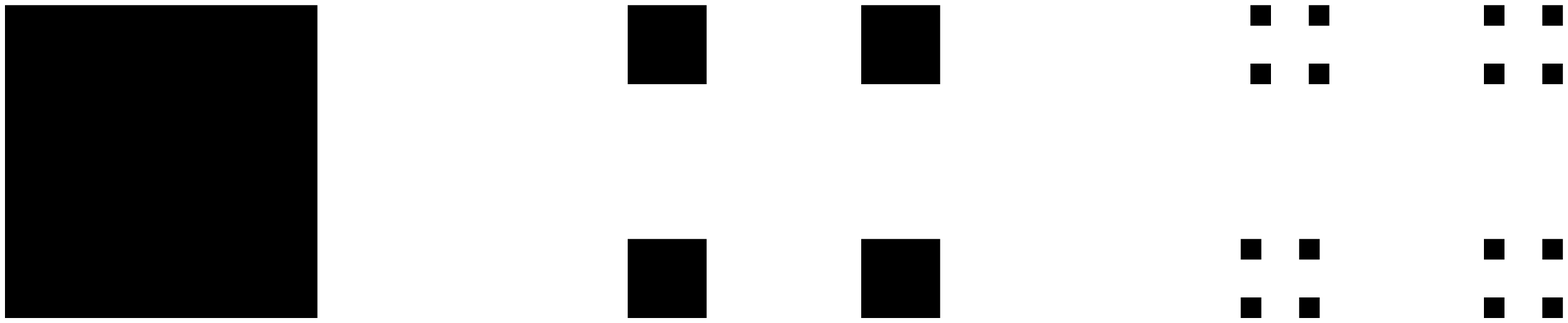}}
		\end{center}
\caption{From left to right: $C_0$, $C_1$, and $C_2$ with $c=\frac14$ and $t=4$}
\label{f:Garnett-example}
\end{figure}
\begin{figure}[hbpt]
\begin{center}
		\scalebox{0.8}{\includegraphics{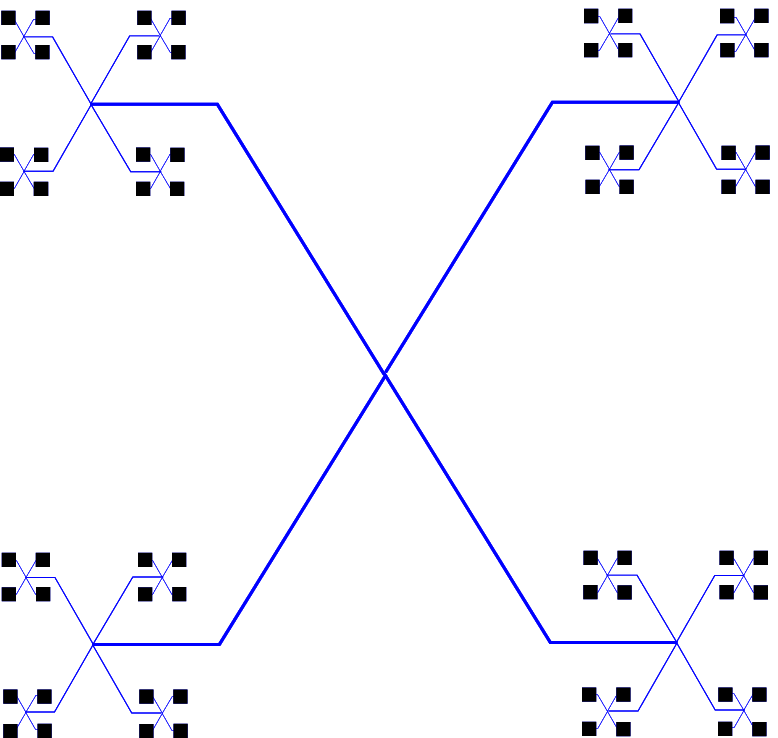}}
		\end{center}
\caption{A Cantor set in $\bR^{2}$, and a quasiconvex tree in $\bR^{2}$ with it as its leaves.
}
\label{f:cantor-tree}
\end{figure}

The set of leaves of $T$, with the metric inherited from the path metric in $T$ is bi-Lipschitz equivalent to  $C$ (with the natural inclusion map).
The Cantor set $C$ can be embedded in $\bR^n$ if $n=n(t,c)$ is sufficiently large.
The bi-Lipschitz map from $C$ into $T$ can be extended to a Lipschitz map from $\bR^n$ to $T$ as described above.
An illustration of a case where this can easily be visualized is given in Figure \ref{f:cantor-tree}.

\subsubsection{Nice metric spaces are images of trees}
So far we have generated maps from $\bR^n$ to geometric  realizations of trees.
We conclude this subsection by mapping such trees onto quite general metric spaces.

{\JARSbluJARS     Recall that a space $X$ is {\it (metrically) doubling} if there is $N>0$ such that any ball in $X$ may be covered by at most $N$ balls of half the radius. A space is {\it path connected} if any two points may be connected by a curve of finite length. This induces a new metric on $X$ called the {\it path metric} $\rho$, where $\rho(x,y)$ is the infimum of the length over all paths connecting $x$ and $y$ in $X$.}

{\JARSbluJARS     Suppose one is given a bounded metric space $X$ that is doubling with respect to its path metric.}
Consider a sequence of $2^{-n}$ nets, $X_n$  for $X$.
For each element $x\in X_n$, let $y\in X_{n-1}$ be a closest element, 
and let  $\gamma_{x,y}$ be a path connecting $x$ and $y$.  
Call such an $y$ the \textit{parent} of $x$, and $x$ the \textit{child} of $y$.
If we further assume that $X$ is quasi-convex, then 
we may take $\gamma_{x,y}$ with length $\sim \dist(x,y)\lesssim 2^{-n}$.
By the doubling assumption, to each $y\in X_{n-1}$ there is a uniformly bounded number of children $t$ to any parent.  
We may now map a tree of valency $t+1$ onto $X$ by a Lipschitz map.

\subsection{Kaufman's example revisited}

We can now give a simple construction of a Lipschitz function that exemplifies the ``rank one" property of the Kaufman example. Let $G\subseteq\bR^{2}$ denote the 4-corner cantor set with $c=\frac{1}{4}$ and $t=4$ as shown in Figure \ref{f:Garnett-example}. It is not difficult to show that its projection along the line making an angle of $\frac{\pi}{6}$ with the $x$-axis is a closed interval. Let $G'$ be $G$ scaled and rotated so that its projection in the first coordinate is $[0,1]$. Let $T$ be a quasiconvex tree in $\bR^{4}$ with its leaves equalling the set $G'\times G'$ (since $G'\times G'$ is also a Cantor set with $c=\frac{1}{4}$ and $t=8$), see Figure \ref{f:cantor-tree}. 
Then the natural inclusion map $G'\times G'\rightarrow T$ is Lipschitz and has a surjective Lipschitz extension $g:\bR^{4}\rightarrow T$. Let $f$ be $g$ composed with the projection into the first and third coordinates. Then $f:\bR^{4}\rightarrow [0,1]^{2}$ is a surjective Lipschitz map whose derivative has rank one almost everywhere.

We conclude by pointing out that the manner in which a  Lipschitz function $f$ was constructed in this section yields, for $n>1$ and any $m\geq 0$,
$$H^{n,m}_\infty(f,[0,1]^{n+m})=0\,.$$

\section{Notation and preliminaries}\label{s:notation-prelim}

\subsection{Basic notation}
A general metric space will be denoted by $\cM$. The corresponding distance will be written as $\dist(\cdot,\cdot)$. We will sometimes abuse notation and replace $\dist(x,y)$ by $|x-y|$. {\JARSbluJARS     For any metric space, $B(x,r)=\{y\in X:|x-y|\leq r\}$.}

Let $\Delta=\Delta(\bR^{n})$ denote the set of dyadic cubes in $\bR^{n}$, i.e.
the collection of half open cubes of the form
\begin{gather*}
Q=[{i_1\over 2^j},{i_1 +1\over 2^j})\times...\times[{i_d\over 2^j},{i_d +1\over 2^j})
\end{gather*}
where  $i_1,...,i_d,j$ are integers.
Endow $\Delta$ with the standard tree/family structure given by calling a dyadic cube $Q$ a parent of a cube $R$ if and only if $R$ has half the sidelength of $Q$ and $R\subset Q$.

{\JARSbluJARS     If $E\subseteq \bR^{n}$, let 
\begin{equation}
\cQ_{E}=\cQ_{E}(\bR^{n})=\{Q\in \Delta(\bR^{n}):Q\cap E\neq\emptyset\}.
\label{e:cQ_{E}}
\end{equation}
We will frequently just write $\cQ_{E}$ if it is clear which space we are dealing with.}
For $Q,R\in\Delta$ and $N>0$ an integer, 
\begin{itemize}
\item $x_{Q}$ denotes the center of $Q$.
\item if $Q$ and $R$ are cubes that are either adjacent to each other  {\JARSbluJARS     (meaning their boundaries intersect)}
and are of the same size
or one is a child of the other, we write 
\begin{equation}\label{e:sim-def}
Q\sim R\,.  
\end{equation}
\item The $N$th-ancestor of $Q$ is denoted by $Q^{N}\in\Delta $. In particular, the parent of a cube $Q$ is denoted $Q^{1}$. 

\item If $Q,R\in \Delta$, $Q_{R}$ denotes the smallest parent of $Q$ containing $R$. 

\item For $\lambda>0$, let 
\[\lambda Q = \{ \lambda(x-x_{Q})+x_{Q}: x\in Q,\]
that is, $\lambda Q$ be the half open cube with center $x_{Q}$, sides parallel to those of $Q$, and diameter $\lambda\diam Q$. To clarify the order of operations, we note that $\lambda Q^{N}$ denotes the cube with the same center as $Q^{N}$ but $\lambda$-times the diameter of $Q^{N}$.\\
\item Let
\begin{equation}
B^{Q}=B(x_{Q},\frac{\diam Q}{2}), \;\;\; B_{Q}=B(x_{Q},\frac{\diam Q}{2\sqrt{n}}),
\label{e:B^Q}
\end{equation}
that is, $B^{Q}$ is the smallest ball containing $Q$ and $B_{Q}$ is the largest ball contained in $Q$.
 \end{itemize}
Suppose $\cM=\bR^D$.
\begin{itemize}
\item If $Q=\prod_{i=1}^{n}[a_{i},a_{i}+\ell(Q)]$, where $\ell(Q)=\frac{\diam Q}{\sqrt{n}}$ is the sidelength of $Q$, let $a_{Q}=a=(a_{1},...a_{n})$ and define $A_{Q}^{f}$ to be the affine map  taking $a+\ell(Q)e_{i}$ to $f(a+\ell(Q)e_{i})$ for each $i$ and $a$ to $f(a)$. 

\item Let $\sigma_{m}(Q)$ denote $\sigma_{m}(A_{Q}^{f})$, i.e. the $m$th largest singular value of the linear part of $A_{Q}^{f}$. If  $f:\bR^{n}\rightarrow\bR^{D}$, we will simply write $\sigma(Q)=\sigma_{n}(Q)$. 

\item For a general affine map $A$, we will write $A'$ to denote the linear part of $A$, so that $A(x)=A'(x)+A(0)$.

\item For a linear transformation $A$ we will write $|A|$ for the operator norm of $A$. 
\end{itemize}

{\JARSbluJARS    
\begin{remark}
\label{r:smaller-lip-constant}
Note that if $f$ is $L$-Lipschitz, then $|(A_{Q}^{f})'|\leq \sqrt{n}L$ for all $Q$. Indeed, by translating and scaling $f$ we may assume that $Q=[0,1]^{n}$ and $f(a_{Q})=a_{Q}=0$, so $(A_{Q}^{f})'=A_{Q}^{f}$. Then, by definition of $A_{Q}^{f}$,
\[|A_{Q}^{f}|\leq \sqrt{\sum_{i=1}^{n} ||A_{Q}^{f}(e_{i})||_{2}^{2}}\leq \sqrt{\sum_{i=1}^{n}L^{2}}=\sqrt{n}L.\]
In the course of the proofs of the main theorems, we will assume that our Lipschitz functions are scaled so that the $A_{Q}^{f}$ have norm at most $1$. This is not necessary, but it will simplify the exposition.
\end{remark}

}


\subsection{$\bt$-numbers}
\label{ss:bt}

For a Lipschitz function $f:[0,1]^{n}\rightarrow \cM$, define
\begin{equation*}
\d_{1}^{f}(x,y,z):=|f(x)-f(y)| + |f(y)-f(z)| -|f(x)-f(z)|.
\end{equation*}
For an interval $I=[a,b]\subset\bR^{n}$, let 
\begin{equation*}
\bt_{f}(I)^2\diam(I)=\diam(I)^{-3}\int_{x=a}^{x=b}\int_{y=x}^{y=b}\int_{z=y}^{z=b} \d_{1}^{f}(x,y,z)dzdydx\,.
\end{equation*}
For a cube $Q\in \bR^n$, define
the quantity $\bt^{(n)}_{f}(Q)$ by
\begin{eqnarray*}
&&\bt^{(n)}_f(Q)^2\side(Q)^{n-1}=\\
&&\phantom{xxx}
\int_{g\in G_n}\int_{x\in \bR^n \circleddash g\bR} \chi_{\{|(x+g\bR)\cap 7Q|\geq \side(Q)\}}\bt((x+g\bR)\cap 7Q)^2dxd\mu(g)
\end{eqnarray*}
where  
$\bR$ is identified with $\{\bR,0,...,0\}\subset \bR^n$,  $G_n$ is the group of all rotations of $\bR$ in $\bR^n$ equipped with the its Haar measure $d\mu$, and $dx$ is the $n-1$ dimensional Lebesgue measure on $\bR^n \circleddash g\bR$, the orthogonal complement of $g\bR$ in $\bR^n$.
This type of quantity is connected to Menger curvature.  See \cite{Schul-survey} for more details.
Note that  any $n\geq 1$, we have that  $\bt^{(n)}$ is scale invariant.  We will usually omit the superscript $^{(n)}$  when the dimension of the cube/interval is clear.

\begin{remark}\label{r:bt-means-flat}
The quantity $\bt(Q)$ measures how close the images of segments in $Q$ under $f$ are to being contained in geodesics. 
It can also be thought of as an analogue to the norms of the Haar-wavelet coefficients of $\nabla f$. Much like their counterpart, the $\bt$-numbers have an $L^2$-type condition that gives us control on the ``straightness" of $f$ on all scales in the form of the following theorem from \cite{Schul-lip-bilip}:
 \end{remark}

\begin{theorem}\label{t:TST}
For an $L$-Lipschitz function $f:[0,1]^{n}\rightarrow\cM$ and $N$ a fixed integer,
\begin{equation*}
\sum_{Q\in \Delta, Q\subseteq [0,1]^n} \bt_{f}(3Q^{N})^2|Q|  \lec_{N, n} L\,.
\end{equation*}
\label{t:bt-sum}
\end{theorem}

\subsection{Using $\bt$ and $\sigma$}

\begin{figure}[hbpt]
\begin{center}
\includegraphics[width=110pt]{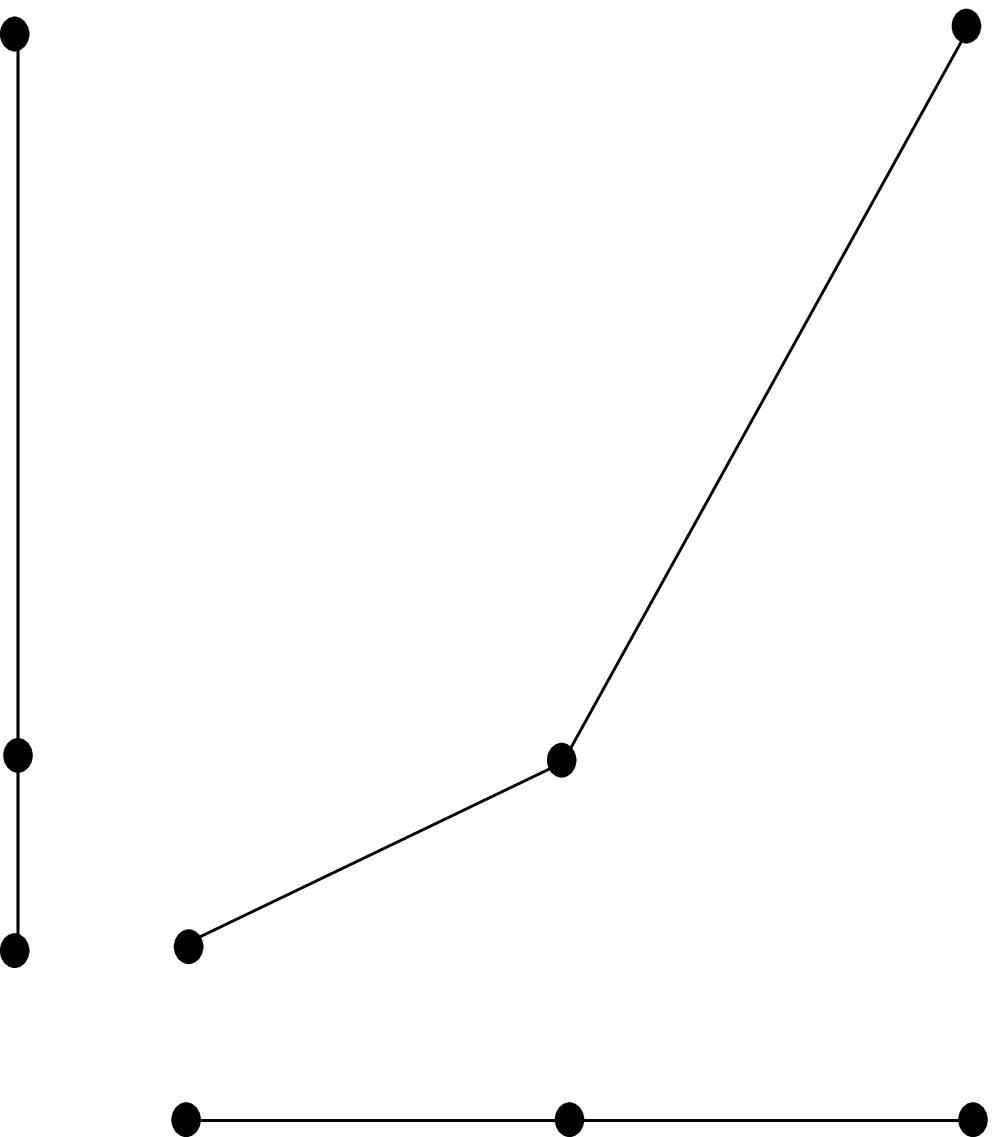}
\end{center}
\begin{picture}(0,0)
\put(142,18){$x$}
\put(185,18){$y$}
\put(230,18){$z$}
\put(100,42){$f(x)$}
\put(100,66){$f(y)$}
\put(180,100){$\tilde{f}$}
\put(100,145){$f(z)$}%
\end{picture}
\caption{Using $\tilde{f}$ to see non-linearity}
\label{f:fig-3333}
\end{figure}

Assume for a moment that the target space is $\cM=\bR^{D}$. As mentioned before, if $\bt_{f}(Q)$ is small, then this tells us $f$ is roughly straight on $Q$, mapping straight lines to approximately straight lines. We may also use it to establish when $f$ is approximately \textit{affine} on a cube. The definition of $\bt_{f}$ alone does not give us this information (in particular, any map from $\bR^{n}$ into a straight line, affine or not, has $\bt_{f}(Q)=0$ for all cubes $Q$). We remedy this by  using the graph $\tilde{f}:x\mapsto (f(x),x)$ in place of $f$. 

For $x,y,z$ colinear, it may be the case that $f(x)$,$f(y)$, and $f(z)$ are colinear as well, in which case $\d_{1}^{f} (x,y,z)=0$, but if $f$ is not linear, then the graph points  $\tilde{f}(x)$,$\tilde{f}(y)$, and $\tilde{f}(z)$ have a ``bend" in the middle, which would imply $\d_{1}^{\tilde{f}}(x,y,z)$ is large 
(see Figure \ref{f:fig-3333}). 
{\JARSbluJARS     As being affine is a stronger property than having lines mapped close to lines, we might expect that $\bt_{\tilde{f}}$ dominates $\bt_{f}$, as is the case:}

\begin{lemma}
Let $f:\bR^n\to\cM$ be a 
Lipschitz function.  Then
$\bt_{\tilde{f}}\geq \frac{1}{\sqrt{2}}\bt_{f}$.
\end{lemma}

%

\begin{proof}
First we note that for nonnegative numbers $a,b,c,A,B,C$,
\begin{multline}
\sqrt{a^{2}+A^{2}}+\sqrt{b^{2}+B^{2}}-\sqrt{c^{2}+C^{2}}
\geq \frac{ac+AC}{\sqrt{c^{2}+C^{2}}}+\frac{bc+BC}{\sqrt{c^{2}+C^{2}}}-\frac{c^{2}+C^{2}}{\sqrt{c^{2}+C^{2}}} \\
=\frac{(a+b-c)c+(A+B-C)C}{\sqrt{c^{2}+C^{2}}}.
\label{e:abcABC}
\end{multline}
This just uses the facts that each term on the far left side of the inequality is the norm of some two-dimensional vector, which is at least the inner product of that vector with any unit vector, and we pick that unit vector to be $\frac{(c,C)}{|(c,C)|}$. If $a+b-c\geq 0$ and $C\leq c$, then we get
\[\sqrt{a^{2}+A^{2}}+\sqrt{b^{2}+B^{2}}-\sqrt{c^{2}+C^{2}}\geq \frac{1}{\sqrt{2}}(A+B-C).\]

Let $x,y,z$ be colinear. By letting 
\[a=|x-y|,b=|y-z|,c=|x-z|\]
\[A=|f(x)-f(y)|,B=|f(y)-f(z)|,C=|f(x)-f(z)|,\]
we obtain
\[\d_{1}^{\tilde{f}}(x,y,z)\geq \frac{1}{\sqrt{2}}\d_{1}^{f}(x,y,z)\]
which implies the lemma.
\end{proof}

\begin{lemma}
Suppose $f:\bR^n\to \bR^D$ is Lipschitz and $\bt_{\tilde{f}}(Q)=0$. Then $f$ is affine on $Q$. 
\label{l:affine-lemma-1}
\end{lemma}

\begin{proof} 

Without loss of generality, let $a_{Q}=f(a_{Q})=0$ and $Q=[0,1]^{n}$. 
We will now show that  $f$ is linear. 

First, we prove the result for $n=1$. Suppose $\d_{1}^{\tilde{f}}(x,y,z)=0$ where $x<y<z$. If $\bt_{\tilde{f}}=0$, then equality holds in \eqn{abcABC}, which in turn happens if and only if $(a,A),(b,B)$, an $(c,C)$ are parallel. With the choice of values as above, this shows

\begin{equation*}
\tilde{f}(z)=\frac{|z-x|}{|y-x|}(\tilde{f}(y)-\tilde{f}(x))+\tilde{f}(x).
\end{equation*}
By the definition of $\tilde{f}$, we must also have
\begin{equation*}
f(z)=\frac{|z-x|}{|y-x|}(f(y)-f(x))+f(x).
\end{equation*}
Letting $x=0$ and $z=1$ gives $f(z)=|z|\frac{f(y)}{|y|}$ and for all $y\in(0,1)$, thus the ratio $\frac{f(y)}{|y|}$ is constant and equal to some $\lambda\in\bR$ and hence $f$ is linear on $[0,1]$.

Suppose now that we have proven the statement for all dimensions less than $n$. If we restrict $f$ to any line segment in $Q$, $f$ is affine along this line. In particular, there are constants $\lambda_{i}$ such that for each $t\in[0,1]$,
\begin{equation*}
f(te_{i})=\lambda_{i} t v_{i}=A^f_Q(te_{i}),
\end{equation*}
where the last equality follows from the definition of $A^f_{Q}$. 

Let $P$ be the $n-1$ dimensional affine plane passing through the points  $\{e_{1},...,e_{n}\}$. {\JARSbluJARS     By the definition of $\bt^{\tilde{f}}$ and the fact that $f$ is Lipschitz, $\bt^{\tilde{f}|_{P}}$ is also zero, and so by the induction hypothesis,  $f$ is  affine on P.}
Furthermore,  $f$ agrees with $A^f_Q$ on $P$, as it agrees on the vertices. {\JARSbluJARS     By checking each segment $L\cap Q$ intersecting both $0$ and $P$ {and recalling that the $n=1$  case has been veryfied}}, we see that $f$ must agree with $A^f_Q$ on these lines as well, and thus $f=A^f_Q$ on all of $Q$.

\end{proof}

\begin{lemma}
Let $f$ be as above. For all $\rho>0$ there is $\ve_{\beta}>0$ such that if $\bt_{\tilde{f}}(Q)<\ve_{\beta}$, then 
\begin{equation}
|A_{Q}^{f}(x)-f(x)|<\rho\diam Q\;\;\; x\in Q.
\label{e:A_Q^f-estimate}
\end{equation}
Moreover, we may pick $\ve_\beta>0$ small, depending only on $\rho$ and $\tau$ so that whenever $R\subseteq Q$ are such that $\diam R\geq \tau\diam Q$, then
\begin{equation}
|(A_{R}^{f})'-(A_{Q}^{f})'|< \rho.
\label{e:A_Q^f'-estimate}
\end{equation}
\label{l:affine-lemma-2}
\end{lemma}

\begin{proof}

Suppose for each $k\in\bN$, there is a $1$-Lipschitz function $f:[0,1]^{n}\rightarrow\bR^{D}$ so that $\bt_{f_{k}}(Q)<\frac{1}{k}$, but
\[|A_{Q}^{f_{k}}(x_{k})-f_{k}(x_{k})|\geq \rho\diam Q\]
for some $x_{k}\in Q$. By rescaling, we may assume $Q=[0,1]^{n}$. By a normal families argument, we obtain a $1$-Lipschitz function $f$ and $x\in Q$ so that
\[\bt_{\tilde{f}}(Q)=0,\;\;\; |A_{Q}^{f}(x)-f(x)|\geq \rho\diam Q.\]
As $\bt_{\tilde{f}}(Q)=0$, by \ref{l:affine-lemma-1}, $f$ is linear, but then we must have $A_{Q}^{f}=f$, a contradiction. \\

A similar argument shows eq. \eqn{A_Q^f'-estimate}. 

\end{proof}

\begin{remark}
Theorem \ref{t:TST} combined with  Lemma \ref{l:affine-lemma-2}  imply
together that the set  $\cB$ of cubes where $f$ does not satisfy \eqn{A_Q^f-estimate} satisfy a Carleson estimate,
\begin{equation}\label{april26}
\sum_{Q\in \cB, Q\subseteq R}|Q|\lec_{\rho}|R|
\end{equation}
for any dyadic cube $R$, which indicates that $f$ is close to being affine on most dyadic cubes. This isn't the only way to arrive at this property: results such as \cite{DS,DS00-regular-mappings} use a stronger result in potential theory due to Dorronsoro \cite{Dorronsoro-potential-spaces} to arrive at \eqref{april26} as a corollary, although the proof we supply above serves as a much simpler and more geometric proof of this Carleson estimate {\JARSbluJARS     (modulo Theorem \ref{t:TST})}.
\label{r:carleson}
\end{remark}

Define
\begin{equation*}
\beta_{f}^{(n-1)}(Q)=\frac{\inf\limits_{P}\ \sup\{\dist(f(x),P):x\in Q\}}{\diam Q}
\end{equation*}
where the infimum is over all $(n-1)$-planes in $\bR^{D}$.

\begin{lemma}
If $f:\bR^{n}\rightarrow\bR^{D}$ is $1$-Lipschitz with $\bt_{\tilde{f}}(Q)<\ve$ and $\ve$ is small enough (depending on $\sigma(Q)$, then
\begin{equation}
\sigma(Q)\sim_{n,D} \beta_{f}^{(n-1)}(Q).
\label{e:bt-sigma}
\end{equation} 
\end{lemma}
\begin{proof}
 Indeed, one inequality is easy, since $\sigma(Q)$ is the width of the parallelepiped spanned by the vectors
\[(A_{Q}^{f})'(e_{i})=\frac{f(\ell(Q)e_{i}+a_{Q})-f(a_{Q})}{\ell(Q)},\;\;\; i=1,...,n\]
but the width of the smallest parallelpiped containing $f(\ell(Q)e_{i}+a_{Q})-f(a_{Q})$ is no more than $\beta_{f}^{(n)}(Q)$. Note that this required no assumption on $\bt_{\tilde{f}}$. 

{\JARSbluJARS     
To show the reverse inequality, let  $\rho<\sigma(Q)$ and pick $\ve>0$ small enough so that the conclusion of Lemma \ref{l:affine-lemma-2} holds. If $\bt_{\tilde{f}}$ is small enough, and $P$ is the image under $(A_{Q}^{f})'$ of the orthogonal compliment of the space spanned by the singular vector corresponding to $\sigma(Q)$, then by Lemma \ref{l:affine-lemma-2},
\begin{multline*}
\beta_{f}^{(n-1)}(Q)\leq \frac{\sup\{\dist(f(x),P+f(a_{Q})):x\in Q\}}{\diam Q}\\
\leq \rho+\frac{\sup\{\dist(A_{Q}^{f}(x),P+f(a_{Q})):x\in Q\}}{\diam Q}\\
\lec \rho+\sigma(Q)\lec \sigma(Q).
\end{multline*}}


\end{proof}

\subsection{Whitney cubes and simplexes}\label{s:whtney}

\begin{definition}
Let $E\subset \bR^n$ be a closed set.
A Whitney decomposition $\cW$ for the open set $E^{c}$ is a collection of dyadic cubes with disjoint interiors such that
\begin{enumerate}
\item $\bigcup_{Q\in \cW}Q=E^{c}$,
\item if $Q\in \cW$, then $3Q\subseteq E^{c}$ and $3Q^{1}\cap E\neq\emptyset$,
\item {\JARSbluJARS     $\diam Q \lec_{n} \dist(Q,E)\lec \diam Q$.}
\item If the boundaries of $Q,R\in\cW$ touch, then $\diam Q\leq 4\diam R$.
\end{enumerate}
{\JARSbluJARS     This collection can easily be constructed by taking $\cW$ to be the maximal collection of cubes so that $3Q\cap E\neq\emptyset$. For more details, see \cite{little-stein}.
}
\label{d:whitney-def}
\end{definition}

We will now recall the construction of Whitney simplexes, which are used in such sources as \cite{Vaisala-extension-properties, TV84} to construct bi-Lipschitz and quasisymmetric extensions. We refer the reader to \cite{Hatcher} for definitions of simplexes and complexes. 

Recall the definition of the join operation: For disjoint sets $A$ and $B$ in Euclidean space, we define
\[A*B=\bigcup\{[x,y]:x\in A,\ \ y\in B\}.\]

Note that the partition $\cW$ defines an $n$-complex. For $k=1,...,n$, let $\cW_{k}$ be the set of $k$-cells in $\cW$. If $R$ is a $k$-dimensional cube in $\cW_{k}$, we 
write $x_{R}$ for its center. Define
\[\cK_{1}=\cW_{1},\]
\[\cK_{k+1}=\{S*x_{R}:S\in \cK_{k},\;\;\; S\subseteq R\in \cW_{k}\}.\]

\begin{figure}[hbpt]
\begin{picture}(200,200)(-50,0)
\scalebox{.7}{\includegraphics{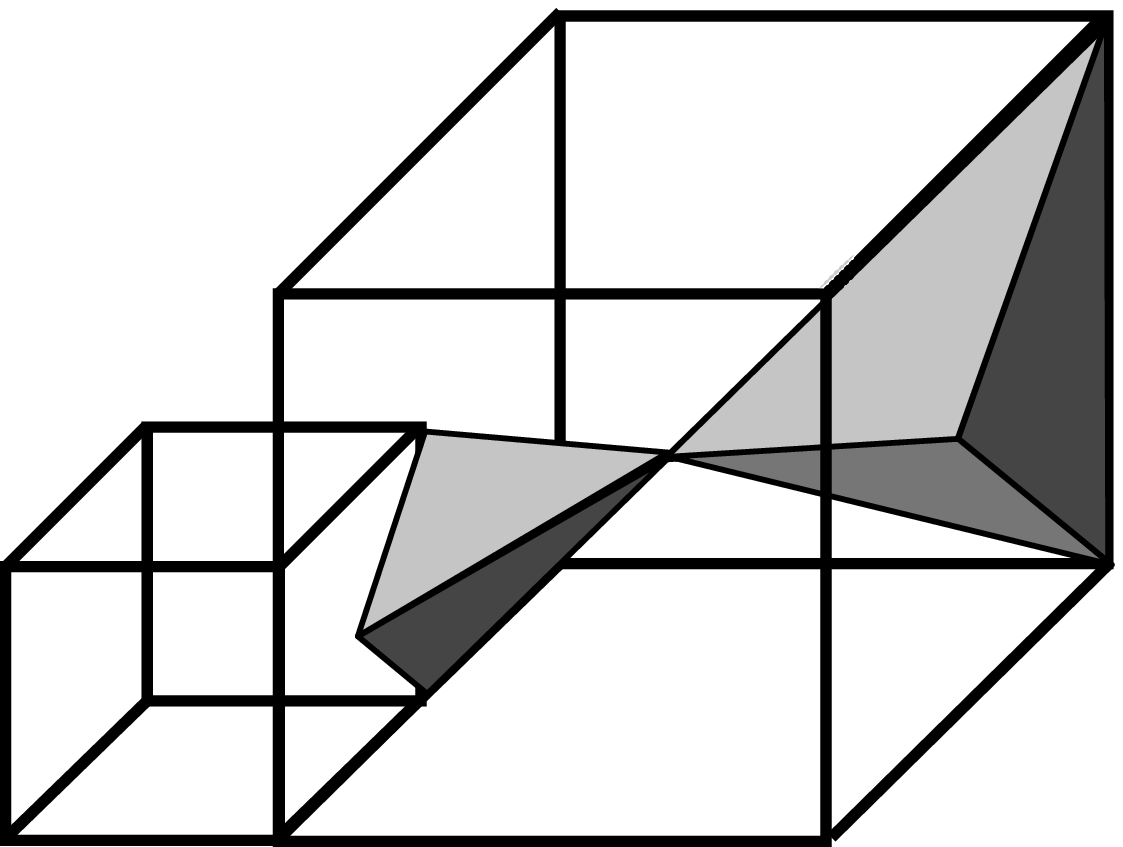}}
\put(0,175){$x_{1}$}
\put(0,45){$x_{0}$}
\put(-45,90){$x_{2}$}
\put(-100,85){$x_{3}$}
\end{picture}
\caption{Examples of simplexes induced by Whitney cubes with $[x_{0},x_{1}]\in \cK_{1}$, $[x_{0},x_{1}]*x_{2}\in \cK_{2}$, and $(x_{0}*x_{1}*x_{2})*x_{3}\in \cK_{3}$.}
\end{figure}

Let $\cS=\cS(E)=\cK_{n}$ denote the collection of $n$-simplexes we obtain in this way.  For $S\in \cS$, let $\Corner(S)$ denote the set of corners {\JARSbluJARS     (i.e. vertices)} of $S$. Define
\[\Corner(\cS):=\bigcup_{S\in \cS}\Corner(S).\] 

\begin{remark}
Here we list some important geometric properties of the family $\cS$ that will be needed later:
\begin{enumerate}[(a)]

\item Every $Q\in\cW$ is a finite union of simplexes in $\cS$, and in particular, $E^{c}=\bigcup_{S\in \cS}S$.

\item 
For all $S\in\cS$, 
\begin{equation}
 |S|\sim_{n} (\diam S)^{n}\sim_{n} |Q|.
\label{e:S-volume}
\end{equation}
Finally, $S\supset B(s,r)$ for some $s\in S$ and $r\sim_n\diam(S)$.
%

\item If $x\in S\in \cS$ and $S\subseteq Q\in \cW$, then
\begin{equation}
\diam S\sim_{n}\diam Q \sim d(Q,E) \sim_{n} \dist(S,E)\sim_{n}\dist(x,E).
\label{e:S-dist}
\end{equation}
\end{enumerate}
\end{remark}

{\JARSbluJARS     The first item follows by construction. It follows from the construction that any such simplex has positive volume, and by Definition \ref{d:whitney-def} (4), there are only a finite number of possible congruent partitions of a Whitney cube into simplexes using the construction above, and this implies the second item in the remark. The final item follows from the second and Definition \ref{d:whitney-def} (3). See also \cite[Section 5]{Vaisala-extension-properties}.}

\section{Reifenberg flat functions - Theorem~{{III}}}\label{s:rf-functions}
This section is concerned with functions $f$ with a Euclidean target space,
$f:\bR^n\to \bR^D$.
\subsection{Introduction}

Reifenberg flatness is defined below in Definition \ref{d:reifenberg-flat}, but loosely speaking, it means that the function $f$ is very close to being affine on each dyadic cube intersecting $E$.
We use this terminology to suggest that this property is a function-analytic analogue of the usual definition of Reifenberg flatness. Recall that a set $\Sigma\subseteq \bR^{\newn}$ is Reifenberg flat if there is $\ve>0$ and $r_{0}>0$ such that for all $x\in \Sigma$ and $r<r_{0}$, there is a hyperplane plane $P_{x,r}$ such that 
\[d_{H}\bigg(B(x,r)\cap \Sigma\ \ , \ \ B(x,r)\cap P_{x,r}\bigg)<\ve r\]
where $d_{H}$ denotes Hausdorff distance. 
See \cite{Reifenberg}.  
One can trace the way we think of these sets to 
\cite{Semmes-chord-arc-I, Semmes-chord-arc-II, Toro95,David-Toro-reifenberg-with-holes}.
It is not hard to show, for example, that if $f:\bR^{n}\rightarrow\bR^{n+1}$ is Reifenberg flat with respect to $E=\bR^{n}$ (see Definition \ref{d:reifenberg-flat} below), then its image $\Sigma = f(\bR^{n})$ is Reifenberg flat.

\begin{definition}
Let $E\subseteq \bR^{n}$ and let $\hitsE=\cQ_{E}$ be the collection of dyadic cubes that intersect $E$. For $\rho,M,\ve_{\sigma}>0$, we say a function $f:E\rightarrow\bR^{\newn}$ is  $(\rho,M,\ve_{\sigma})$-Reifenberg flat if for every dyadic cube $Q\in\hitsE$, there is an affine mapping $A_{Q}$ such that 
\begin{equation}
|f(x)-A_{Q}(x)|<\rho\diam Q,\;\;\; x\in 3Q\cap E,
\label{e:f-close-to-A_Q-on-E}
\end{equation}
\begin{equation}
\sigma(A_{Q})>\ve_{\sigma}, \;\;\; |A_{Q}'|\leq M
\label{e:sigma(A_Q)>delta}
\end{equation}
and if $Q\stackrel{\eqref{e:sim-def}}{\sim} R$,
\begin{equation}
|A_{Q}'-A_{R}'|<\rho
\label{e:close-derivatives}
\end{equation}
{\JARSbluJARS     
If $M=1$, we will write that $f:E\rightarrow \bR^{D}$ is $(\rho,\ve_{\sigma})$-Reifenberg flat instead of $(\rho,1,\ve_{\sigma})$. 
}
\label{d:reifenberg-flat}

\end{definition}

\begin{remark}
We record here some simple estimates concerning the $A_{Q}$. 

\begin{enumerate}[(a)]
\item Note that the conditions of the definition imply that for $Q,R\in \hitsE$, $Q\sim R$,
for all $x\in\bR^{n}$,
\begin{equation}
|A_{Q}(x)-A_{R}(x)|\lec \rho(|x-x_{Q}|+\diam Q),
\label{e:close-tangents}
\end{equation}
c.f. Lemma 5.13, \cite{DS}.

Indeed, note that $3R\cap 3Q \cap E\neq\emptyset$, so let $x'\in 3R\cap 3Q\cap E$. Then
\begin{multline*}
|A_{Q}(x)-A_{R}(x)|\\
\leq |A_{Q}(x')-A_{R}(x')|+\av{\int_{0}^{1}(A_{Q}'-A_{R}')(x'+t(x-x'))dt}\\
 \leq 2\rho\diam Q+\rho|x-x'|\lec \rho(|x-x_{Q}|+\diam Q).
\end{multline*}

\item  We can also obtain estimates relating distant cubes as follows: For dyadic cubes $Q$ and $R$, let $\cubeDist(Q,R)$ equal the length of the shortest sequence of cubes $Q=Q_{1},...,Q_{\cubeDist(Q,R)}=R$ such that for $1\leq j<\cubeDist(Q,R)$, 
$Q_{j}\sim Q_{j+1}$ (c.f. \cite{Jones-extension-theorems-for-bmo}). Then eq. \eqn{close-derivatives} and eq. \eqn{close-tangents} imply

\begin{equation}
|A_{Q}'-A_{R}'|\lec \cubeDist(Q,R)\rho
\label{e:A_Q'-A_R'}
\end{equation}
and 
\begin{equation}
|A_{Q}(x)-A_{R}(x)|\lec \cubeDist(Q,R)\rho(|x-x_{Q}|+\diam Q_{R}),\;\;\; x\in \bR^{n}
\label{e:angle-QR}
\end{equation}
where $Q_{R}$ is the smallest parent of $Q$ so that $R\subseteq 3Q_{R}$. 

\item In some situations, we can do better than the above estimate. If $Q\in \cQ_{E}$ and $K\in\bN$, then for $x\in 3Q$, \eqn{close-tangents} implies
\begin{equation}
|A_{Q}(x)-A_{Q^{K}}(x)|\leq \sum_{i=1}^{K}|A_{Q^{i-1}}(x)-A_{Q^{i}}(x)|\lec \rho\sum_{i=1}^{K}\diam Q^{i}\lec \rho\diam Q^{K}.
\label{e:geom-series}
\end{equation}

%
\end{enumerate}

\end{remark}

{\JARSbluJARS     
\begin{remark}
It is not hard to show that, for $\rho<\frac{\ve_{\sigma}}{12\sqrt{n}}$, if $f:E\rightarrow\bR^{\newn}$ $(\rho,M,\ve_{\sigma})$-Reifenberg flat, then $f$ is in fact $(\frac{ \ve_\sigma}{2},M+\frac{\ve_{\sigma}}{2})$-bi-Lipschitz on $E$. Indeed, if $x,y\in E$ are distinct points, let $Q$ be the smallest dyadic cube containing $x$ such that $y\in 3Q$. Then $\diam Q\leq3\sqrt{n} |x-y|$, which implies
\begin{multline*}
|f(x)-f(y)|
 \geq |A_{Q}(x)-A_{Q}(y)|-|f(x)-A_{Q}(x)|-|f(y)-A_{Q}(y)| \\
 \stackrel{ \eqn{f-close-to-A_Q-on-E}, \eqn{sigma(A_Q)>delta}}{\geq} \ve_{\sigma}|x-y|-2\rho\diam Q
 \geq \frac{\ve_{\sigma}}{2}|x-y|
\end{multline*}
and
\begin{multline*}
|f(x)-f(y)|  \leq |A_{Q}(x)-A_{Q}(y)|+|f(x)-A_{Q}(x)|+|f(y)-A_{Q}(y)| \\
  \stackrel{ \eqn{f-close-to-A_Q-on-E}, \eqn{sigma(A_Q)>delta}}{\leq}
 M|x-y|+2\rho\diam Q\leq (M+\rho6\sqrt{n})|x-y|<(M+\frac{\ve_{\sigma}}{2})|x-y|.
 \end{multline*}
\label{r:bilip-on-E}
\end{remark}
}

{\JARSbluJARS     
\begin{remark}

If $f:E\rightarrow \bR^{D}$ is $(\rho,M,\ve_{\sigma})$-Reifenberg flat with respect to a collection of affine maps $\{A_{Q}\}_{Q\in \cQ_{E}}$, then $\frac{1}{M}f:E\rightarrow \bR^{D}$ is $(\frac{\rho}{M},1,\frac{\ve_{\sigma}}{M})$-Reifenberg flat with respect to the maps $\{\frac{1}{M}A_{Q}\}_{Q\in \cQ_{E}}$. Hence, to prove Theorem~{{III}}, it suffices to prove the following proposition (recall the notation at the end of Definition \ref{d:reifenberg-flat}).

\label{r:A_Q-and-f-lip}
\end{remark}
}

{\JARSbluJARS     
\begin{proposition}
There is $C(D)>0$ such that the following holds. For all $\ve_{\sigma}>0$ there is a $\rho>0$ such that if $E\subseteq \bR^{n}$ is closed and $f:E\rightarrow\bR^{\newn}$ is a $(\rho,\ve_{\sigma})$-Reifenberg flat function from a subset $E\subset\bR^{n}$ to $\bR^{D}$, then $f$ admits an $(\frac{\ve_{\sigma}}{C},C)$-bi-Lipschitz extension to a function $f:\bR^{D}\rightarrow\bR^{D}$.
\label{p:reifenberg-flat}
\end{proposition}

Here, we consider $E\subseteq\bR^{n}$ as also lying in $\bR^{D}$ via the natural embedding $\bR^{n}\rightarrow \bR^{n}\times(0,...,0)$. \\

}

{\JARSbluJARS    
Before moving on to the proof, we first mention a few technical lemmas.
The first of these lemmas says that we can alter a Reifenberg flat function to be affine on a  collection of isolated cubes and the extended function will remain Reifenberg flat.

\begin{lemma}
Let $E\subseteq \bR^{n}$ (possibly empty), and let $\{Q_{j}\}\subseteq \Delta$ be a collection of dyadic cubes such that 
the $\{3Q_{j}^{2}\}$ have disjoint interiors. 
Let $\island=\{\{x\}:x\in E\}\cup\{Q_{j}\}$ and let 
\[\cQ'=\{Q \in \Delta:\exists R\in \island,R\subseteq Q\}.\] 
Let 
$$E'=(E\backslash \bigcup 3Q_{j}^{2})\cup  \bigcup3Q_{j}  \,.$$
Suppose $f:E\rightarrow\bR^{D}$is a Lipschitz function such that there are affine functions $\{A_{Q}:Q\in\cQ'\}$ satisfying
the conditions of Definition \ref{d:reifenberg-flat}
for some $\rho,\ve_{\sigma}>0$. Define a new function on  $E'$ as follows.  For $x\in E'$,
\[F(x)=\isif{f(x) & x\in  E \backslash \bigcup 3Q_{j}^2 \\ A_{Q_j}(x) & x\in 3Q_{j}.}\]
Then $F$ is $(C\rho,\ve_{\sigma})$-Reifenberg flat for some $C=C(n,D)>0$
\label{l:first-replacement-lemma}
\end{lemma}

\begin{remark}
In the proof of this lemma we will make use of \eqn{A_Q'-A_R'} and \eqn{angle-QR}, but some caution should be taken because our function $f$ is not necessarily Reifenberg flat at all positions and scales. We only have a collection of $\{A_{Q}:Q\in \cQ'\}$ that satisfy the Reifenberg flat properties However, the family $\cQ'$ is coherent in the sense that if $Q\in \cQ'$, then every ancestor of that cube is in $\cQ'$. Hence, it is not difficult to show that, if we define for $Q,R\in \cQ'$
\[\cubeDist^{\cQ'}(Q,R):=\inf\{N:\exists \{Q_{j}\}_{j=1}^{N}\subseteq \cQ' \mbox{ s.t. } Q=Q_{1}\sim \cdots \sim Q_{N}=R\},\]
then $\cubeDist^{\cQ'}$ agrees with $\cubeDist$ on $\cQ'\times\cQ'$, and hence \eqn{A_Q'-A_R'} and \eqn{angle-QR} still hold for this function $f$ and cubes in $\cQ'$.
\end{remark}

\begin{proof}[Proof of Lemma \ref{l:first-replacement-lemma}]

Let 
\[\cQ''=\cQ'\cup\bigcup_{j} \Delta(3Q_{j}).\]
For $Q\in \cQ''\backslash \cQ'$, $Q\in \Delta(3Q_{j})$ for some $j$ and we define $A_{Q}=A_{Q_{j}}$. Note that the conditions of Definition \ref{d:reifenberg-flat} still hold for $f:E'\backslash \bigcup 3Q_{j}^{2}\rightarrow \bR^{D}$ with the collection $\{A_{Q}:Q\in \cQ''\}$. 

For each $R\in \cQ_{E'}$, we will assign a cube $Q(R)\in \cQ''$ and define maps $T_{R}=A_{Q(R)}$. After doing this, we'll verify that the maps $\{T_{R}:R\in \cQ_{E'}\}$ satisfy the conditions of Definition \ref{d:reifenberg-flat} for $f:E'\rightarrow\bR^{D}$.

Let $R\in \cQ_{E'}$. 

\begin{itemize}
\item If $R\subseteq 3Q_{j}^{2}$ for some $j$ (which is unique if it exists by the separation property of the $\{Q_{j}\}$) then $R$ must intersect $3Q_{j}$, and we pick $Q(R)$ to be a maximal cube in $\Delta(3R)\cap \Delta(3Q_{j})$. This gives that $A_{Q(R)}=A_{Q_j}$.
\item If $R$ is not contained in such a $3Q_{j}^{2}$, then pick $Q(R)$ to be any maximal cube in $\Delta(3R)\cap \cQ'$. This is necessarily nonempty since, if $R\cap (E\backslash \bigcup 3Q_{j}^{2})=\emptyset$, then $R\cap 3Q_{j}\neq\emptyset$ for some $j$, and if $\side R\leq \side Q_{j}$, then $R\subseteq 3Q_{j}^{2}$, a contradiction. Hence $\side R>\side Q_{j}$ and so $3R\supseteq Q_{j}$, and $\Delta(3R)\cap \cQ'\neq\emptyset$ as a result.
\end{itemize}
Note that $\cubeDist(R, Q(R))\lec 1$ for all $R\in\cQ_{E'}$. Thus if $R\sim R'$ are in $\cQ''$, then $\dist(Q(R),Q(R'))\lec 1$ and  hence
\[|T_{R}'-T_{R'}'|=|A_{Q(R)}-A_{Q(R')}|\stackrel{\eqn{A_Q'-A_R'}}{\lec} \cubeDist(Q(R),Q(R'))\rho,\]
and so \eqn{close-derivatives} holds for $\{T_{R}:R\in \cQ_{E'}\}$. Moreover, \eqn{sigma(A_Q)>delta} holds trivially as $\{T_{R}:R\in \cQ_{E'}\}\subseteq \{A_{Q}:Q\in \cQ'\}$, so it remains to verify \eqn{f-close-to-A_Q-on-E} for $f:E'\rightarrow \bR^{D}$. Let $x\in E'\cap R$ for some $R\in \cQ_{E'}$.

\begin{itemize}
\item If $x\in E\backslash \bigcup 3Q_{j}^{2}$, then $Q(R)\in \cQ'$, and if $Q'$ is the maximal cube in 
{$\Delta(3R)$}
containing $x$, 
then $\cubeDist(Q(R),Q')\lec 1$ and
\begin{multline*}
|F(x)-T_{R}(x)|
=|f(x)-A_{Q(R)}(x)|
\leq |f(x)-A_{Q'}(x)|+|A_{Q'}(x)-A_{Q(R)}(x)|\\
\stackrel{\eqn{f-close-to-A_Q-on-E}, \eqn{angle-QR}}{\lec} \rho\diam Q'+d(Q(R),Q')\rho(|x-x_{Q'}|+\diam 3R) \lec \rho\diam R.
\end{multline*}

\item If $x\in 3Q_{j}$ for some $j$, then $R\cap 3Q_{j}\neq\emptyset$ and $F(x)=A_{Q_{j}}(x)$. 
\begin{itemize}
\item If $\side R\leq \side Q_{j}$, then $R\subseteq 3Q_{j}^{2}$ and 
$Q(R)\subseteq 3Q_{j}$, 
implying $T_{R}=A_{Q(R)}=A_{Q_{j}}$, in which case \eqn{f-close-to-A_Q-on-E} holds trivially. 

\item If $\side R>\side Q_{j}$, then $3R\supseteq Q_{j}$, and if $Q'$ is the largest parent of $Q_{j}$ in $3R$, then $\dist(Q',Q(R))\lec 1$ and
\begin{multline*}
|F(x)-T_{R}(x)| 
=|A_{Q_{j}}(x)-A_{Q(R)}(x)| \\
\leq |A_{Q_{j}}(x)-A_{Q'}(x)|+|A_{Q'}(x)-A_{Q(R)}(x)| \\
\stackrel{\eqn{geom-series},\eqn{angle-QR}}{\lec} \rho\diam Q'+\cubeDist(Q',Q(R))\rho(|x-x_{Q'}|+3R)\\
 \lec \rho \diam R.
\end{multline*}
\end{itemize}
\end{itemize}
\end{proof}

The next lemma says not only can we change a Reifenberg flat function $f$ to be affine on a collection of isolated cubes on a set $E$, we can pick it to be affine for large values.

\begin{lemma}
Let $E\subseteq R_{0}\subseteq \bR^{n}$, where $R_{0}\in \Delta$ and $E$ is possibly empty. Let $\{Q_{j}\}$ be a collection of dyadic cubes in $R_{0}$ such that $Q_{j}^{2}\subseteq R_{0}$, 
and the cubes $\{3Q_{j}^{2}\}$ have disjoint interiors.

Let $\island=\{\{x\}:x\in E\}\cup\{Q_{j}\}$ and let 
\[\cQ'=\{Q  \in \Delta(R_{0}):\exists R\in \island,R\subseteq Q\}.\] 
Let 
$$E'=(3R_{0})^{c}\cup \bigcup3Q_{j} \cup(E\backslash \bigcup 3Q_{j}^{2})\,.$$
Suppose $f:E\rightarrow\bR^{D}$ is a Lipschitz function such that there are affine functions $\{A_{Q}:Q\in\cQ'\}$ satisfying
the conditions of Definition \ref{d:reifenberg-flat}
for some $\rho,\ve_{\sigma}>0$. Define a new function on  $E'$ as follows.  For $x\in E'$,
\[F(x)=\isif{f(x) & x\in  (3R_{0}\cap E) \backslash \bigcup 3Q_{j}^{2} \\ A_{Q_j}(x) & x\in 3Q_{j} \\
A_{R_{0}}(x) & x\in (3R_{0})^{c}.}\]
Then $F$ is $(C\rho,\ve_{\sigma})$-Reifenberg flat for some $C=C(n,D)>0$
\label{l:replacement-lemma}
\end{lemma}

\begin{proof}

Define $\bar{f}:E\cup \overline{(3R_{0})^{c}}$ by 
\[\bar{f}:=\one_{E}f+\one_{(3R_{0})^{c}}A_{R_{0}}.\]
Let 
\[\cQ''=\cQ'\cup\{Q: Q\cap (3R_{0})^{c}\neq\emptyset\}\cup\{R_{0}^{1}\}.\]

 For $Q\in \cQ''\backslash \cQ'$, define $A_{Q}=A_{R_{0}}$. Note that $\cQ''$ has the property that if $Q^{k}\in \cQ''$ for all $Q\in \cQ''$. Clearly, if $Q\cap(3R_{0})^{c}\neq\emptyset$, then the same holds for $Q^k$. Moreover, the parents of $R_{0}^{1}$ also intersect $(3R_{0})^{c}\neq\emptyset$, so they too are in $\cQ''$. Finally, if $Q\in \cQ'$, then any parent contained in $R_{0}$ is in $\cQ'$ and hence in $\cQ''$, and then by the preceding sentence every parent of $Q$ is in $\cQ''$. 

Let 
\[\bar{E}=(3R_{0})^{c}\cup E.\]

We claim that $\bar{f}:\bar{E}\rightarrow \bR^{D}$ satisfies the conditions of Lemma \ref{l:first-replacement-lemma}: 
the cubes $\{Q_{j}\}$ serve as the cubes $\{Q_{j}\}$ of Lemma \ref{l:first-replacement-lemma}; 
$\cQ''$ serve as $\cQ'$; and the maps $\{A_{Q}:Q\in \cQ''\}$  serves  as $\{A_{Q}:Q\in \cQ'\}$.
Thus Lemma \ref{l:replacement-lemma} follows immediately as soon as  we verify that we have \eqn{f-close-to-A_Q-on-E}, \eqn{sigma(A_Q)>delta}, and \eqn{close-derivatives}.  Note that \eqn{sigma(A_Q)>delta} is true by the definition of the $A_{Q}$.

Let, $Q,Q'\in \cQ''$ be such that $Q\sim Q'$. If $Q,Q'\in\{Q: Q\cap (3R_{0})^{c}\neq\emptyset\}\cup\{R_{0}^{1}\}$, and  $Q,Q'\in \cQ'$, then \eqn{close-derivatives} holds trivially (as $A_{Q}=A_{Q'}=A_{R_{0}}$ in the first case, and by the conditions of the lemma for the second case). Suppose now that $Q\in \{Q: Q\cap (3R_{0})^{c}\neq\emptyset\}\cup\{R_{0}^{1}\}$ and $Q'\in\cQ'$. Then $Q$ necessarily must be $R_{0}$, for if $Q$ were properly contained in $R_{0}$, then $Q'$ would have to be contained in $\frac{3}{2} R_{0}$, which is not possible given our choice of $Q'$. Thus, once again, \eqn{close-derivatives} holds trivially. 

Now we verify \eqn{f-close-to-A_Q-on-E} for $\bar{f}$. Let $x\in \bar{E}$ be contained in some cube $Q$. If $x\in (3R_{0})^{c}$, then $A_{Q}=A_{R_{0}}$, and hence
\[|\bar{f}(x)-A_{Q}(x)|=|A_{R_{0}}(x)-A_{R_{0}}(x)|=0.\]
If $x\in E$, then \eqn{f-close-to-A_Q-on-E} holds if $Q\in \cQ'$. Otherwise, $Q$ must contain $R_{0}$, and $A_{Q}=A_{R_{0}}$, thus
\[|\bar{f}(x)-A_{Q}(x)|=|f(x)-A_{R_{0}}(x)|\stackrel{\eqn{f-close-to-A_Q-on-E}}{\leq} \rho \diam R_{0}\leq \rho\diam R_{0}.\]

\end{proof}

}

The rest of this section is dedicated to the proof of  Proposition \ref{p:reifenberg-flat}.

\subsection{Reducing the proof of Proposition \ref{p:reifenberg-flat} to the case $n=D$}

\begin{lemma}\label{l:D-equals-n-suffices}
For any $\rho'>0$, there exists a $\rho>0$ such that the following holds.
Suppose   $E\subseteq \bR^{n}$ is  closed and $f:E\rightarrow\bR^{\newn}$ is a $(\rho,\ve_{\sigma})$-Reifenberg flat function.
Then we have that
$f$ is $(\rho',\ve_{\sigma})$-Reifenberg flat as a function from $E$ considered as a subset of $\bR^D$ to $\bR^D$. 
\end{lemma}

This immediately gives:
\begin{corollary}\label{c:n=D-enough}
Proposition \ref{p:reifenberg-flat} follows from verifying the case  $n=D$.
 \end{corollary}

\begin{proof}
If the $D=n$ case  of Proposition \ref{p:reifenberg-flat} is true, we may conclude (for $\rho'$ small enough) that $f:E\subseteq \bR^{n}$ permits a bi-Lipschitz extension $f:\bR^{D}\rightarrow \bR^{D}$. Restricting $f$ to $\bR^{n}\subseteq \bR^{D}$ gives a bi-Lipschitz extension of $f$ on $\bR^{n}$.
\end{proof}

{\JARSbluJARS     In proving Lemma \ref{l:D-equals-n-suffices}, we use some techniques from or inspired by those in \cite{TV84}. }
We start with some preliminary lemmas.

For $1\leq n\leq D$, let $\cF_{n}^{D}\subseteq \bR^{nD}$ denote the set of orthonormal frames $v=(v_{1},...,v_{n})$. 
The following Lemma will be used  with $\cM=\cF_n^D$.
\begin{lemma}
Let $\cM$ be a metric space.
Suppose $E\subseteq\bR^{n}$ and $\phiphiphi:\cQ_{E}\rightarrow \cM$ satisfies $|\phiphiphi(Q)-\phiphiphi(R)|<q$ whenever $Q,R\in \cQ_{E}$ and $Q\sim R$. Then there is an extension of $\phiphiphi$ to all of $\Delta(\bR^{n})$ that satisfies $|\phiphiphi(Q)-\phiphiphi(R)|\leq C'q$ whenever $Q\sim R$, for some constant $C'>0$ depending only on $n$.

\label{l:frame-extension}
\end{lemma}

{\JARSbluJARS     
\begin{proof}

To define the extension, we will assign to each cube $Q\in \Delta$ a cube $R(Q)\in \cQ_{E}$ and  define $\phi(Q)=\phi(R(Q))$. 
Let
\[\cQ'=\{Q\in \Delta: 3Q\cap E\neq\emptyset\}\supseteq \cQ_{E},\]
and let $\cW=\{Q_{j}\}$ be the Whitney cube decomposition for $E^{c}$. 

\begin{itemize}
\item If $Q\in \cQ_{E}$, set $R(Q)=Q$.
\item If $Q\in \cQ'\backslash \cQ_{E}$, let $R(Q)$ be a maximal cube in $\cQ_{E}\cap \Delta(3Q)$. 
\item If $Q\in \Delta\backslash \cQ'$, then $3Q\cap E=\emptyset$. By definition of the Whitney cubes, $Q\subseteq Q_{j}$ for some $Q_{j}\in \cW$, so $Q_{j}^{1}\in \cQ'$ and we set $R(Q)=R(Q_{j}^{1})$.
\end{itemize}

\noindent \Claim: $\cubeDist(R(Q),R(Q'))\lec 1$ for every pair $Q,Q'\in \Delta$ such that $Q\sim Q'$. Clearly, if the claim is true, then the Lemma will follow. Before proving the claim, we first  note that by construction of the map $Q\mapsto R(Q)$,
\begin{equation}
\phantom{phantom} \mbox{if } Q\in \cQ', \mbox{ then } \cubeDist(Q,R(Q))\lec 1.
\label{e:Q-R(Q)}
\end{equation}

\begin{itemize}
\item If $Q,Q'\in \cQ'$, the claim follows by \eqn{Q-R(Q)} and the triangle inequality.
\item If $Q,Q'\in \Delta\backslash \cQ'$, then $Q$ and $Q'$ are contained in Whitney cubes $Q_{i}$ and $Q_{j}$. Since $Q\sim Q'$, we must have $Q_{i}\sim Q_{j}$ and thus $R(Q)=R(Q_{i}^{1})\sim R(Q_{j}^{1})\sim R(Q')$.
\item If $Q\in \Delta\backslash \cQ'$ and $Q'\in \cQ'$, let $Q_{j}$ be the Whitney cube containing $Q$. Then $3Q'\cap E\neq\emptyset$ implies
\[\diam Q\leq \diam Q_{j} \lec_{n} \dist(Q_{j},E)\leq \dist(Q,E)\leq \diam 3Q'\lec \diam Q,\]
thus $\cubeDist(Q,R(Q))=\cubeDist(Q,Q_{j}^{1})\lec 1$, and so
\begin{multline*}
\cubeDist(R(Q),R(Q'))\leq \cubeDist(R(Q),Q)+\cubeDist(Q,Q')+\cubeDist(Q',R(Q)) \\
\stackrel{\eqn{Q-R(Q)}}{\lec} 1.
\end{multline*}

\end{itemize}

\end{proof}
}

Recall a lemma from \cite{TV84}.
\begin{lemma}
There is a number $q_{D}$ such that for $q\in(0,q_{D})$ there is $r_{q}>0$ with the following property: Let $1\leq n\leq D-1$ and let $\phiphiphi:\Delta(\bR^{n})\rightarrow \cF_{n}^{D}$ be a map such that 
\begin{equation}
\label{e:r_q}
|\phiphiphi(Q)-\phiphiphi(R)|_{\infty}\leq r_{q}
\end{equation}
(where $|\cdot|_{\infty}$ is the $\ell^{\infty}$ norm on $\bR^{nD}$) whenever $Q\sim R$. Then there is a map $\psi:\Delta(\bR^{n})\rightarrow \cF_{D}^{D}$ such that 
\begin{enumerate}
\item $(\psi(Q))_{i}=(\phi(Q))_{i}$ for $i=1,...,n$,
\item $|\psi(Q)-\psi(R)|_{\infty}\leq q$ whenever $Q\sim R$.
\item If, for some cube $Q\in \Delta(\bR^{n})$, we pick $v\in\cF_{D}^{D}$ such that $v_{i}=(\phiphiphi(Q))_{i}$ for all $1\leq i\leq n$, then we may choose $\psi(Q)=v$.
\end{enumerate}
\label{l:TV}
\end{lemma}
For the proof of the above lemma, we refer the reader to \cite{TV84},  
page~161.

{\JARSbluJARS     

 Identify $\bR^{D}$ as $\bR^{n}\oplus \bR^{D-n}$, let 
$E'=E\times \{0\}\subset \bR^D$ 
denote the natural inclusion of $E$ in $\bR^{D}$, and write  vectors in $\bR^{D}$ as $(x,y)$ with $x\in \bR^{n}$ and $y\in \bR^{D-n}$. 

\begin{proof}[Proof of lemma  \ref{l:D-equals-n-suffices}]
Suppose $f:E\rightarrow\bR^{D}$ is $(\rho,\ve_{\sigma})$-Reifenberg flat and define
\[F:E'\rightarrow\bR^{D}, \;\;\; F:(x,0)\mapsto f(x).\]
We will show $F$ is Reifenberg flat.
Let $A_Q$ be as in Definition \ref{d:reifenberg-flat}.
Let 
\[G:U=\{A\in M_{D,n}(\bR): \sigma_{n}(A)>0\} \rightarrow \cF_{n}^{D}\]
be the map that takes a $D\times n$ matrix $A$ to the orthogonal frame spanned by its column vectors generated by the Grahm-Schmidt process, and define $\phi:\cQ_{E}\rightarrow \cF_{n}^{D}$ by
\[\phi(Q)=G(A_{Q}').\]
Note that $G$ is continuously differentiable on the open set $\{A\in M_{D,n}(\bR): \sigma_{n}(A)>0\} $ and in particular, the set
\[K:=\{A\in M_{D,n}(\bR): \sigma_{n}(A) \geq  \ve_{\sigma}, \;\; |A|\leq 1\}\]
is a compact subset of $U$, hence $G$ is $C$-Lipschitz on $K$ with $C=C(n,D,\ve_{\sigma})$ (where $U$ is equipped with the operator norm and its range with the $\ell^{\infty}$ norm). In particular, for $Q\sim R$ we have
\[|\phi(Q)-\phi(R)|_{\infty}=|G(A_{Q}')-G(A_{R}')|_{\infty} \leq C|A_{Q}'-A_{R}'|\stackrel{\eqn{close-derivatives}}{<}C\rho.\]
Without loss of generality we may assume  $\rho'\in (0,q_{D}]$.
Pick $\rho=\frac{r_{\rho'/2}}{CC'}$ (where $r_{q}$ is as in the statement of Lemma \ref{l:TV} and $C'$ is the constant from Lemma \ref{l:frame-extension} with $\cM=\cF_{n}^{D}$). Then by Lemma \ref{l:frame-extension} $\phi$ has an extension to all of $\Delta(\bR^{n})$ satisfying \eqn{r_q}. Lemma \ref{l:TV} now implies the existence of $\psi:\Delta(\bR^{n})\rightarrow \cF_{D}^{D}$ satisfying items (1) and (2) with $q=\frac{\rho'}{2}$, that is
\begin{equation}
\label{e:rho'/2}
 |\psi(Q)-\psi(R)|_{\infty}\leq \frac{\rho'}{2}.
\end{equation}
Let 
\[\psipsipsi(Q):=[(\psi(Q))_{n+1}|\cdots |(\psi(Q))_{D}]\in M_{D,D-n}(\bR).\] 
For each dyadic cube $Q\in \cQ_{E'}(\bR^{D})$, $Q\cap\bR^{n}=P(Q)\in \cQ_{E}(\bR^{n})$, where $P:\bR^{D}\rightarrow\bR^{n}$ is the orthogonal projection onto the first $n$-coordinates. For each such $Q$ set
\begin{equation}
M_{Q}'=\left[A_{P(Q)} | \psipsipsi(P(Q)) \right], \;\;\; M_{Q}=M_{Q}'+[A_{P(Q)}(0) | 0].
\label{e:breve(A)}
\end{equation}
Then for $(x,0)\in 3Q\cap E'$, $x\in P(Q)\cap E$, so $P(Q)\in \cQ_{E}$ and hence
\[|F(x,0)-M_{Q}(x,0)|=|f(x)-A_{Q}(x)|<\rho\diam (Q\cap \bR^{n})\leq\rho\diam Q\]
and for $Q,R\in \cQ_{E}(\bR^{D})$ with $Q\sim R$,
\begin{multline*}
|{M_{Q}}'-{M_{R}}'|
\leq |A_{P(Q)}'-A_{P(R)}'|+|\psipsipsi_{P(Q)}-\psipsipsi_{P(R)}|
\stackrel{\eqn{close-derivatives}}{<} \rho+|\psipsipsi_{P(Q)}-\psipsipsi_{P(R)}|_{\infty}\\
\stackrel{\eqn{rho'/2}}{\leq} \rho+\frac{\rho'}{2}<\rho'
\end{multline*}
if $\rho$ also satisfies $\rho<\frac{\rho'}{2}$.
Moreover, $\sigma(M_{Q})=\sigma(A_{P(Q)})\geq \ve_{\sigma}$. Hence, $F$ is $(\rho',\ve_{\sigma})$-Reifenberg flat.
\end{proof}

}

\begin{remark}
 In the proof above, we made a distinction between $f$ and $F$. The fact  that the lemma holds means that this distinction is of no importance.
 We will omit it in the future. 
 \end{remark}

\begin{remark}
We note here that we get more from these proofs. The function $f$ is Reifenberg flat with respect to the affine transformations
\begin{equation*}
M_{Q}=\left[A_{Q\cap \bR^{n}}| \Psi(Q)\right]
\end{equation*}
where $\psi:\Delta(\bR^{n})\rightarrow \cF_{D}^{D}$ is the function from Lemma \ref{l:TV}.
\label{r:staritA_Q}
\end{remark}

\subsection{The bi-Lipschitz extension}
In this section we define the extension of $f$ to all of $\bR^{\newn}$ and introduce a sequence of lemmas from which we may deduce the bi-Lipschitzness of $f$. We will assume $\rho<\frac{1}{2}$ and choose it to be smaller as need be for each lemma. 

{ 
\bf By Corollary \ref{c:n=D-enough} we will assume $D=n$ in this section.
}
Let $\cS=\cS(E)$ be the decomposition into Whitney simplexes as in Section \ref{s:whtney}.
For each $x\in \Corner(\cS)$, let $Q_{x}\in\hitsE$ be a cube of minimum diameter such that $x\in 3Q$. Define 
\[f(x)=A_{Q_{x}}(x).\]
For each simplex $S$, let $A_{S}$ denote the unique affine map that agrees with $f$ on $\Corner(S)$ and extend $f$ into each such simplex by letting
\[f(x)=A_{S}(x), \;\;\; x\in S.\]

For a simplex $S$, let $c_{S}$ be a point in $E$ closest to $S$. By construction, for all $x\in S$,
\begin{equation}
\dist(x,E)\leq |x-c_{S}|\sim \diam S\stackrel{\eqn{S-dist}}{\sim} \dist(x,E)
\label{e:S-to-E}
\end{equation}

Let $Q_{S}\in \cQ$ denote the smallest cube containing $c_{S}$ such that $S\subseteq 3Q_{S}$. It is not difficult to show, using the properties of the Whitney simplexes,
\begin{equation}
\diam Q_{S}\sim \diam S\sim \dist(S,E)
\label{e:S-QS}
\end{equation}

\begin{lemma} For $S\in\cS$ a simplex,

\begin{equation}
|A_{S}(x)-A_{Q_{S}}(x)|\lesssim \rho\diam Q_{S}, \;\;\; x\in 3Q_S
\label{e:AS-AQ}
\end{equation}
and
\begin{equation}
|A_{S}'-A_{Q_{S}}'|\lec \rho.
\label{e:AS'-AQ'}
\end{equation}
\label{l:f-A_S}
\end{lemma}

\begin{proof}
First, note that if $x$ is a corner of $S$ and $y\in (3Q_{S}\backslash Q_{S})\cap S$ (which is nonempty by the minimality of $Q_{S}$), then
\[
\diam Q_{x} 
 \sim \dist(x,E) 
\stackrel{eq. \eqn{S-dist}}{\sim}  \diam S 
\stackrel{eq. \eqn{S-dist}}{\sim}\dist(y,E)
\sim \diam Q_{S}.
\]
and combining this with the fact that $3Q_{x}\cap Q_{S}\neq\emptyset$ gives
\[\cubeDist(Q_{S},Q_{x})\lec 1.\]

Let $\{x_{0},...,x_{n}\}=\Corner(S)$. Then each $x\in 3Q_{S}$, may be written as a  combination $x=\sum_{j=0}^{n}t_{j}x_{j}$ where $|t_{j}|\lec_{n}1$. We then have
\begin{align*}
|A_{S}(x)-A_{Q_{S}}(x)| 
& \leq \sum t_{j}|A_{S}(x_{j})-A_{Q_{S}}(x)|
=\sum t_{j}|A_{Q_{x_{j}}}(x_{j})-A_{Q_{S}}(x)|\\
& \stackrel{eq. \eqn{angle-QR}
}{\lec} \sum t_{j}\rho(|x-x_{j}|+\diam Q_{S}) \lec_{n} \rho(\diam S+\diam Q_{S}) \\
& \lec\rho\diam Q_{S}
\end{align*}
which establishes \eqn{AS-AQ}.

To prove \eqn{AS'-AQ'}, by translating $f$ we may assume $x_{0}=A_{S}(x_{0})=0$ so that $A_{S}$ is linear. Then for $x\in S$,
\[|A_{S}(x)-A_{Q}(x)|\lec \rho\diam Q\]
and by eq. \eqn{S-volume}, one can show this implies
\[|A_{S}'-A_{Q}'|\lec_{n} \rho.\]
\end{proof}

\begin{lemma}
If $Q\in\hitsE$, then
\begin{equation}
|A_{Q}(x)-f(x)|\lec\rho'\diam Q, \;\;\; x\in 3Q
\label{e:almost-affine-extension}
\end{equation}
where $\rho'=\rho\log\frac{1}{\rho}$.
\label{l:almost-affine-extension}
\end{lemma}

\begin{proof}
This certainly holds for $x\in E\cap 3Q$ by Definition \ref{d:reifenberg-flat}. If $x$ is in some simplex $S$, we divide into two cases:
\begin{itemize}
\item[\bf Case 1:] $\diam Q_{S}\geq \rho\diam Q$. If so, then $\cubeDist(Q_{S},Q)\lec \log\frac{1}{\rho}$, and by \eqn{angle-QR}, \eqn{S-QS}, and the fact that 
\[\diam Q_{S}\sim \dist(S,E)\leq \diam Q,\]
we have
\begin{multline*}
|A_{Q}(x)-f(x)|  \\
= |A_{Q}(x)-A_{S}(x)|
 \leq |A_{Q}(x)-A_{Q_{S}}(x)|+|A_{Q_{S}}(x)-A_{S}(x)| \\
 \lec \rho\cubeDist(Q_{S},Q)(|x-x_{Q_{S}}|+\diam Q_{S})
 \lec \rho\log\frac{1}{\rho}\diam Q=\rho'\diam Q.
\end{multline*}

\item[\bf Case 2:] $\diam Q_{S}< \rho\diam Q$. 
 Since $|A_{Q}'|\leq 1$, we have that $|A_{S}'|\leq 2$  by Lemma \ref{l:f-A_S} (if $\rho<1$), hence

\begin{multline*}
|A_{Q}(x)-A_{S}(x)|
 \leq |A_{Q}(x)-A_{Q}(c_{S})|+|A_{Q}(c_{S})-f(c_{S})|+|f(c_{S})-A_{Q_{S}}(c_{S})|\\
 +|A_{Q_{S}}(c_{S})-A_{S}(c_{S})|+|A_{S}(c_{S})-A_{S}(x)|\\
 \lec |x-c_{S}|+|A_{Q}(c_{S})-f(c_{S})|+\rho\diam Q_{S}+\rho\diam Q_{S}+ |c_{S}-x|\\
\end{multline*}
Now we observe that 
\begin{equation}
|x-c_{S}|\lec \diam Q_{S}<\rho\diam Q
\label{e:x-c_S}
\end{equation}
and hence, for $\rho<\frac{1}{\sqrt{n}}$,
\[c_{S}\in 3\diam Q\cap E.\]
Thus we can use \eqn{x-c_S} and \eqn{f-close-to-A_Q-on-E} in the above estimates to obtain
\[|A_{Q}(x)-f(x)|=|A_{Q}(x)-A_{S}(x)|\lec \rho\diam Q <\rho'\diam Q\]
as desired.

\end{itemize}
\end{proof}

\begin{lemma}
If $S_{1}$ and $S_{2}$ are two adjacent simplexes, then
\[|A_{S_{1}}'-A_{S_{2}}'|\lec\rho.\]
\label{l:adjacent-S}
\end{lemma}

\begin{proof}
Let $Q$ be the smallest cube in $\hitsE$ such that $S_{1}\cup S_{2}\subseteq 3Q$, then as $\cubeDist(Q,Q_{S_{j}})\lec 1$ for $j=1,2$, we have by \eqn{AS'-AQ'}
\[|A_{S_{1}}'-A_{S_{2}}'|\leq \sum_{j=1}^{2}|A_{S_{j}}'-A_{Q_{S_{j}}}'|+|A_{Q_{S_{1}}}'-A_{Q_{S_{2}}}'|\lec \rho.\]
\end{proof}

\medskip

We are now ready to complete the proof of Proposition \ref{p:reifenberg-flat}. 
We establish  that 
\begin{equation*}
\ve_{\sigma}|x-y|\lec_{n} |f(x)-f(y)|\lec_{n} |x-y|
\end{equation*}
by going over different cases as follows.

\begin{itemize}
\item[\bf Case 1:] $x,y\in E$. This follows from Remark \ref{r:bilip-on-E}.
\item[\bf Case 2:] $x\in E$, $y\in E^{c}$. Let $Q$ be the smallest cube containing $x$ such that $y\in 3Q$. Then $\diam Q\sim |x-y|$, and by Lemma \ref{l:almost-affine-extension} and the fact that $|A_{Q}'|\leq 1$,

\begin{align*}
|f(x)-f(y)| 
& \leq |f(x)-A_{Q}(x)|+|A_{Q}(x)-A_{Q}(y)|+|A_{Q}(y)-f(y)| \\
& \lec  \rho\diam Q + |x-y|+\rho'\diam Q \lec |x-y|.
\end{align*}

Furthermore, for $\rho$ small enough, again using Lemma \ref{l:almost-affine-extension}
\begin{multline*}
|f(x)-f(y)|\geq |A_{Q}(x)-A_{Q}(y)|-C\rho'\diam Q\geq \ve_{\sigma}|x-y|-C\rho'|x-y|
\\
\geq \frac{\ve_{\sigma}}{2}|x-y|.
\end{multline*}
\item[\bf Case 3:] $x,y\in E^{c}$. Let $Q\in\cQ$ be the smallest cube containing so that $x,y\in 3Q$. 
\begin{align*}
|f(x)-f(y)|
& \leq |A_{Q}(x)-A_{Q}(y)|+(|f(x)-A_{Q}(x)|+|A_{Q}(y)-f(y)|) \\
& \stackrel{\eqn{almost-affine-extension}}{\leq} C_{1}|x-y|+C_{2}\rho'\diam Q 
\end{align*}
and similarly,
\[|f(x)-f(y)|\geq \ve_{\sigma} |x_{1}-x_{2}|-C_{2}\rho' \diam Q.\]
If $|x-y|\geq \frac{2C_{2}\rho'}{\ve_{\sigma}}\diam Q$, then the above estimates give $|f(x)-f(y)|\sim |x-y|$. 

Assume now $|x-y|<\frac{2 C_{2}\rho'}{\ve_{\sigma}}\diam Q$. This corresponds to when $x$ and $y$ are far away from $E$ with respect to their mutual distance. 
In this case, $x,y$ could be in the same simplex, or they could lie in adjacent simplexes.

Let $S$ and $S'$ be simplexes containing $x$ and $y$ respectively and assume 
\begin{equation}
\diam S\geq\diam S'.
\label{e:S<S'}
\end{equation}

Since
\[|x-y|<\frac{2C_{2}\rho'}{\ve_{\sigma}}\diam Q\sim \frac{\rho'}{\ve_{\sigma}}  \dist(x,E) \stackrel{eq. \eqn{S-dist}}{\sim}  \frac{\rho'}{\ve_{\sigma}}\diam S,\]

by picking $\rho$ small enough we can guarantee
\begin{equation}
\dist([x,y],E)\geq \frac{1}{2}\dist(S,E).
\label{e:[x,y]-to-E}
\end{equation}

Hence the segment $[x,y]$ passes through finitely many distinct simplexes $S_{1},...,S_{k}$. 

Let $\gamma:[0,|x-y|]\rightarrow \bR^{n}$ be the path 
\[\gamma(t)=vt+x,\;\;\; \mbox{ where }\;\;\; v=\frac{(y-x)}{|y-x|}.\]  
Then the path $f\circ \gamma$ is piecewise linear on $[0,|x-y|]$, and its tangent vector at $t$ is $A_{S_{j}}'v$ if $f(t)\in S_{j}$ (except at finitely many points). We will use this path to estimate $|f(x)-f(y)|$, but to do so we will need estimates on the norms of the $A_{S_{j}}'$. 

\Claim: $k\lec_{n}1$. By eq. \eqn{S-dist}, $\diam S_{j}\sim \diam S$ for all $j$, and 
\[\bigcup S_{j}\subseteq B(x,C\diam S)\]
for some $C$ depending only on $n$. Since
\begin{align*}
k(\diam S)^{n} &
\lec \sum_{j=1}^{k}(\diam S_{j})^{n} 
 \stackrel{eq. \eqn{S-volume}}{\lec} \sum_{j=1}^{k} |S_{j}|\\
 & \leq |B(x,C\diam S)|\lec (\diam S)^{n},
 \end{align*}
 we have that $k$ is uniformly bounded by some constant $n_{0}=n_{0}(n)$, which proves the claim.

By Lemma \ref{l:adjacent-S}, for $j=1,...,k$,
\begin{equation}
|A_{S_{j}}'(v)-A_{S}'(v)|\lec n_{0}\rho.
\label{e:S-S_j}
\end{equation}

\Claim:  $\diam Q_{S}\gec \diam Q$. Let
\begin{equation}
\eta = \frac{\diam Q_{S}}{\diam Q}.
\label{e:eta}
\end{equation} 
As $S\subseteq 3Q_{S}$, we know
\[S\subseteq B(c_{S},\diam 3Q_{S}).\]
Since 
\[\dist(S,S')\leq |x-y|\leq \frac{2C_{2}\rho'}{\ve_{\sigma}}\diam Q\]
we further know
\begin{align*}
S\cup S'
& \subseteq B(c_{S},\diam 3Q_{S}+\frac{2C_{2}\rho'}{\ve_{\sigma}}\diam Q+\diam S')\\
& \stackrel{\mbox{eq. \eqn{S<S'}}}{\subseteq} B(c_{S},4\diam Q_{S}+\frac{2C_{2}\rho'}{\ve_{\sigma}}\diam Q) \\
& \stackrel{\mbox{eq. \eqn{eta}}}{\subseteq} B(c_{S},(4\eta+\frac{2C_{2}\rho'}{\ve_{\sigma}})\diam Q).
\end{align*}

By picking $\rho$ small enough so that $\rho'<\frac{\ve_{\sigma}}{40C_{2}\sqrt{n}}$, this shows that if $\eta<\frac{1}{20\sqrt{n}}$, then  
\[S\cup S'\subseteq B(c_{S},\frac{1}{5\sqrt{n}}\diam Q)\]
This means that the triple of the dyadic cube having diameter $\frac{1}{2}\diam Q$ and containing $c_{S}$ also contains $S\cup S'$, but this contradicts the minimality of $Q$. Hence, $\eta \geq \frac{1}{20\sqrt{n}}$, which proves the claim.

The above, combined with  the fact that $3Q_{S}\cap 3Q\neq\emptyset$ implies

\begin{equation}
\cubeDist(Q_{S},Q)\lec 1.
\label{e:cubeDist(Q_{S},Q)<1}
\end{equation}

Combining eq. \eqn{S-S_j}, eq. \eqn{AS'-AQ'}, and eq. \eqn{A_Q'-A_R'} gives
\[|A_{Q}'-A_{S_{j}}'|\lec \rho, \;\;\; j=1,...,k\]
so that, for some constant $C$ depending only on $n$,
\begin{align*}
|f(x)-f(y)|
& =|f(\gamma(0))-f(\gamma(|x-y|))|
=\av{\int_{0}^{|x-y|}(f\circ\gamma)'(t)dt} \\
& \geq \av{\int_{0}^{|x-y|}A_{Q}(v)dt}-C\rho|x-y| \\
& \geq (|A_{Q}'(v)|-C\rho)|x-y|
 \geq \frac{\ve_{\sigma}}{2}|x-y|
\end{align*}
for $\rho<\frac{\ve_{\sigma}}{2C}$  (recall that since $|v|=1$, $|A_{Q}(v)'|\geq \sigma(A_{Q}')\geq \ve_{\sigma}$). 

The proof of the reverse inequality
\[|f(x)-f(y)|\lec |x-y|\]
is similar and we omit the details.

{
This concludes the proof of Proposition \ref{p:reifenberg-flat} and Theorem~{{III}}.
}
\end{itemize}

%
%
%
%

\section{Proof of Theorem~{{II}}}\label{s:pf-of-thm-3}

Below, $\bt$ always refers to $\bt_{\tilde f}$. In addition, a cube will always be either a dyadic cube, or a triple of a dyadic cube.  A triple of a cube will always be written as $3Q$ for some dyadic cube $Q$.

We set $\ve_\sigma=\kappa$.

\subsection{Sorting cubes}\label{ss:sorting-cubes}

In this section, we will sort the dyadic cubes into a finite number of collections, using the scheme from \cite{Jones-lip-bilip}, with some minor alterations.

Fix $\ve_{\sigma}>0$, let $\rho>0$ and $N\in\bN$ to be determined later, and let $\ve_{\beta}$ be the constant so that the conclusion of Lemma \ref{l:affine-lemma-2} holds. The constant $N$ will only depend on $D$, and in fact, $N\sim_{D} \log\frac{1}{\ve_{\sigma}}$. 

\begin{definition}
\label{d:semi-adjacent}
We say two distinct cubes $Q$ and $R$ are \textit{semi-adjacent} if they have the same size and
\begin{equation*}
Q\subseteq 3R^{N}\;\;\; \mbox{ or } \;\;\; R\subseteq 3Q^{N}.
\end{equation*}
\end{definition}

For $N_{0}\in\bN$, define
\begin{equation*}\cE_{1}=\big\{Q: \bt(3Q^{N})>\ve_{\beta} \big\},\end{equation*}
\begin{equation*}\cE_{2}= \bigg\{x:\sum_{\bt(3Q^{N})>\ve_{\beta}}\one_{Q}(x)>N_{0} \bigg\}
\end{equation*}
\begin{equation}
\bcubes= \big\{Q\not\in \cE_{1}:\sigma(3Q^{N})<\ve_{\sigma} \big\}
\label{e:b-cubes}
\end{equation}
and

 
\begin{equation*}
\cA:= \Delta\setminus(\cE_{2}\cup \bcubes).
\end{equation*}
Order the pairs of semi-adjacent cubes in $\cA\times\cA$ so that pairs of larger size come before pairs of smaller size. We will associate to each cube $Q$ a word $w(Q)$ (initially the empty word) with letters in $\{\pm 1\}$ and, in the case $D=n$, an orientation $\epsilon(Q)\in\{\pm 1\}$ (initially $+1$), by executing an algorithm that runs through each pair of cubes in order, changing the word $w(Q)$ and orientation $\epsilon(Q)$ at most finitely many times in the process. (If $D>n$, then $\ve(Q)$ is not necessary, and can be set to 1 in the work below.) After all the changes have been done, each cube will have been given one of no more than $2^{c(n)N_{0}}$ many possible words, where $c(n)$ grows exponentially in $n$. 
{\JARSbluJARS     
Suppose we have started our process and have reached a pair of cubes $(Q,R)$.
\begin{enumerate}
\item[\bf Case 1:] If $Q\not\in \cE_{1}$, then set $w(Q)=w(Q^{1})$. If $Q'$ was the smallest ancestor of $Q$ such that $Q\not\in \cE_{1}$ and 
$\det(A^f_{3Q^{N}})\det(A^f_{3(Q')^{N}})<0$ (assuming we're in the case $D=n$)
then let $\epsilon(Q)=-\epsilon(Q')$. Otherwise, set $\epsilon(Q)=\epsilon(Q')$.  {\JARSbluJARS     If no such ancestor $Q'$ exists, set $\epsilon(Q)=\epsilon(Q^1)$}. \\

In the next few cases, we will assume $Q\in \cE_{1}$.

\item[\bf Case 2:] Suppose first that the lengths of the words $w(Q^{1})$ and $w(R^{1})$ are equal.
Then

\begin{itemize} 
\item if  $w(Q^{1})\neq w(R^{1})$, then set   $w(Q)= w(Q^{1})$ . 
\item If  $w(Q^{1})=w(R^{1})$, let $w(Q)=(w(Q^{1}),-1)$ and $w(R)=(w(R^{1}),1)$.
\end{itemize}

\item[\bf Case 3:] Suppose now that $w(R^{1})$ has length $\ell$ strictly less than the length of $w(Q^{1})$. In this case, set $w(Q)=w(Q^{1})$ and $w(R)=(w(R^{1}), -(w(Q^{1}))_{\ell})$ where $(w(Q^{1}))_{\ell}$ is the $\ell$th letter of the word $w(Q^{1})$. 
\end{enumerate}
}
After this process, each cube will have an orientation $\pm 1$ and a word of length no more than $N_{1}=2^{c(n)N_{0}}$, for otherwise that cube would be contained in $\cE_{2}$ by definition. Order these word-orientation pairs by $w_{1},...,w_{N_{1}}$, so each $w_{j}$ equals a pair $(w,\ve)$ where $w$ is a word of length at most $N_{1}$ with letters in $\pm 1$, and $\epsilon\in\{\pm 1\}$. 
\begin{equation}
E_{j}=\{x: \sum_{(w(Q),\ve(Q))=w_{j}, Q\in \cA} \one_{Q^{\circ}}=\infty\}.
\label{e:E_j}
\end{equation}

%

The role of the orientations will be explained later in the paper. {\JARSbluJARS     

\begin{remark}
Let $E$ be one of the $E_{j}$. We record some simple properties of $E$.
\begin{enumerate}
\item  We have constructed $E$ to omit the boundaries of all dyadic cubes, which form a measure zero set. Hence, if $Q,R\in\Delta$ and $Q\cap R\cap \Ejpm\neq\emptyset$, then $Q^{\circ}\cap R^{\circ}\neq\emptyset$.
\item If $Q\in \cQ_{E}\cap \cE_{1}$, then $3Q^{N}\cap E=Q\cap E$. This is because of our labeling process: any cube $R$ that is semi-adjacent to $Q$ has $w(Q)\neq w(R)$, and hence $R\cap E$ by definition of $E$.
\end{enumerate}
\label{r:interiors}
\end{remark}

{\JARSbluJARS    
\begin{remark}
By Remark \ref{r:smaller-lip-constant} and the fact that $E$ is disjoint from $\bigcup_{Q\in\cE_{3}} Q$, we know that if $Q\in \cap \cQ_{E}([0,1]^{n})\backslash \cE_{1}$, then $A_{3Q^{N}}^{f}$ is $(\ve_{\sigma},1)$-bi-Lipschtz.
\label{r:A_Q-are-bilip}
\end{remark}
}

\begin{lemma}\cite{Jones-lip-bilip,Schul-lip-bilip}
Theorem~{{II}}, eq. \eqn{inverse-theorem-1} holds with this choice of sets $\{E_{j}\}$. \end{lemma}

\begin{proof}

By the definition of $\cE_{2}$, Theorem \ref{t:TST}, and Chebichev's inequality, 
\begin{align*}
|\bigcup \cE_{2}| & \leq \frac{1}{N_{0}}\int \sum_{\bt(3Q^{N})>\ve_{\beta}}\one_{Q} \leq 
\frac{1}{N_{0}\ve_{\beta}^{2}}\sum_{Q}\bt(3Q^{N})^{2}|Q|\lec \frac{1}{N_{0}\ve_{\beta}^2}<\ve_{\sigma}
\end{align*}
for $N_{0}>\frac{1}{\ve_{\sigma}\ve_{\beta}^{2}}$. Clearly, we then have
\begin{equation}
|f(\bigcup \cE_{2})|\lec \ve_{\sigma}.
\label{e:e_2}
\end{equation}

Let $Q\in \cE_{3}$. Since $\diam f(Q)\leq \diam Q$ and $\sigma(Q^{N})<\ve_{\sigma}$, we know $\beta_{f}^{(n-1)}(Q^{N})\lec \ve_{\sigma}$ by equation \eqn{bt-sigma}. This implies that $f(Q)$ is contained in an 
$C\ve_\sigma\diam(Q)$ neighborhood of an $n-1$ dimensional parallelogram (for some constant $C$ depending on $\newn$), of diameter $\lesssim \diam Q$, call this set $W_{Q}$. Then 
\[\cH^{n}_{\infty}(f(Q))\leq \cH^n_\infty(W_Q)\lec \ve_{\sigma} |Q|.\]
Let $Q_{j}$ denote the maximal cubes in $\cE_{3}$. Since $\cH^{n}_{\infty}$ is countably  sub-additive,
\begin{equation}
\label{e:e_3}
\cH^{n}_{\infty}(f(\bigcup\cE_{3}) 
 \leq \sum_{Q\in \cE_{3}\mbox{ maximal}}\cH^{n}_{\infty}(f(Q)) 
 \lec \sum_{Q\in \cE_{3}\mbox{ maximal}} \ve_{\sigma} |Q|\leq \ve_{\sigma}. 
\end{equation}

As $[0,1]^{n}\backslash\bigcup E_{i}^{\pm}=\bigcup_{Q\in \cE_{2}\cup\cE_{3}}Q$, \eqn{inverse-theorem-1} follows from \eqn{e_2} and \eqn{e_3}.

\end{proof}


\subsection{Stopping and Restarting cubes}
Let $E$ be one of the $\Ejpm$ that is nonempty, and let $Q_0\supseteq E$ be the smallest dyadic cube containing $E$ (recall that $E\subseteq [0,1]^{n}$, so such a cube exists).

\medskip

We will define collections of cubes $\highc_{k}$ and $\lowc_{k}$ as follows: Define 
\[\lowc_{0}=\{Q_{0}\}.\] 
If we have defined $\lowc_{k}$, for each $\low{Q}\in\lowc_{k}$, let 
\begin{equation*}
\highc_{k+1}(\low{Q})=\{R^{1}: R\subsetneq\low{Q}\mbox{ maximal in }\cE_{1}\mbox{ s.t. }R\cap E\neq\emptyset\}
\end{equation*}
and
\begin{equation*}
\highc_{k+1}=\bigcup_{\low{Q}\in\lowc_{k}}\highc_{k+1}(Q).
\end{equation*}
If $Q\in\highc_{k+1}$, let $\low{Q}\subseteq Q$ be the smallest cube for which $E\cap Q=E\cap\low{Q}$ and let $\lowc_{k+1}=\{\low{Q}:Q\in\highc_{k+1}\}$.

{\JARSbluJARS     
\begin{lemma}\label{l:good-kids}
If $Q\in \cQ_{E}\cap \Delta(Q_{0})$, let $\low{Q}\subseteq Q$ denote the smallest dyadic cube such that $Q\cap E=\low{Q}\cap E$. If $R\subsetneq \low{Q}\subseteq R^{N}$ and $R\cap E\neq\emptyset$, then 
\[\bt(3R^{N})\leq\ve_{\beta}.\]
In particular, $R^{N}\subseteq \low{Q}$ whenever $\low{Q}\in \lowc_{k}$  and $R\in \highc_{k+1}(\low{Q})$.
\end{lemma}
\begin{proof}
If $R^{N}\supseteq \low{Q}$ has $\bt(3R^{N})>\ve_{\beta}$, then $R\in \cE_{1}$ and by Remark \ref{r:interiors}
\[ R\cap E\subseteq \low{Q}\cap E\subseteq R^{N}\cap E=R\cap E.\]
Since $R\subsetneq \low{Q}$, this contradicts the minimality of $\low{Q}$. 

Now, suppose $\low{Q}\in \lowc_{k}$  and $R\in \highc_{k+1}(\low{Q})$. Then $R=T^{1}\subseteq \low{Q}$ where $T\in \cE_{1}$, hence $T^{N}$ can't possibly contain $\low{Q}$, hence $R^{N-1}=T^{N}\subsetneq \low{Q}$, that is, $R^{N}\subseteq \low{Q}$.
\end{proof}

%


\begin{lemma}
For each $Q\in \highc_{k}$,
\begin{equation*}
\bt(3Q^{N})<\ve_{\beta},  \;\;\; \mbox{ and } \;\;\;\; \bt(3\low{Q}^{N-1})<\ve_{\beta}.
\end{equation*}
\end{lemma}
\begin{proof}

Let $T\in \Delta$ be so that $Q\in \highc_{k}(\low{T})$ and $R\subsetneq \low{T}$ be so that $Q=R^{1}$. Since $R$ is a maximal cube in $\cE_{1}$ such that $R^{1}\subseteq \low{T}$, this means $Q=R^{1}\in\Delta(\low{T})\backslash \cE_{1}$, that is, $\bt(3Q^{N})<\ve_{\beta}$. 
%
The second inequality of the lemma follows from Lemma \ref{l:good-kids} since any child $R$ of $\low{Q}$ intersecting $E$ satisfies $R^{N}\supseteq \low{Q}$.


\label{l:Ck-have-beta-small}
\end{proof}

}

 We record some simple, and yet crucial, properties of the cubes in $\highc_{k}$. 
\begin{lemma}
If $Q,R\in \highc_{k}$ are distinct, then
\begin{equation}
(3Q^{N-1}\backslash Q)\cap E=(3R^{N-1}\backslash R)\cap E
=\emptyset,
\label{e:3Q^N-1-E-disjoint}
\end{equation}
\begin{equation}
3Q^{N-1}\cap R^{\circ}=\emptyset = 3R^{N-1}\cap Q^{\circ},
\label{e:3Q^{N-1}-R-disjoint}
\end{equation}
and
\begin{equation}
(3Q^{N-2})^{\circ} \cap (3R^{N-2})^{\circ}=\emptyset.
\end{equation}
\label{l:disjoint-triples}
\end{lemma}

\begin{proof}
To see the first equation, first note that by construction $Q$ is the parent of a cube $Q'\in\cE_{1}$, and so by Remark \ref{r:interiors},  $3Q^{N-1}\cap E=3(Q')^{N}=Q\cap E$. 

To show the second equation, suppose $3Q^{N-1}\cap R^{\circ}\neq\emptyset$. If $\side(R)\leq\side(Q^{N-1})$, then $3Q^{N-1}\backslash Q\supseteq R^{\circ}$. Since $R\in \highc_{k}$ and (by definition of $E$) $\d R\cap E=\emptyset$, we know $R^{\circ}\cap E=R\cap E\neq\emptyset$, thus $3Q^{N-1}\backslash Q\cap E\neq\emptyset$, contradicting \eqn{3Q^N-1-E-disjoint}. Alternatively, if $\side(R)>\side(Q^{N-1})$, then $3Q^{N-1}\cap R^{\circ}\neq\emptyset$ implies $Q\subseteq Q^{N-1}\subseteq 3R$. Since cubes in $\highc_{k}$ have disjoint interiors, this implies $Q^{\circ}\subseteq 3R\backslash R\subseteq 3R^{N-1}\backslash R$, but again $Q^{\circ}\cap E\neq\emptyset$ since $Q\in \highc_{k}$, which contradicts \eqn{3Q^N-1-E-disjoint}.

For the final equality, note that if $(3Q^{N-2})^{\circ}\cap (3R^{N-2})^{\circ}\neq\emptyset$ and $\side(Q)\geq \side(R)$, say, then $3Q^{N-1}\supseteq R^{N-2}$, so in particular, $3Q^{N-1}\cap R^{\circ}\neq\emptyset$, contradicting \eqn{3Q^{N-1}-R-disjoint}.


\end{proof}
}

In the rest of this section we will proceed as follows.
We will define bi-Lipschitz homeomorphisms of $\bR^{n}$ that agree with $f$ on  pieces of $E$. 
Later, we will restrict the domains of these extensions, in such a way that the new 
domains partition $\bR^{n}$, 
and sew them together at the boundaries to obtain a bi-Lipschitz extension of $f|_{E}$.

\subsection{Extending inside cubes in $\lowc_{k}$.}

If $R_{0}\in\RESTART_{k-1}$, let 
\[E_{R_{0}}=E\cap R_{0},\] 
and enumerate the cubes in $\STOP_{k}(R_{0})$ by
\[\{R_{1},R_{2},R_{3},...\}=\STOP_{k}(R_{0}).\]
Set
\[Q_{j} =R_{j}^{N-4}, \;\;\; j=1,2,...\]
\[\island=\island(R_{0})=\{Q_{j}\}\cup\{\{x\}:x\in E_{R_{0}}\},\]

and
\[\cQ'=\{Q\in  \Delta(R_{0}):\exists R\in \island, 
R\subseteq Q.\}.\]

We are preparing  to apply Lemma \ref{l:replacement-lemma} to the set $E_{R_{0}}\subseteq R_{0}$ and cubes $Q_{j}$, and function $f:E_{R_{0}}\rightarrow \bR^{D}$, and so we need to show that this data satisfies the conditions of the lemma.

\begin{lemma}
\label{l:Q_j-satisfy-lemma}
With notation as above, 
\begin{enumerate}
\item$Q_{j}^{2}\subseteq R_{0}$ for $j=1,2,...$
\item  $\{3Q_{j}^{2}\}_{j\in \bN}$ have disjoint interiors.
\end{enumerate}

\end{lemma}
{\JARSbluJARS     
\begin{proof}
\noindent 1. Since $R_{j}\in \lowc_{k-1}$, by Lemma \ref{l:good-kids}, $Q_{j}^{4}=R_{j}^{N}\subseteq R_{0}$.

%
%
%
%
%

\noindent 2. The cubes $Q_{j}$ 
satisfy 
\[(3Q_{i}^{2})^{\circ}\cap (3Q_{j}^{2})^{\circ}=(3R_{i}^{N-2})^{\circ}\cap (3R_{j}^{N-2})^{\circ}=\emptyset\]
 whenever $i\neq j$ by  Lemma \ref{l:disjoint-triples}. 
 
 \end{proof}

}

Let 
 \[E_{R_{0}}'=\bigcup 3Q_{j}\cup\bigcup \ps{E_{R_{0}}\backslash\bigcup 3Q_{j}^{2}}.\]

\begin{remark}
 {\JARSbluJARS     
Note that if $x\in E_{R_{0}}\backslash \bigcup 3Q_{j}$, then $x$ is not contained in any cube $Q\in \Delta(R_{0})\cap \cE_{1}$. Moreover, if $Q\in \cQ'$ contains some $Q_{j}$, then $Q=T^{N-4}$ for some $T\in \Delta(R_{0})$ containing $R_{j}$. Since $R_{j}\in \highc_{k}(R_{0})$, it is the parent of a maximal cube in $\cQ_{E}\cap \cE_{1}\cap \Delta(\bR_{0})$, hence any cube in $\Delta(R_{0})$ properly containing it must be in $\Delta(R_{0})\backslash \cE_{1}$, so in particular, $T\in \Delta(R_{0})\backslash \cE_{1}$. These two observations imply that, for $Q\in \cQ'$, $Q=T^{N-4}$ for some $T\in \Delta(R_{0})\backslash \cE_{1}$, that is, $\bt(3Q^{4})=\bt(3T^{N})<\ve_{\beta}$.} By Lemma \ref{l:affine-lemma-2} (for $\ve_{\beta}$ small enough depending on $N$ and $\rho$), 
$\cQ'$ together with the function $f:E_{R_{0}}'\rightarrow\bR^{D}$,  and the collection {\JARSbluJARS     $\{A_{Q}:Q\in\cQ'\}$  where $A_{Q}:=A_{3Q^{4}}^{f}$},
satisfy \eqn{f-close-to-A_Q-on-E} and \eqn{close-derivatives}. 
\label{r:new-A_Q}
\end{remark}

Because of Remark \ref{r:new-A_Q} and Lemma \ref{l:Q_j-satisfy-lemma}, we can apply Lemma \ref{l:replacement-lemma} to obtain a function $F:E_{R_{0}}'\rightarrow\bR^{D}$ that satisfies
 \[F_{R_{0}}|_{3Q_{j}}=A_{Q_{j}},\]
 \[F_{R_{0}}|_{E_{R_{0}'}\backslash \bigcup 3Q_{j}}=f,\]
 and
 \[F_{R_{0}}|_{(3R_{0})^{c}}=A_{R_{0}}\]
and is Reifenberg flat as a function from a subset of $\bR^{n}$ into $\bR^{D}$.
By Lemma \ref{l:D-equals-n-suffices}, $F_{R_{0}}$ is also Reifenberg flat as a function from a subset of $\bR^{D}$ to $\bR^{D}$ with associated affine transformations 
$\{M_{Q}:Q\in \cQ_{E'_{R_{0}}}(\bR^{D})\}$ (recall that $\cQ_{E'_{R_{0}}}(\bR^{D}))$ are $D$-dimensional dyadic cubes intersecting $E_{R_{0}}'$). 

For $Q\in \Delta(\bR^{n})$, define
\[\starit{Q}=Q\times [0,\side{Q}]^{D-n} \in \Delta(\bR^{D}).\]

If $n<D$, we would like to extend $F_{R_{0}}$ to all of $\bR^{D}$ so it remains affine on the cubes $3\starit{Q}_{j}$ and outside $(3\starit{R}_{0})^{c}$, which would be possible if $M_{\starit{Q}}$ remained constant for $Q\subseteq Q_{j}$ and $Q\supseteq \starit{R}_{0}$ (just as the $A_{Q}$ do in these cases). The statement of Lemma \ref{l:TV}, however, doesn't say this happens explicitly. One could go back to the proof of the original Lemma and show that it is possible to make a choice of $\psi$ so that this happens. For the sake of brevity, however, we content ourselves with applying Lemma \ref{l:replacement-lemma} a second time with the cubes $\{\starit{Q}_{j}\}$, function $F_{R_{0}}$, and set $E_{R_{0}}$ (as a subset of $\bR^D$) to obtain $f_{R_{0}}$ that is also Reifenberg flat as a function from a subset of $\bR^{D}$ to $\bR^{D}$ and satisfies
 \begin{equation}
 f_{R_{0}}|_{(3\starit{R}_{0})^{c}}=M_{\starit{R}_{0}}=\left[A_{R_{0}} \bigg| \psi_{R_{0}}(R_{0})\right],
 \label{e:brev(A_R_0)} 
 \end{equation}
 \begin{equation}
 f_{R_{0}}|_{E_{R_{0}}'\backslash\bigcup 3Q_{j}^{2}}=f,
 \end{equation}
and
 \begin{equation}
 f_{R_{0}}|_{3\starit{Q}_{j}}=M_{\starit{Q}_{j}}=\left[A_{Q_{j}} \bigg| \psi_{R_{0}}(Q_{j})\right],
 \label{e:brev(A_Q^f)}
 \end{equation}
where $\psi_{R_{0}}:\Delta(\bR^{n})\rightarrow\cF_{D}^{D}$ is the function from Lemma \ref{l:TV} (see also Remark \ref{r:staritA_Q}).

If $n=D$, then we just let $f_{R_{0}}=F_{R_{0}}$.

\begin{remark}
From here on, we will abuse notation and write $f_{R_0}$ for the extension of this alteration to all of $\bR^D$, whose existence follows from the use of Proposition \ref{p:reifenberg-flat} (if we choose $\rho$ small enough depending on $\ve_{\sigma}$).
 \end{remark}
 
 \begin{remark}
It is important that Lemma \ref{l:TV} (3) grants us some freedom in selecting $\psi_{R_{0}}$. 
 {\JARSbluJARS     
 If $R_{0}=Q_{0}$, then we pick it arbitrarily. Inductively, if $R_{0}=\low{Q}$ for some $Q \in \STOP_{k}$, where we have already chosen a corresponding $\psi_{Q}$, we pick $\psi_{R_{0}}$ so that $M_{\starit{R}_{0}}$ has the same orientation as $M_{\starit{Q}}$ does. If $n=D$, then these still have the same orientation by our sorting process that constructed the set $E$. This property will be crucial in the next section.
 }
 \label{r:orientation}
\end{remark}

\subsection{Extending inside cubes in $\highc_{k}$} \label{s:Extending-inside-stopped-cubes}

We will now define a similar map $f_{Q}$ for cubes $Q\in\highc_{k}$.

Let $r=\frac{\diam \starit{Q}^{N-4}}{\sqrt{n}}$ and observe that

\[B(x_{\low{Q}},r)\subseteq 3Q^{N-4}\]

and hence
\[\starit{B}(x_{\low{Q}},r)\subseteq 3\starit{Q}^{N-4}\]
where $\starit{B}$ denotes the closed ball in $\bR^{D}$ (whereas $B$ without a star denotes a ball in $\bR^{n}$).

For $N$ large enough (depending only on $n$ and $\ve_{\sigma}$), we may guarantee that $B^{3\low{Q}}\subseteq B(x_{\low{Q}},\frac{\ve_{\sigma}r}{4})$ (recall \eqn{B^Q} for notation).

\begin{figure}[hbpt]
\begin{picture}(200,200)(-75,0)
\put(83,80){$\low{Q}$}
\put(90, 52){$B(x_{\low{Q}},\frac{r\ve_{\sigma}}{4})$}
\put(110,120){$Q^{N-4}$}
\put(172,190){$3Q^{N-4}$}
\put(35,20){$B(x_{\low{Q}},r)$}
\scalebox{0.4}{\includegraphics{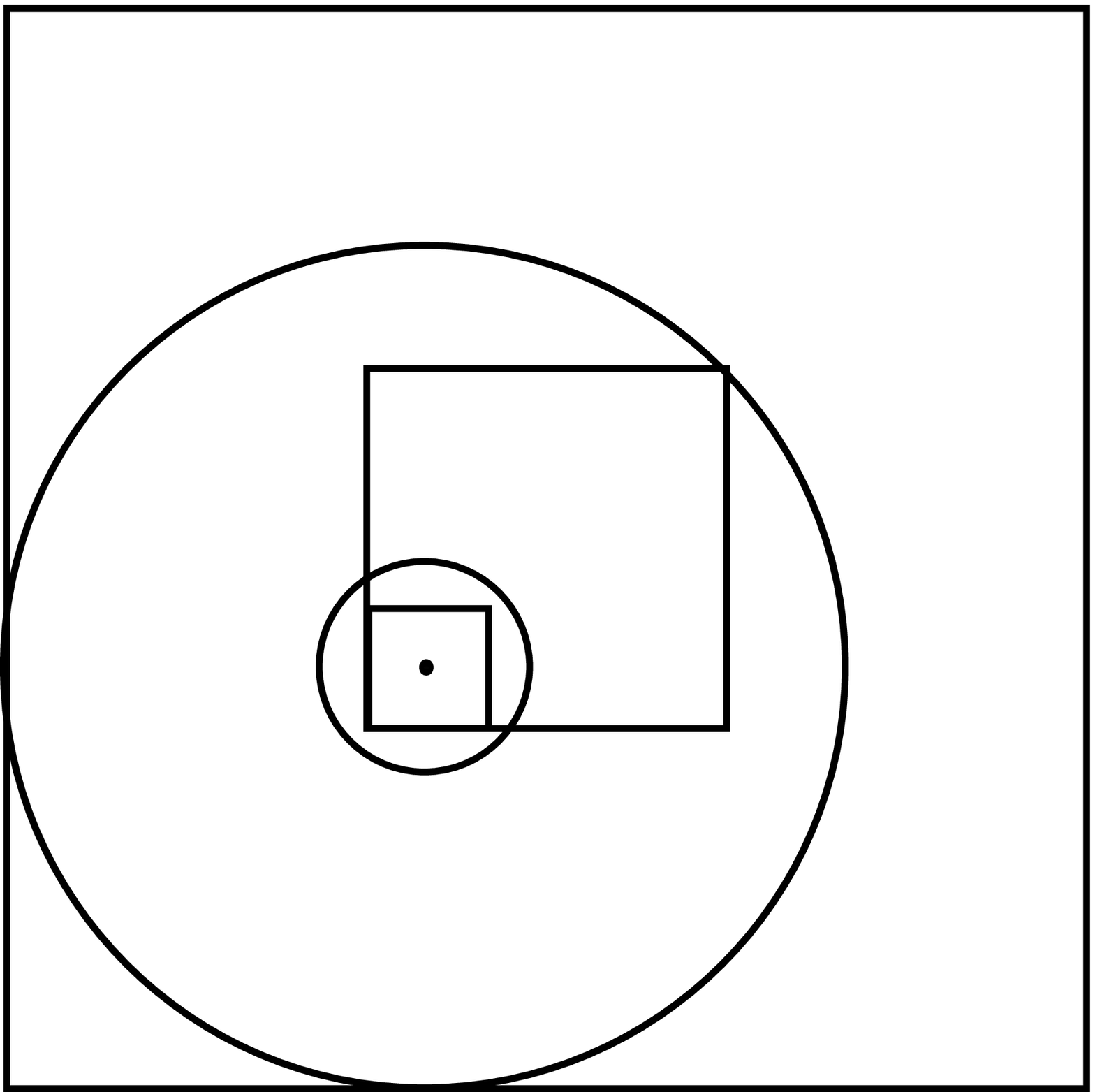}}
\end{picture}
\caption{Extending inside $\highc_k$}
\end{figure}

First define $f_{Q}$ to be the function on $\starit{B}(x_{\low{Q}},r)^c \cup \starit{B^{3\low{Q}}}$ satisfying

\begin{equation*}
f_{Q}(x)=\isif{
M_{Q}(x) 		& x\in  \starit{B}(x_{\low{Q}},r)^c \\
M_{\check{Q}}(x) 	& x\in \starit{B^{3\low{Q}}}},
\end{equation*}

 (recall how these matrices are defined in \eqn{brev(A_R_0)} and \eqn{brev(A_Q^f)}).

\begin{lemma} 
For $\rho>0$ small enough (depending on $\ve_{\sigma}$),
\[M_{\low{Q}}(\starit{B}^{3\low{Q}})\subseteq M_{Q}(\frac{1}{2}\starit{B}(x_{\low{Q}},r)).\]
\label{l:ball-in-half-ball}
\end{lemma}

\begin{proof}

By Remark \ref{r:new-A_Q}, the $A_{\low{Q}}$ and $A_{Q}$ satisfy \eqn{f-close-to-A_Q-on-E} and \eqn{close-derivatives}, and so
\begin{multline}
|A_{\low{Q}}(x_{Q})-A_{Q}(x_{Q})|
 \leq |A_{\low{Q}}(x_{Q})-f(x_{Q})|+|f(x_{Q})-A_{Q}(x_{Q})| \\
 \stackrel{\eqn{f-close-to-A_Q-on-E}}{\leq} \rho(\diam \low{Q}+\diam Q^{N-4})\lec \rho\diam Q^{N-4}
\label{e:top-and-bottom}
\end{multline}
By Remark \ref{r:A_Q-are-bilip}, $A_{Q}$ is $(1,\frac{1}{\ve_{\sigma}})$-bi-Lipschitz. Hence, for some constant $C=C(n)>0$ and all $x\in B^{3\low{Q}}\subseteq \ball(x_{\low{Q}},\frac{\ve_{\sigma}r}{4})$,

\begin{multline*}
|(M_{Q})^{-1}\circ M_{\check{Q}}(x)-x_{Q} |
=|((M_{Q})')^{-1}(M_{\check{Q}}(x)-M_{Q}(0))-x_{Q}| \\
=|((M_{Q})')^{-1}(M_{\check{Q}}(x)-M_{Q}(x_{Q})+\underbrace{M_{Q}(x_{Q})-M_{Q}(0)}_{(M_{Q})'(x_{Q})})-x_{Q}| \\
= |((M_{Q})')^{-1}(M_{\check{Q}}(x)-M_{\check{Q}}(x_{Q})
+M_{\check{Q}}(x_{Q})-M_{Q}(x_{Q}))| \\
=|((M_{Q})')^{-1}((M_{\check{Q}})'(x-x_{Q})+A_{\low{Q}}(x_{Q})-A_{Q}(x_{Q}))| \\
\leq \frac{1}{\ve_{\sigma}}|x-x_{Q}|+C\rho\diam Q^{N-4}
\stackrel{\eqn{top-and-bottom}}{\leq} \frac{1}{\ve_{\sigma}}\frac{\ve_{\sigma} r}{4}+\frac{r}{4}=\frac{r}{2}
\end{multline*}
if $\rho>0$ is small enough (depending on $N$, $n$, and $\ve_{\sigma}$). In the penultimate line, as $x_{Q}\in \bR^{n}$, we know $M_{\check{Q}}$ agrees with $A_{\check{Q}}$ here as $M_{\check{Q}}|_{\bR^{n}}=A_{\check{Q}}$. In establishing the last line, we have used the fact that 
$((M_{Q})')^{-1}$    has norm at most $\frac{1}{\ve_{\sigma}}$ and $M_{\check{Q}}$ has norm at most $1$. This establishes the lemma.
\end{proof}

Note that $f_{Q}$ is affine on $\d \starit{B}(x_{\low{Q}},r)$ and $\d \starit{B^{3\low{Q}}}$. By Remark  \ref{r:orientation}, $M_{\low{Q}}$ and $M_{Q}$ have the same orientation. By Lemma \ref{l:ball-in-half-ball} above and Lemma \ref{l:interpolation-lemma}, we may now deduce that one can bi-Lipschitz extend $f_{Q}$ into all of $3Q^{N-2}$ (and hence to a bi-Lipschitz homeomorphism of all of $\bR^{D}$):

\begin{lemma}[Interpolation Lemma]
Let $B_{1}\subseteq B_{2}$ be balls such that $\dist(B_{1},\d B_{2})\geq \frac{1}{4}\diam B_{2}$, $A_{1}$ and $A_{2}$ are two affine transformations of the same orientation with $\min\{\sigma(A_{1}),\sigma(A_{2})\}\geq \sigma>0$ and such that
\begin{equation}
A_{1}B_{1}\subseteq \frac{1}{2}A_{2}B_{2}.
\label{e:ellipse-spacing}
\end{equation}

Then there is a bi-Lipschitz map that is equal to $A_{j}$ on $\d B_{j}$ for $j=1,2$.
\label{l:interpolation-lemma}
\end{lemma}

To see the details, see Lemma \ref{l:interpolation-lemma} in the Appendix.

\subsection{Sewing the functions together} \label{s:sewing}


Set $f_{0}^{1}=f_{Q_{0}}$. For $k>0$, let
\[f_{k}^{0}(x)=\isif{f_{k-1}^{1}(x), & x\in \ps{\bigcup\{ 3\starit{Q}^{N-2}: Q\in \highc_{k}\}}^{c} \\
f_{Q}(x), & x\in 3\starit{Q}^{N-2},\;\;\; Q\in \highc_{k}}\]
and
\[f_{k}^{1}(x)=\isif{f_{k}^{0}(x), &  x\in \ps{\bigcup\{3\starit{R}:R=\low{Q}\in\lowc_{k}\}}^{c} \\
f_{\low{Q}}(x), & x\in 3\starit{R},\;\;\; R=\low{Q}\in\lowc_{k}}.\]

Let $f_{2k}=f_{k}^{0}$ and $f_{2k+1}=f_{k}^{1}$.

\begin{lemma}
Each $f_{j}$ is uniformly $\frac{C}{\ve_{\sigma}}$-bi-Lipschitz.
\end{lemma}

\begin{proof}
Suppose we have shown the lemma for each $j<k$. The function $f_{k}$ is obtained from $f_{k-1}$ by replacing $f_{k-1}$ on a collection of cubes $\{Q_{j}\}$ in $\bR^{D}$ with bi-Lipschitz homeomorphisms $f_{Q_{j}}$ such that 
\[f_{Q_{j}}(Q_{j})=f_{k-1}(Q_{j})\;\;\;\mbox{ and } \;\;\; f_{Q_{j}}|_{\d Q_{j}}=f_{k-1}|_{\d Q_{j}}.\]
Thus, $f_{k}$ is also homeomorphism.

The cubes $Q_{j}$ are either of the form $3\starit{Q}^{N-2}$ for some $Q\in \highc_{j}$ and some integer $j$, or $3\starit{R}$ for some $R=\low{Q}\in \lowc_{j}$. In either case, their are disjoint by Lemma \ref{l:disjoint-triples}.

Moreover, $f_{k-1}$ is affine on each $Q_{j}$, so their images are convex. Let $x,y\in \bR^{n}$ and $I_{j}=[f_{k}(x),f_{k}(y)]\cap f_{k}(Q_{j})$ and $F=[f_{k}(x),f_{k}(y)]\backslash\bigcup I_{j}$. Since $f_{k}$ is a homeomorphism, we have
\[f_{k}^{-1}([f_{k}(x),f_{k}(y)])=\bigcup f_{k}^{-1}(I_{j})\cup f_{k}^{-1}(F)=\bigcup f_{Q_{j}}^{-1}(I_{j})\cup f_{k}^{-1}(F)\]
which is a path connecting $x$ and $y$.
As $f_{k}$ is bi-Lipschitz on each $Q_{j}$ respectively, 
and is bi-Lipschitz on $\ps{\bigcup Q_{j}}^{c}$, we have
\begin{align*}
|f_{k}(x)-f_{k}(y)|
& =\cH^{1}([f_{k}(x),f_{k}(y)])
=\cH^{1}(F) +\sum_{j}\cH^{1}(I_{j})\\
& \gec \ve_{\sigma}\cH^{1}(f_{k}^{-1}(F))+\ve_{\sigma}\sum \cH^{1}(f_{Q_{j}}^{-1}(I_{j}))\\
& =\ve_{\sigma}\cH^{1}(f_{k}^{-1}[f_{k}(x),f_{k}(y)])\geq \ve_{\sigma}|x-y|.
\end{align*}
A similar proof shows that $|f_{k}(x)-f_{k}(y)|\lec \frac{1}{\ve_{\sigma}}|x-y|$. Thus $f_{k}$ is bi-Lipschitz, and particular, its bi-Lipschitz constant is the maximum of those of $f_{Q_{j}}$ and $f_{k-1}$.\\
\end{proof}

The sequence $f_k$ stabilizes after $2N_0$ iterations, and we obtain a bi-Lipschitz extension $f=f_{2N_0}:\bR^{D}\rightarrow\bR^{D}$, and restricting $f$ to $\bR^{n}$ gives our desired bi-Lipschitz extension $f:\bR^{n}\rightarrow\bR^{D}$
The bi-Lipschitz constant is obtained from the bi-Lipschitz constant given by Proposition \ref{p:reifenberg-flat} and is 
$\sim_{D} \frac{1}{\epsilon_\sigma}$.
\bigskip
 
This completes the proof of Theorem~{{II}},with $\kappa=\ve_{\sigma}$.


\section{Proof of Theorem~{{I}}}\label{s:pf-of-thm-1}

Theorem~{{I}} will follow  from Theorem \ref{t:TST} coupled together with 
Theorem \ref{t:main-verson-1-restated} stated below,
 as well as a small lemma.
\begin{theorem}
Let $f:\bR^{n+m}\rightarrow \cM$ be a $1$-Lipschitz function. 
Let 
$Q_1\subseteq \bR^{n+m}$ be a dyadic cube.
Suppose 
$$\delta=\frac{\cH^n_\infty(f(Q_1))}{\side(Q_1)^n} >0\,.$$
Assume further that 
$0<\cH^n(f(Q_1)\leq 1$.
There are  constants $\Lip>1$, $\ve_\beta>0$, and  $\eta>0$, as well as an integer $N$,
all  of which depend only on 
$n,m$ and $\delta$,  such that  
if
$$\bt_{\tilde{f}}(3Q_1^N)<\ve_\beta$$
then there is 
a set $E\subset Q_1$, and a 
homeomorphism $g:\bR^{n+m}\rightarrow \bR^{n+m}$ such that if $F=f\circ g^{-1}$ then

\begin{enumerate}[(i)]
\item $\cH^{n+m}(E)\geq \eta\cH^{n+m}(Q_1)$ ,
 \item
 $g$ is  $\Lip$-bi-Lipschitz,
 
 \item
for $(x,y)\in \bR^{n}\times\bR^{m}$ if $(x,y) \in g(E)$, then 
$$
F^{-1}\big(F(x,y)\big) \cap g(E)\subseteq 
g(E)\cap \big(\{x\}\times\bR^{m}\big),$$
 \item
for all $y\in\bR^{m}$, $F|_{(\bR^{n}\times\{y\})\cap g(E)}$ is $\Lip$-bi-Lipschitz.
\end{enumerate}


\label{t:main-verson-1-restated}
\end{theorem}

%
%

Before we prove the Theorem \ref{t:main-verson-1-restated}, we show how it will be used.

\begin{lemma}
For $N\in\bN$, $\ve>0$, and $0<\delta<\frac12 \cH^{n,m}_\infty(f, Q_0)$, 
there a cube $Q_1\subseteq Q_0$ with 
$\side(Q_1)\gtrsim 2^{-K}\side(Q_{0})$ 
such that 
$K=K(\ve,\delta,N)$,
$\tilde{\beta}(3Q_1^N)<\epsilon$ 
and 
\begin{equation}
\frac{H^{n}_\infty(f(Q_1))}{\side(Q_{1})^{n}} \geq\delta .
\label{e:quotient-bigger-than-delta}
\end{equation}
\label{l:bt-small-sigma-large}
\end{lemma}

\begin{proof}
Let $\ve>0$ and $0<\delta<\frac12 \cH^{n,m}_\infty(f, Q_0)$ be given.

Suppose that $K$ is an integer such that for $k\leq K$, all cubes of sidelength at least $2^{-k}$ do not satisfy either
$\bt(3Q^N)<\epsilon$ or $H^{n}(f(Q))>\delta$ (or both).

Let
\[I_{k}^{1}=\{Q:\side(Q)=2^{-k},\bt(3Q^N)\geq \epsilon\}\]
and
\[I_{k}^{2}=\ck{Q:\side(Q)=2^{-k},\bt(3Q^N)<\epsilon, \frac{\cH^{n}_{\infty}(f(Q))}{\side(Q)^{n}}<\delta}\,.\]
Let $U_{k}^{j}$ denote the union of the cubes in $I_{k}^{j}$ for $j=1,2$. By Theorem \ref{t:TST},

\begin{align*}
1 
& =|Q_{0}|=\frac{1}{K}\sum_{k=1}^{K}\sum_{\side(Q)=2^{-k}}|Q| \\
&  =\frac{1}{K}\sum_{k=1}^{K}\left(\sum_{Q\in I_{k}^{1}}|Q|+\sum_{Q\in I_{k}^{2}}|Q|\right)\\
& \leq \frac{1}{K}\sum_{k=1}^{K}\left(\frac{1}{\epsilon^{2}}\sum_{\side(Q)=2^{-k}}\bt(3Q^N)^{2}|Q|+|U_{k}^{2}|\right)\\
& \leq \frac{C}{K\epsilon^{2}}+\frac{1}{K}\sum_{k=1}^{K}|U_{k}^{2}|
\end{align*}
where $C$ above depends only on $n$ and $N$. 

For $K>\frac{C\sqrt{n}}{\delta\cdot \ve^2}$, we have,
by the Pigeonhole principle, that  for some $k\leq K$, $|U_{k}^{2}|>1-\frac{\delta}{\sqrt{n}}$. 
Thus, for this $k$,  $|U_{k}^{1}|<\frac{\delta}{\sqrt{n}}$. 
Hence, we have

\begin{align*}
\cH^{n,m}_\infty(f, Q_0) 
& \leq \left(\sum_{Q\in I_{k}^{1}}+\sum_{Q\in I_{k}^{2}}\right)\cH_{\infty}^{n}(f(Q))\side(Q)^{m}
 \leq \sum_{Q\in I_{k}^{1}}\sqrt{n}|Q|+\sum_{Q\in I_{k}^{2}}\delta|Q|
\\
& \leq \sqrt{n}|U_{k}^{1}|+\delta|U_{k}^{2}|\leq 2\delta<
\cH^{n,m}_\infty(f, Q_0).
\end{align*}
Above we used the fact that 
\[\frac{\cH^{n}_{\infty}(f(Q))}{\side(Q)^{n}}\leq \frac{(\diam f(Q))^{n}}{\side(Q)^{n}} \leq \sqrt{n},\] which follows from 1-Lipschitzness of $f$, hence 
$$\cH^{n}_{\infty}(f(Q))\side(Q)^{m}\leq \sqrt{n}\cdot\side(Q)^{n+m}\,.$$ 
This is  a contradiction, and so the lemma is proved.
\end{proof}

\begin{proof}[Proof of Theorem~{{I}}]
Theorem~{{I}} now follows 
by  choosing a cube $Q_1$ guarenteed by  
Lemma \ref{l:bt-small-sigma-large} and then applying 
Theorem 
\ref{t:main-verson-1-restated} to $Q_1$.
\end{proof}

\subsection{Technical lemmas for the proof of Theorem \ref{t:main-verson-1-restated}}\label{ss:tech-lemmata}
We first need state and prove  a sequence of lemmas. An upper bound for the choices of the constant $\ve_\beta$ will come out of these lemmas.  Other constants will appear and be determined as well. In order to simplify equations,  we replace $\ve_\beta$ with $\ve$ in this section. 

By the Kuratowski embedding theorem, we may assume $\cM\subseteq\ell^{\infty}$. We again define $\tilde{f}(x)=(f(x),x)$, where its range space $\ell^{\infty}\oplus \bR^{n}$ is equipped with the norm
\begin{equation}
||u\oplus v||=\sqrt{|u|_{\infty}^{2}+|v|_{2}^{2}}
\label{e:product-metric}
\end{equation}
and we define $\bt=\bt_{\tilde{f}}$ using this metric.

\begin{lemma}
Let $0<\alpha<1$ and $\ve'>0$. If $\ve=\ve(\alpha,\ve')>0$ is sufficiently small, depending on $\alpha$, then for any cube $Q$ such that $\bt(3Q)<\ve$ and any line $L$ intersecting $Q$ 
, then there is $\sigma_0(L)\geq 0$ such that 
\[\av{\frac{|f(x)-f(y)|}{|x-y|}-\sigma_0(L)}<\ve'\] 
for all $x,y\in L\cap Q$ such that $|x-y|>\alpha\diam Q$
\label{l:almost-isometric-on-lines}
\end{lemma}
\begin{proof}
Let $x,y,z\in L\cap Q$ be $\alpha|x-y|$-separated, and let
\[
\begin{array}{lll}
a=|x-y|, & b=|y-z|, & c=|x-z|, \\
A=|f(x)-f(y)|, & B=|f(y)-f(z)|, & C=|f(x)-f(z)|.
\end{array}\]

Let $\tau>0$. Since $\bt(3Q)<\ve$, and since $\tilde{f}$ is Lipschitz, we have that for all such colinear $x,y,z\in Q$, if $\ve>0$ is small enough
\[||\tilde{f}(x)-\tilde{f}(y)||+||\tilde{f}(y)-\tilde{f}(z)||<||\tilde{f}(x)-\tilde{f}(z)||+\tau^{2}\ell(Q)\]
for some constant $C_{0}>0$. Without loss of generality, assume $\ell(Q)=1$. Then

\[\sqrt{a^{2}+A^{2}}+\sqrt{b^{2}+B^{2}}<\sqrt{c^{2}+C^{2}}+\tau^{2}.\]
Squaring both sides,
\begin{multline*}
a^{2}+A^{2}+b^{2}+B^{2}+2\sqrt{a^{2}+A^{2}}\sqrt{b^{2}+B^{2}}
<c^{2}+C^{2}+2\tau^{2} \sqrt{c^{2}+C^{2}}+\tau^4\\
\leq a^{2}+b^{2}+2ab+A^{2}+B^{2}+2AB+2\tau^{2}\sqrt{c^{2}+C^{2}}+\tau^{4}
\end{multline*}
where we have used the fact that $c\leq a+b$ and $C\leq A+B$. Canceling terms and dividing by two, we obtain
\[\sqrt{a^{2}+A^{2}}\sqrt{b^{2}+B^{2}}<ab+AB+\tau^{2}\sqrt{c^{2}+C^{2}}+\tau^{4}.\]
Choose $\tau<\alpha^{\frac{1}{2}}$. Since $c>\alpha$ and $C\leq c$, we have
\[\sqrt{a^{2}+A^{2}}\sqrt{b^{2}+B^{2}}<ab+AB+(\sqrt{2}+1)\tau^{2} c<ab+AB+3\tau^{2} c.\]
Squaring both sides,
\[a^{2}b^{2}+A^{2}B^{2}+a^{2}B^{2}+A^{2}b^{2}<a^{2}b^{2}+A^{2}B^{2}+9\tau^{4}c^{2}+6\tau^{2} c(ab+AB)+2abAB.\]
Canceling terms and subtracting $2abAB$ from both sides gives
\[(aB-Ab)^{2}<9\tau^{4}c^{2}+6\tau^{2} c(ab+AB)\leq 9\tau^{4}c^{2}+12\tau^{2}abc.\]
Note that either $a$ or $b$ must be at least $c/2$, let's say it is $b$. Divide both sides by $(ab)^{2}$, and get
\[
\left(\frac{B}{b}-\frac{A}{a}\right)^{2}\leq \frac{9\tau^{4}c^{2}+12\tau^{2}cab}{a^{2}b^{2}}\leq \frac{36\tau^{4}}{a^{2}}+\frac{24\tau^{2}}{a}<\frac{60\tau^{2}}{a}.
\]

Since $a\geq \alpha$, taking square roots of both sides gives
\[\av{\frac{B}{b}-\frac{A}{a}}<\frac{8\tau}{\alpha^{\frac{1}{2}}},\]
or in other words,
\[\av{\frac{|f(y)-f(z)|}{|y-z|}-\frac{|f(x)-f(y)|}{|x-y|}}<\frac{8\tau}{\alpha^{\frac{1}{2}}}.\]

which for $\tau$ small enough (depending on $\alpha$ and $\ve'$) implies the lemma.

\end{proof}

For $N\in\bN$ and for $x,y\in 3Q$ with $|x-y|\geq \alpha\ell(Q)$, define
\begin{eqnarray*}
\sigma(x,y)=\sigma(x,y, 3Q^N,  \frac13 2^{-N}\alpha)&=&\\
\inf_{x',y'\in L_{x,y}\cap 3Q^{N}\atop |x'-y'|\geq \alpha\ell (Q)}&& \frac{|f(x')-f(y')|}{|x'-y'|}
\end{eqnarray*}
where $L_{x',y'}$ is the line passing through $x'$ and $y'$. 
For a line $L$ intersecting $Q$, take $x,y\in L\cap 3Q$, $|x-y|\geq \alpha\ell(Q)$ and 
let 
$$\sigma(L):=\sigma(x,y, 3Q^N, \frac13 2^{-N}\alpha)\,.$$ 

By Lemma \ref{l:almost-isometric-on-lines}, if we assume that $\bt(Q)<\ve=\ve(\frac13 2^{-N}\alpha, \ve')$ then
for any $x,y\in L\cap 3Q^{N}$ satisfying $|x-y|\geq \alpha\ell(Q)$, we have
\begin{equation}
\sigma(x,y)\leq \frac{|f(x)-f(y)|}{|x-y|}\leq  \sigma(x,y)+\ve'.
\label{e:quotient-close-to-sigma}
\end{equation}

\begin{lemma}
Suppose $\bt(3Q^{N})<\ve$. If $\ve=\ve>0$ is small enough (depending on $\alpha$, $\ve'$ and $n+m$), then for $L$ and $L'$ parallel and intersecting $Q$,
\[\sigma(L)\leq \sigma(L')+\frac{2\sqrt{n+m}}{2^{N}}+\ve'.\]
\label{l:shifting-lines}
\end{lemma}

\begin{proof}
Let $x,y\in L$ and $x',y'\in L'$ be their orthogonal projections onto $L'$. Then
\begin{align*}
\sigma(L) & \leq \frac{|f(x)-f(y)|}{|x-y|} 
 \leq \frac{|f(x')-f(y')|+|f(x)-f(x')|+|f(y)-f(y')|}{|x-y|}\\
& \leq \sigma(L')+\ve' +\frac{2\diam Q}{|x-y|}
\end{align*}
since $|x-y|=|x'-y'|$ and the distance between $L$ and $L'$ is at most $\diam Q$ and using \eqn{quotient-close-to-sigma}. 
Picking $x,y\in 3Q^{N}\cap L$ so that $|x-y|=2^{N}\ell(Q)$ finishes the proof.
\end{proof}

\begin{corollary}
For any $\rho>0$, there is $N=N(\rho,n+m)$ and $\ve=\ve(\alpha,\rho)$ such that if $\bt(3Q^{N})<\ve$, then for all parallel lines $L$ and $L'$ intersecting $Q$,
\begin{equation}
\sigma(L)\leq \sigma(L')+\rho
\label{e:shift-sigma}
\end{equation}
and if $x,y\in Q$, $|x-y|\geq \alpha\ell(Q)$, and $|z|\leq \ell(Q)$,
\begin{equation}
\bigg| |f(x+z)-f(y+z)| - |f(x)-f(y)|\bigg|<\rho\ell(Q).
\label{e:shift-f}
\end{equation}
Moreover, if $\rho$ is small enough, depending on $\alpha$, and $x_{Q}=f(x_{Q})=0$, then for every $x,y\in Q$,
\begin{equation}
|f(x+y)|\leq |f(x)|+|f(y)|+\alpha\ell(Q)
\label{e:t-ineq}
\end{equation}
\label{c:shifting-lines}
\end{corollary}

\begin{proof}
Let $\rho'>0$ and choose $\ve'$ and $N$ so that 
\[\frac{2\sqrt{n+m}}{2^{N}}+\ve'<\rho'\]
and hence
\[\sigma(L)\leq \sigma(L')+\rho'\]
for $L$ and $L'$ parallel and intersecting $Q$. 

For $x,y\in Q$ and $|z|\leq \ell(Q)$,
\begin{align*}
|f(x)-f(y)|
& = |x-y|\frac{|f(x)-f(y)|}{|x-y|}\\
& \stackrel{\eqn{quotient-close-to-sigma}}{\leq} |x-y|\sigma(L_{x,y})+\rho'|x-y|  \\
& \leq |x-y|\sigma(L_{x+z,y+z})+2\rho'|x-y| \\
& \stackrel{\eqn{quotient-close-to-sigma}}{\leq} |x-y|\frac{|f(x+z)-f(y+z))|}{|(x+z)-(y+z)|}+2\rho'|x-y| \\
& =|f(x+z)-f(y+z))|+2\rho'|x-y|
\end{align*}
Similarly,
\[|f(x+z)-f(y+z)|\leq |f(x)-f(y)|+2\rho'|x-y|.\]
Suppose now that $x_{Q}=f(x_{Q})=0$. Then,
\[|f(-x)-f(0)|
\stackrel{\eqn{quotient-close-to-sigma}}{\leq} |x| \sigma(L_{-x,0})+\rho|x|
=|x|\sigma(L_{x,0})+\rho|x|
\stackrel{\eqn{quotient-close-to-sigma}}{\leq}  |f(x)-f(0)|\]
Thus,
\begin{align*}
|f(x+y)|=
& |f(x+y)-f(0)|
\leq |f(x)-f(-y)|+2\rho'|x+y|\\
& \leq |f(x)-f(0)|+|f(0)-f(-y)|+2\rho'|x+y|\\
& \leq |f(x)-f(0)|+|f(0)-f(y)|+\rho'(2|x+y|+|y|)\\
& =|f(x)|+|f(y)|+\rho'(2|x+y|+|y|).
\end{align*}
Using all the above inequalities and picking $\rho'<\frac{\rho}{5\sqrt{n}}$ (noting that $x\in Q$ implies $|x|\leq \sqrt{n}$) implies the corollary.

\end{proof}

Let 
\[\sigma(Q)=\inf_{L\cap Q\neq\emptyset} \sigma(L, 3Q^N, \frac13 2^{-N} \alpha).\]

\begin{lemma}\label{l:lemma-03}
With parameters as above,  for all $x,y\in Q$ with $|x-y|\geq \alpha \ell(Q)$,
\[\frac12\sigma(Q)|x-y| <|f(x)-f(y)|\leq |x-y|\]
\label{l:parallel-lines}
as long as $\rho<\sigma/2$.

\end{lemma}

\begin{proof}
The second inequality is just a restatement that $f$ is Lipschitz. For the other inequality, we observe that for $|x-y|>\alpha\ell(Q)$,
\[|f(x)-f(y)|\geq \sigma(x,y)|x-y|\]
and hence 
\[|f(x)-f(y)|\geq (\sigma(Q)-\rho)|x-y|\]
for all $x,y\in Q$.
\end{proof}

\begin{remark}
In what follows below, although we will continuously alter $\alpha,\rho,$ and $\ve'$ to be small as needed, we will consistently pick $\ve'<\rho<\alpha\ll \delta$.
\end{remark}

Form an orthonormal basis in 
$\bR^{n+m}$ inductively by picking 
\[u_{j} \in \Span\{u_{1},...,u_{j-1}\}^{\perp}\]
to be the vector minimizing
$\sigma(L_{u_j})$,
where $L_{u_j}$ is the line going through $x_Q$ in the direction $u_j$.
Let $v_j=u_{n+m-j+1}$ be a reverse order of the basis vectors.
Denote $d_j=\sigma(L_{v_j})$.
\begin{lemma}\label{l:d_n>delta}
Let $Q\subseteq\bR^{n+m}$ be such that $\cH^{n}_{\infty}(f(Q))\geq \delta\side(Q)^n$ and $\bt(3Q^N)\leq \epsilon$. 
If $\alpha<\frac{\delta}{C(n)}$, then 
\begin{equation}
d_{n}\gec _{n}\delta.
\label{e:d_n>delta}
\end{equation}
\end{lemma}

\begin{proof}
Without loss of generality, assume $\side(Q)=1$.
Also, assume $x_{Q}=f(x_{Q})=0$ and $\ell(Q)=1$. Let $V_{j}=\Span\{v_{1},...,v_{j}\}$. Then for any $x\in Q$, by Lemma \ref{l:shifting-lines},

Note that $P_{V_{j-1}}(x)=\sum_{i=n+m-j}^{n+m}a_{i}u_{i}$ with $|a_{i}|\lec_{n,m}1$, so that (assuming $\rho<\alpha$)
\begin{align*}
|f(x)-f(P_{V_{j-1}}(x))| 
& \stackrel{\eqn{shift-f}}{\leq}
|f(P_{V_{j-1}\perp}(x))|+\rho 
 = \av{f\ps{\sum_{i=n+m-j}^{n+m} a_{i}u_{i}}}+\alpha +\rho\\
& \stackrel{\eqn{t-ineq}}{\leq} \sum_{i=n+m-j}^{n+m}|f(a_{i}u_{i})|+(j+2)\alpha 
 \leq C_{n}(d_{j}+\alpha)
\end{align*}
so that
\[f(Q)\subseteq \{y:\dist(y,f(V_{j-1}\cap Q))\leq C_{n}(d_{j}+\alpha)\}.\]
In particular, 
\begin{equation}
f(Q)\subseteq \{y:\dist(y,f(V_{n-1}\cap B^{Q}))<C_{n}(d_{n}+\alpha)\}.
\label{e:V_n-1}
\end{equation}

Let $\tau=2C_{n}(d_{n}+\alpha)$ and let $\{x_{j}\}_{j=1}^{k}$ be a maximal $\tau$-net in $V_{n-1}$, so $k\lec_{n}\tau^{-n+1}$. Since $f$ is $1$-Lipschitz, the balls $B_{j}=B(f(x_{j}),\tau))$ cover $f(V_{n-1})$, and by \eqn{V_n-1}, their doubles also cover $f(Q)$, thus
\[\delta<\cH^{n}_{\infty}f(Q)\leq \sum_1^k (2\diam B_{j})^{n}\lec_{n} \tau\lec_{n}d_{n}+\alpha.\]
Choosing $\alpha$ to be much smaller than $\delta$ gives $d_{n}\gec_{n} \delta$.

\end{proof}

Without loss of generality, assume $v_{1},...,v_{n}$ are the standard basis vectors, so $V_{n}=\bR^{n}$ (since we may restrict $f$ to a slightly smaller rotated cube in $Q$ so that this is the case). Define $h:\bR^{n+m}=\bR^{n}\times \bR^{m}\rightarrow \ell^{\infty}\times \bR^{m}$ by 
\[h(x,y)=(f(x,y),y).\]

We will soon switch to discussing
$\bt, \sigma$, etc. all with respect to the function $h$, rather than the function $f$.  We will first need a couple of lemmas. 

\begin{lemma}
$\bt_{\tilde{h}}\leq 3\bt_{\tilde{f}}$.
\end{lemma}

\begin{proof}
Note that $\tilde{h}(x,y)=(f(x,y),y,x,y)$, and 
\begin{multline*}
\d_{1}^{\tilde{h}}((x_{1},y_{1}),(x_{2},y_{2}),(x_{3},y_{3}))\\
=\sum_{i=1}^{2}\sqrt{|f(x_{i},y_{i})-f(x_{i+1},y_{i+1})|^{2}+|x_{i}-x_{i+1}|^{2}+2|y_{i}-y_{i+1}|^{2}|}\\
-\sqrt{|f(x_{1},y_{1})-f(x_{3},y_{3})|^{2}+|x_{1}-x_{3}|^{2}+2|y_{1}-y_{3}|^{2}|}\\
=A+B-C
\end{multline*}
for $(x_{i},y_{i})\in \bR^{n+m}$ colinear. Let
\[a=\sqrt{2|f(x_{1},y_{1})-f(x_{2},y_{2})|^{2}+2|x_{1}-x_{2}|^{2}+|y_{1}-y_{2}|^{2}},\]
\[b=\sqrt{2|f(x_{2},y_{2})-f(x_{3},y_{3})|^{2}+2|x_{2}-x_{3}|^{2}+|y_{2}-y_{3}|^{2}},\]
\[c=\sqrt{2|f(x_{1},y_{1})-f(x_{3},y_{3})|^{2}+2|x_{1}-x_{3}|^{2}+|y_{1}-y_{3}|^{2}}.\]
Using these values in \eqn{abcABC}, and the fact that $C\leq \sqrt{2}c$, we get 
\[\frac{1}{\sqrt{3}}\d_{1}^{\tilde{h}}\leq \sqrt{3}\d_{1}^{\tilde{f}}.\]
\end{proof}

%

\begin{lemma}\label{l:sigma_h}
Assume  $\bt(3Q^N)\leq \epsilon$. Then
$$d_n\lesssim \sigma_{h}(Q)$$ 

\end{lemma}
\begin{proof}
Without loss of generality assume $\side(Q)=1$. 
We think of $\bR^{n+m}$ as $\bR^n\oplus\bR^m$.
Let $L$ be a line with $L\cap Q\neq \emptyset$.
First note,  that 
if $L\subset \bR^n$, then 
$$\sigma_h(L)\geq \sigma_f(L)\geq d_n\stackrel{\eqn{d_n>delta}}{\gec} \delta\,.$$
and if $L\subset \bR^m$, then 
$$\sigma(L)=1\,.$$

Now suppose that $L=\textrm{Span}\{v\}$, where $v$ is a unit vector, $v=v_1+v_2$ with $v_1\in \bR^n$ and
$v_2\in \bR^m$.  We may assume  $v_1\neq 0$ by the above. 
Take $(x,y),\ (z,w)\in L\cap 3Q^N$ of mutual distance $\geq \alpha2^{-N}$.
We distinguish between two cases.
\begin{itemize}
\item[\bf Case 1:] $|v_2|\leq \frac12 d_n|v_1|$. For $\ve'$ and $\rho$ small enough (depending on $\delta\lec d_{n}$),

\begin{multline*}
|h(x,y)-h(z,w)|  
\geq |f(x,y) - f(z,w)|\geq |f(x,y) - f(z,y)| - |f(z,y) - f(z,w)|\\
\stackrel{\eqn{shift-f}}{\geq} |f(x_Q) - f(x_Q+(z,y)-(x,y))| - \rho|z-x| - |(z,y) - (z,w)| \\
\geq |f(x_Q) - f(x_Q+(z,y)-(x,y))| - (\rho+\frac12 d_n) |z-x|  \\
\stackrel{\eqn{quotient-close-to-sigma}}{\geq}  |z-x|\bigg(d_n - \ve'\bigg) - (\rho+\frac12 d_n) |z-x|   \\
\gec |z-x|d_n 
\gec |(x,y)-(z,w)|d_n
\end{multline*}
%
 
We have used the fact that, because of our case assumption and since $|(x,y)-(z,w)|\geq\alpha 2^{-N}$, we have $|z-x|\geq \frac{1}{3}\alpha 2^{-N}$, which permits us to apply \eqn{quotient-close-to-sigma}.
 
\item[\bf Case 2:] $|v_2|\geq \frac12 d_n|v_1|$. Then

\[|h(x,y)-h(z,w)|\geq |y-w|\geq \frac{d_n}{2} |(x,y) - (z,w)|.\]

 \end{itemize}

In any case,
$$\sigma_h(L, 3Q^N,  2^{-N} \alpha)\gec d_n\,.$$

This almost gives us 
$\sigma_{h}(Q)\gec d_{n}$, except for the missing factor of $\frac13$ in the definition of 
$\sigma_{h}(Q)$. By reducing $\alpha$, this is of no consequence.


\end{proof}

From here on, we will be concerned with $\bt, \sigma$, etc. all with respect to the function $h$, rather than the function $f$. 

Recall that $B_{Q}$ is the largest ball contained in $Q$.

\begin{lemma}
{\JARSbluJARS     With $h,Q,$ and $\delta$ as above,} there is $\alpha=\alpha(\delta,n,m)>0$, 
(and hence $N$ and $\ve$ depending on $\alpha$) so that if $\bt_{\tilde{h}}(3Q^N)<\ve$ then
\begin{equation}
\upsilon:=\cH^{n+m}_{\infty}(h(Q))\gec_{n,m}\delta |B_{Q}|.
\label{e:projection-substitute}
\end{equation}
\label{l:projection-substitute}
\end{lemma}
We note that the conditions imply the existence of $\ve'$ and $\rho$ as in the previous lemmas of this section.

{\JARSbluJARS    
\begin{proof}
Assume that $\side(Q)=1$. Restrict the size of $\alpha$ so that $\alpha< \frac{\delta}{C\sqrt{n+m}}$ for some large constant $C$ to be specified later. Let $Z'$ be a maximal $2\alpha$-net for $\d B_{Q}$ and extend it to a maximal net $Z$ for $B_{Q}$. By Lemmas \ref{l:d_n>delta} and \ref{l:sigma_h}, the map $H=h^{-1}$ is an $L$-bi-Lipschitz on $h(Z)$ with $L\lec \frac{1}{d_{n}}\lec\frac{1}{\delta}$. Extend $H=(H_{1},...,H_{n+m})$ to a $\sqrt{n+m}\cdot L$-Lipschitz map of $X$ into $\bR^{n+m}$ by extending each each coordinate function $H_{j}$ to obtain $L$-Lipschitz maps $H_{j}:X\rightarrow \bR$ for $j=1,...,n+m$ (c.f. \cite[p. 43-44]{Heinonen}). By picking $C$ large enough, we can guarantee that $H(h(\d B_{Q}))\subseteq (\d B_{Q})_{\frac{1}{4}}$. 

\Claim: $H(h( B(Q)))\supseteq \frac{1}{2} B_{Q}$. 

Suppose there was $x_{0}\in \frac{1}{2}B_{Q}\backslash H\circ h(B_{Q})$. Let $\psi$ be the map that takes a point $x\in \bR^{n}\backslash x_{0}$ to the unique point on $\d B_{Q}$ lying on the line passing through $x$ and $x_{0}$.  Then $\psi\circ H\circ h:\d B_{Q}\rightarrow \d B_{Q}$ is a continuous map that is homotopic to the identity map. However, $\psi\circ H \circ h:B_{Q}\rightarrow \d B_{Q}$ is also continuous, but as $\d B_{Q}$ has trivial $(n+m)^\textrm{th}$ singular homology, $\psi\circ H\circ h|_{B_{Q}}$ is homotopic to the constant map, and hence so is $\psi\circ H\circ h|_{\d B_{Q}}$, which is a contradiction.

Thus, since $H$ is $\frac{C}{\delta}$-Lipschitz for some constant $C$,  we get
\[\cH_{\infty}^{n+m}(h(Q))\gec 
\delta\cH^{n+m}(H\circ h(Q))\geq 
\delta |\frac{1}{2}B_{Q}|=\frac{\delta}{2^{n+m}}|B_{Q}|.\]

\end{proof}}

Apply Theorem \ref{t:lip-bilip} to  to the function $h$, with 
$$\kappa:=\frac{\cH_{\infty}^{n+m}(h(Q))}{2c_1\cH^{n+m}(Q)}\sim \delta$$
where $c_{1}$ is as in Theorem \ref{t:lip-bilip}. {\JARSbluJARS     Denote by $E$ the set $E_i$ of largest measure. Without loss of generality, we will assume $h(E)$
is Borel in $\cM\times \bR^{m}$,   for otherwise we could replace $E$ with a set $E'=h^{-1}(F)\cap E$, where $F$ is a Borel set such that $\cH^{n+m}(h(E))=\cH^{n+m}(F)$, and hence $\cH^{n+m}(E)=\cH^{n+m}(E')$. We do this in order to ensure that $E$ is $\cH^{n}_{\cM}\times \cH_{\bR^{m}}^{m}$-measurable and to permit us to use Fubini's theorem freely in the proof of the next lemma.}

\begin{lemma}\label{l:lemma-0.6}
Let $E$ be as above
so $h|_{E}$ is bi-Lipschitz and $|E|\gec_{\kappa} 1$. For $y\in \bR^{m}$, let $E_{y}=([0,1]^{n}\times\{y\})\cap E$. There is $y'\in\bR^{m}$ such that

{\JARSbluJARS     \[ \int \cH^{n}(f(E_{y'})\cap f(E_{z}))d\cH^{m}(z) \gec_{\kappa} 1\]}

\label{rf:115}

\end{lemma}

\begin{proof}
The proof is  an application of Fubini's theorem. 
Let $P_{\cM}:\cM\times [0,1]^{m}$ denote the map $(x,y)\mapsto x$
and
let $P_{\bR^m}:\cM\times [0,1]^{m}$ denote the map $(x,y)\mapsto y$.
To avoid confusion in the string of equations below, we will label the Hausdorff measures coresponding to their domains, e.g. $d\cH_{\bR^m}(x)$.

\begin{eqnarray}
&&\int\int\cH_{\cM}^{n}(f(E_{y} )\cap f(E_{z}))d\cH_{\bR^m}^{m}(y)d\cH_{\bR^m}^{m}(z)\\ 
&=&\int\ps{\int\int \one_{f(E_{y})}(x)\one_{f(E_{z})}(x)d\cH^{m}_{\bR^m}(y)d\cH_{\bR^m}^{m}(z)}d\cH^{n}_\cM(x) \\
&=&\int_{f(E)} \cH_{\bR^{m}}^m(P_\cM^{-1}(x)\cap h(E))^{2}d \cH_{\cM}^{n}(x)\\
 &\geq& \frac{1}{\cH_{\cM}^n(f(E))}\ps{\int \cH_{\bR^{m}}^m(P_{\cM}^{-1}(x)\cap h(E))d\cH^{n}_{\cM}(x)}^{2} \\ 
&=&\frac{1}{\cH_{\cM}^n(f(E))}\ps{\cH_{\cM}^{n}\times \cH_{\bR^{m}}^{m}(h(E))}^2\\
&=&\frac{1}{\cH_{\cM}^n(f(E))}\ps{\int_{\bR^{m}}\cH_{\cM}^{n}(P_{\bR^m}^{-1}(y)\cap h(E)) d\cH_{\bR^{m}}^{m}(y)}^2\\
&=&\frac{1}{\cH_{\cM}^n(f(E))}\ps{\int_{\bR^{m}}\cH_{\cM}^{n}(h(P_{\bR^m}^{-1}(y)\cap E))d\cH_{\bR^{m}}^{m}(y)}^2 \\
&\gec& \frac{1}{\cH_{\cM}^n(f(E))}\ps{\int_{\bR^{m}} \cH_{\cM}^{n}(E_{y})d\cH^{m}_{\bR^m}(y) }^2\label{e:pppppp}\\
 &\geq &\frac{\cH^{n+m}(E)^{2}}{\cH_{\cM}^{n}(f(E))} \\
  &\gec_{\kappa} &1 \label{e:last}
 \label{e:integrals}
\end{eqnarray}
The inequality in \eqn{last} follows since 
$ \cH_{\cM}^{n}(f(E)) \leq \cH_{\cM}^{n}(f([0,1]^{n+m}))\leq 1$, where the latter inequality is assumed in the statement of Theorem~{{I}}; \eqref{e:pppppp} follows from $h$ being $\kappa$-bi-Lipschitz and \eqn{c_1}.

{\JARSbluJARS     Note that in the above computations it would have been tempting to use $\cH_{\cM\times \bR^{m}}^{n+m}$ (with respect to the metric \eqn{product-metric}) instead of $\cH_{\cM}^{n}\times \cH_{\bR^{m}}^{m}$. However, these two measures aren't usually equal. In fact, the most we know is that for $A\subseteq \cM$ and $B\subseteq \bR^{m}$ measurable, we have $\cH_{\cM\times \bR^{m}}^{n+m}(A\times B)\lec \cH^{n}_{\cM}(A)\cH_{\bR^{m}}^{m}(B)$ (see \cite[Theorem 3.2.23]{Federer}). For this reason, we sidestep this issue by working solely with $\cH_{\cM}^{n}\times \cH_{\bR^{m}}^{m}$, which won't hurt our computations at all. Of course, in the last line above, $\cH^{n+m}(E)=\cH^{n}\times \cH^{m}(E)$ since $E$ is a subset of Euclidean space.}

\end{proof}

\subsection{Back to the proof of Theorem \ref{t:main-verson-1-restated}}

Let 
\[E':=\bigcup_{z\in[0,1]^{m}} E_{z}\cap f^{-1}(f(E_{y'}))=E\cap f^{-1}(f(E_{y'}))\]
Let $Y=[0,1]^{n}\times \{y'\}$, and let $f|_{Y}:Y\rightarrow f(Y)$ denote the restriction of $f$ to $Y$. Note that it has a bi-Lipschitz inverse on $f(E_{y'})$. This is easy to check since $h(x,y)=(f(x,y),y)$ is bi-Lipschitz on $E$, and in particular, on $E_{y'}$. Define a map $p:\cM\rightarrow\bR^{n}$ by
\[p(z)= P_{\bR^{n}}\circ (f|_{Y})^{-1}(z).\]
Define $g:E'\rightarrow \bR^{n+m}$ by
\[g(x,y)=(p\circ f(x,y),y).\]
Note that $p$ is bi-Lipschitz on $f(E_{y'})$, and hence $g$ is bi-Lipschitz on $E'$.

Now, extend $g$ to all of $Q$ using the Whitney extension theorem and apply Theorem~{{II}} to obtain sets $E_{1},...,E_{M}\subseteq Q$ such that  $g|_{E_{i}}$ is bi-Lipschitz for each $i$ and
\[\cH^{n+m}_\infty\bigg(g\ps{Q\backslash \bigcup E_{i}}\bigg)<\frac{\cH^{n+m}(E')}{2}.\]
By the Pigeonhole principle and the fact that $g$ is bi-Lipschitz on $E'$, there is $i$ such that 
\begin{equation}\label{e:one-over-M}
\cH^{n+m}(E_{i}\cap E')\gec \frac{1}{M} \cH^{n+m}(g(E')). 
\end{equation}

Note that $M$ depends quantitatively on $\cH^{n+m}(E')\leq 1$, but now we will show that this value is bounded below by a constant depending on $\delta$ (or equivalently, $\kappa$).
Let $E''=E_{i}\cap E'$.

\medskip

{\JARSbluJARS    
\noindent \Claim: $\cH^{n+m}(E'')\gec_{\kappa} 1$.

\noindent Proof of claim:  As $g$ is bi-Lipschitz on $E'$,
and we have eq. \eqref{e:one-over-M}, it suffices to show that 
$\cH^{n+m}(g(E'))\gec_{\kappa} 1$.
For $z\in \bR^{m}$, 
\[
g(E_{z}')
=g(E_{z}\cap f^{-1}(f(E_{y'}))
=(p(f(E_{z})\cap f(E_{y'})))\times \{z\}\,.
\]
For $(x,y)\in E'$, if $P_{\bR^{m}}(x,y)=y$, then $P_{\bR^{m}}(g(x,y))=y$ and vice-versa. 
Recalling that 
$p$ is bi-Lipschitz on $f(E_{y'})$,
\begin{multline*}
\cH^{n}(g(E')_{z})=\cH^{n}(g(E_{z}'))
=\cH^{n}(p(f(E_{z})\cap f(E_{y'})))\times \{z\})\\
\sim \cH_{\cM}^{n}(f(E_{z})\cap f(E_{y'})).
\end{multline*}
Finally, by this and Lemma \ref{l:lemma-0.6}

\begin{multline*}
\cH^{n+m}(g(E'))
=\int\cH^{n}(g(E_{z}' ))d\cH^{m}(z)
= \int \cH^{n}_{\cM}(f(E_{z})\cap f(E_{y}))d\cH^{m}_{\bR^{m}}(z) \\
\gec_{\kappa}1
\end{multline*}
which gives the claim. 

}
\medskip

By Theorem~{{II}}, the function $g|_{E_{i}}$ has a bi-Lipschitz extension from $E_{i}\supset E''$ to all of $\bR^{n+m}$; we abuse notation and also call this extension $g$ from here on.
Now,  for $(x,y)\in g(E'')$, we may invert the definition of $g$ to get

$$
(x,y)=g\circ g^{-1}(x,y)
=
(p\circ f\circ g^{-1}(x,y),y).
$$
Hence,
\[x
=
p\circ f\circ g^{-1}(x,y)\]
and since $(f|_{Y})^{-1}(z)=(p(z),y')$ for $z\in f(E_{y'})$, we have
\[f(x,y')=f|_{Y}(x,y')=f|_{Y}(p\circ f\circ g^{-1}(x,y),y')=f\circ g^{-1}(x,y).\]

We now only need to notice two things about the left hand side (LHS) of the above equation.
\begin{itemize}
\item  (LHS)  is independent of $y$, and
\item (LHS)  is, for fixed $y$, bi-Lipschitz in $x$ as by construction $E''\subset E$ chosen in Lemma \ref{l:lemma-0.6}.
\end{itemize}
This proves the Theorem  \ref{t:main-verson-1-restated} with the set $E''$ renamed to $E$ and with $Q_{1}=Q$.
The constant $\Lip$ may be estimated by tracking the constants in the above proofs.

\subsection{Proof of Corollary \ref{c:lower-bd-content}}
\begin{proof}
Since $g$ is bi-Lipschitz with constant $\Lip$, it suffices to show that
$$\cH^{n,m}_\infty (F,E)> \xi'\,,$$
for some $\xi'>0$ depending only on $n,m,\delta$.
Consider $\tilde F(x,y): \bR^{n+m}\to \cM\times \bR^m$ given by $\tilde F(x,y)= \big(F(x,y),y\big)$.
The map $F$ is bi-Lipschitz with constant depending only on $\Lip$.
Thus 
$$\cH^{n+m}_\infty (\tilde F(E))\gec_\Lip \eta\,.$$
Now,
let $\{Q_i\}$ be a measure theoretic cover for $E$ such that
$$2\cH^{n,m}_\infty (F,E)> \sum_i  \cH^n_\infty(f(Q_i))\side(Q_i)^m\,,$$
and
let $2\cH^n_\infty(f(Q_i))\geq \sum_j \diam(B_{i,j})^n$ where 
$\cup_j B_{i,j}\supset Q_i$.
We now compute
\begin{multline*}
4\cH^{n,m}_\infty (F,E)
\geq 
2\sum_i  \cH^n_\infty(f(Q_j))\side(Q_j)^m
\geq 
\sum_{i,j}  \diam(B_{i,j})^n\side(Q_j)^m\\
\geq
\cH^{n+m}_\infty (\tilde F(E)) \gec_\Lip \eta
\end{multline*}
\end{proof}

%
%
%
%
%
\subsection{A quick note on Hausdorff content}\label{ss:Hausdorff-cont-note}
Let $f:\bR^{n+m}\to \bR^n$.
Recall that 
\[\cH^{n,m}_{\infty}(f,A)=\inf \sum_{j} \cH^{n}_{\infty}(f(Q_{j}))\side(Q)^{m}\,,\]
where the infimum is taken over all measure theoretic partitions of $A$ by $\{Q_j\}$.
We view this quantity as a {\it coarse} version of the Jacobian of a function.
Recall that  $J^n_f(x)$ is defined as
\begin{equation*}
J^n_f(x)=\det_{n}D_{x}f=\prod_{i=1}^{n}\sigma_{i}(D_{x}f). 
\end{equation*}
and that we have the following for all Lipschitz functions $f:\bR^{n+m}\to\bR^n$.
\begin{equation}
\int_{\bR^{n+m}} J^n_f(x) d\cH^{n+m}(x) = \int_{f(\bR^{n+m})} \cH^m(f^{-1}(z))d\cH^n(z)\,,
\end{equation}
c.f. \cite{Federer},Theorem 3.2.11, p. 248.

The following lemma holds.

\begin{lemma}\label{l:content-vs-jacobian}
If $f:\bR^{n+m}\rightarrow \bR^{n}$, then
\begin{equation*}
\cH^{n,m}_{\infty}(f,[0,1]^{n+m})\lec 
\min\ck{\cH^{n}_{\infty}\big(f([0,1]^{n+m})\big)\ \ \ ,\ \ \ \int_{[0,1]^{n+m}} J_{f}^{n}\ d\cH^{n+m}}.
\end{equation*}
\end{lemma}

In particular, this means that if one has a map $f:\bR^{n+m}\to \bR^n$ such that its rank at almost every point is $<n$, then  the content of this function will be $0$.

\begin{proof}[Proof of Lemma \ref{l:content-vs-jacobian}]
Indeed, by definition it follows that this is at most $\cH^{n}(f([0,1]^{n+m}))$. Let $\lambda\in(0,1)$, $T$ be the set of points $x\in[0,1]^{n+m}$ where $f$ is differentiable, $x$ is a Lebesgue point for $J_{f}$, and $T_{i}\subseteq T$ be those points where $\lambda^{i+1}\leq J_{f}^{n}(x)<\lambda^{i}$. Fix an $i$ for the moment and let $\ve_{i}>0$ to be fixed later. For each $x\in T_{i}$, let $Q_{x}$ be a maximal cube such that for all $Q\subseteq Q_{x}$ containing $x$,
\[J_{f}^{n}(x)<\frac{1}{|Q|}\int_{Q}J_{f}\ d\cH^{n+m}+\ve_{i}.\]
Let $Q_{x}'\ni x$ be a maximal cube in $Q_{x}$  so that for all $Q\subseteq Q_{x}'$ containing $x$,
\begin{equation}
\sup_{Q}\{|f-A_{Q}^{f}|<\ve_{i}\cdot \side(Q)
\label{e:f-A_Q_x}
\end{equation}
and $Q_{x}''\ni x$ a maximal cube in $Q_{x}'$ such that
\[|\det_{n}A_{Q_{x}''}-J_{f}^{n}(x)|<\ve_{i}.\]
The existence of these cubes follows from the definition of $T$.

Let $\cC$ be the collection of maximal cubes in the collection $\{Q_{x}'':x\in T\}$, and for each such $Q\in \cC$, $Q=Q_{x}$ for some $x\in T$, and we let $j_{Q}$ be such that $x\in T_{j_{Q}}$.  

For such a $Q$, the above equations are satisfied with $x$ and $Q$ and $i=j_{Q}$. By \eqn{f-A_Q_x}, 
\[f(Q)\subseteq \{y\in\bR^{n}:\dist(y,A_{Q}^{f}(Q))<\ve_{j_{Q}}\},\]
hence, for $\ve_{j}\ll \lambda^{j}$, 
\[\cH^{n}_{\infty}(f(Q))\leq 2\cH^{n}(A_{Q}^{f}(Q))\lec \lambda^{j_{Q}}\side(Q)^{n}.\]
We sum over all $Q\in\cC$ to get
\begin{align*}
\cH^{n,m}_{\infty}(f,[0,1]^{n+m}) 
& \leq \sum_{Q\in \cC} \cH^{n}_{\infty}(f(Q))\side(Q)^{m}\\
& \leq \sum_{j}\sum_{Q\in \cC\atop j_{Q}=j}\cH^{n}_{\infty}(f(Q))\side(Q)^{m}\\
& \lec \sum_{j}\sum_{Q\in \cC\atop j_{Q}=j} \lambda^{j}\side(Q)^{n+m}\\
& \leq\sum_{j}\sum_{Q\in \cC\atop j_{Q}=j}   \lambda^{-1} \int_{Q} J_{f}^{n}d\cH^{n+m}\\
& =\lambda^{-1}\int J_{f}^{n} d\cH^{n+m}
\end{align*}
as $\ve_{j}\rightarrow0$ for each $j$. Taking $\lambda\rightarrow 1$ completes the proof.
\end{proof}

\begin{remark}
The quantity $\cH^{n,m}_{\infty}(f,Q)$ can easily be estimated in some simple scenarios. For example, let $\Sigma$ is an $n$-Ahlfors-David regular metric space, meaning
\begin{equation}
\cH^{n}(B(x,r))\sim r^{n}, \;\;\; x\in \Sigma, r>0
\label{e:AD-regular}
\end{equation}
and suppose $f:\bR^{n+m}\rightarrow \Sigma$ is a regular mapping, meaning
\begin{equation}
 \cH^{n+m}(f(B(x,r)))\gec r^{n}.
 \label{e:f-regular}
 \end{equation}
Then
\[\cH^{n,m}_{\infty}(f,Q)\sim \side(Q)^{n+m}\]
with constants depending only on the constants in \eqn{AD-regular} and \eqn{f-regular} and the Lipschitz constant of $f$.
\end{remark}

We refer the reader to \cite{DS00-regular-mappings}
for an idea on how one may relate this to co-Lipschitz functions.

{\JARSbluJARS    
\begin{remark}
\label{r:open}
As mentioned in Remark \ref{r:can't}, we can't ask for a decomposition of our domain into sets as in Theorem \ref{t:lip-bilip} (that is, we can't hope find sets $E_{1},...,E_{M}$ satisfying Theorem~{{I}}  so that their images exhaust most of the image in the sense of equation \eqn{c_1}). However, one could instead ask the following. If $\kappa>0$, do there exist $M=M(m,n,\kappa)$, $E_{1},...,E_{M}$, and $g_{1},...,g_{M}$ so that (ii)-(iv) of Theorem~{{I}} are satisfied for the triple $(f,g_{i},E_{i})$ for each $i=1,...,M$, and 
\[\cH_{\infty}^{n,m}(f, [0,1]^{n+m}\backslash \bigcup E_{j})<\kappa  \;\;\; {\textrm ?}\]
We do not know whether this is true or not. 
\end{remark}
}

\section{Appendix: extensions between concentric affine mappings}

In this section, we build some functions that will help us extend $f$. 

\subsection{Interpolating rotations}

For $j=0,1$, let $A_{j}\in SO(D)$ and $p$ a function such that 
\begin{equation}
p=\sum_{j=0}^{1}\one_{B(0,2^{j})}A_{j}.
\label{e:two-rotations}
\end{equation}

\begin{lemma}
For $p$ as in eq. \eqn{two-rotations}, there is a bi-Lipschitz map from $B(0,1)$ to itself that agrees with $\rho$ on $\d B(0,1/2)\cup \d B(0,1)$.
\label{l:two-rotations}
\end{lemma}

To prove this, we will first need a simple fact about $SO(D)$ (and compact Lie groups in general):

\begin{lemma}
The set $SO(D)$ with its operator norm is quasiconvex: there is $C>0$ such that for each $U,V\in SO(D)$, there is a path $\gamma\subseteq SO(D)$ whose endpoints are $U$ and $V$, and whose length is at most $C|U-V|$.
\end{lemma}

\begin{proof}
The proof of this is rather simple: as $SO(D)$ is compact, it suffices to show that it is locally $C$-quasiconvex for some universal $C$. Since it is also a compact Lie group, we need only verify that it is $C$-quasiconvex on a neighborhood of the identity. Considering $SO(D)$ as a submanifold of $\bR^{D^{2}}$, this follows since we may find a neighborhood of $I\in SO(D)$ bi-Lipschitzly diffeomorphic to a ball in $\bR^{\frac{D^{2}-D}{2}}$ (with respect to the $\ell^{2}$-norm on the domain and range, which in $\bR^{D^{2}}$ is equivalent to the operator norm), and quasiconvexity is a bi-Lipschitz invariant. 
\end{proof}

Let $\gamma_{t}$ be a $C$-Lipschitz function from $[0,1]$ to $SO(D)$ such that $\gamma_{j}=A_{j}$. Define
\begin{equation*}
\Gamma_{A_{0},A_{1}}(x)= \gamma_{|x|}(x),\;\;\; x\in B(0,1)\backslash B(0,1/2).
\end{equation*}

We claim that this function is bi-Lipschitz. For $x,y\in B(0,1)\backslash B(0,1/2)$, since $\gamma_{t}$ is Lipschitz and gives an isometry for every $t$, we get

\begin{align*}
|\Gamma_{A_{0},A_{1}}(x)-\Gamma_{A_{0},A_{1}}(y)|
& \leq |\gamma_{|x|}(x)-\gamma_{|x|}(y)|+|\gamma_{|x|}(y)-\gamma_{|y|}(y)|\\
& = |x-y|+|(\gamma_{|x|}-\gamma_{|y|})(y)| \\
& \leq |x-y|+C||x|-|y||\cdot|y|\leq (1+C)|x-y|.
\end{align*}

This demonstrates that $\Gamma$ is Lipschitz. To show that it is also bi-Lipschitz, we simply note that
\[\Gamma_{A_{0},A_{1}}^{-1}=\gamma_{|x|}^{-1}(x)\]
and $\gamma_{t}^{-1}$ is also $C$-Lipschitz on $[0,1]$, since the inverse function on $SO(D)$ is Lipschitz.

\subsection{Interpolating between ellipses}

Let $A$ be a linear transformation with $|A|\leq 1$, $\sigma=\sigma(A)$, and $U\Sigma V$ its singular value decomposition, so that the diagonal entries $d_{1},...,d_{n}=\sigma$ of $\Sigma $ are in descending order. Let
\begin{equation}
p=\sigma \one_{\d B(0,1)}UV+\one_{\d B(0,2)}A.
\label{e:A-and-UV}
\end{equation}

\begin{lemma}
For $p$ as in eq. \eqn{A-and-UV}, there is a $C(\sigma)$-bi-Lipschitz map on $B(0,2)\backslash B(0,1)$ that agrees with $p$ on $\d B(0,1)\cup \d B(0,2)$.
\label{l:ellipse-prop-1}
\end{lemma}

Let 

\[\Lambda_{1,A}(x)=\kappa(UV(x),A(x),|x|-1)=U\circ \kappa(\sigma I,\Sigma ,|x|-1)\circ V(x)\]

where

\[
\kappa(x,y,t):=xt+(1-t)y, \;\;\; t\in[0,1].
\]

Showing bi-Lipschitzness of this map can be done in a way similar to lemma \ref{l:two-rotations} (the crucial detail is that $\Sigma \geq \sigma I$, which guarantees that we have an inverse which we may also show is Lipschitz). We omit the details.

Similarly we have the following:

\begin{lemma} 
Suppose $A$ is as in lemma \ref{l:ellipse-prop-1}, and $p$ satisfies
\[p=\one_{\d B(0,1)}A+\one_{\d B(0,2)}UV.\]
Then there is a $C(\sigma)$-bi-Lipschitz map $\Lambda_{2,A}$ on $B(0,2)\backslash B(0,1)$ that agrees with $p$ on $\d B(0,1)\cup \d B(0,2)$. 
\label{l:ellipse-prop-2}
\end{lemma}

\subsection{A repositioning map}

For a ball $B=B(x,r)$, let
\[T_{B}:B\rightarrow B(0,1),\;\;\; y\mapsto \frac{y-x}{r}\]
and for balls $B_{1}$, $B_{2}$, let
\[T_{B_{1},B_{2}}:B_{1}\rightarrow B_{2},\;\;\;  y\mapsto T_{B_{2}}^{-1}\circ T_{B_{1}}(y).\]

\begin{lemma}
For $B_{j}=B(x_{j},r_{j})\subseteq B(0,1)$, $j=1,2$, there is a $C$-bi-Lipschitz map $s_{B_{1},B_{2}}$ of $B(0,1)$ to itself that extends $T_{B_{1},B_{2}}$, where
\[C=\frac{\max\{r_{2}/r_{1},r_{1}/r_{2}\}}{\ps{1-\frac{|x_{1}|}{1-r_{1}}}\ps{1-\frac{|x_{2}|}{1-r_{2}}}}.\]
\end{lemma}

\begin{proof}
We first give $s_{B_{1},B_{2}}$ when $B_{1}=B(0,r_{2})$. Set

\[
s_{B(0,r_{2}),B_{2}}(x)=\isif{x+x_{2}\frac{1-|x|}{1-r_{2}}, & r_{2}<|x|<1 \\
x+x_{2} & |x|\leq r_{2}}.
\]

This map is $C_{B_{2}}$-bi-Lipschitz, where
\[C_{B_{2}}=\ps{1-\frac{|x_{2}|}{1-r_{2}}}^{-1}.\]

The proof of this is simple, and we will only show the lower Lipschitz bound:

\begin{align*}
|s_{B(0,r_{2}),B_{2}}(x)-s_{B(0,1),B_{2}}(y)| 
& \geq |x-y|-\frac{|x_{2}|\cdot||y|-|x||}{1-r_{2}}\\
& \geq |x-y|\ps{1-\frac{|x_{2}|}{1-r_{2}}}.
\end{align*}

Next, define,
\[s_{B(0,r_{1}),B(0,r_{2})}(x)=\isif{ \kappa(x,\frac{r_{2}}{r_{1}}x,\frac{|x|-r_{1}}{1-r_{1}}), & r_{1}<|x|<1 \\
\frac{r_{2}}{r_{1}}x, & |x|\leq r_{1}}\]
and finally
\[s_{B_{1},B_{2}}=s_{B(0,r_{2}),B_{2}}\circ s_{B(0,r_{1}),B(0,r_{2})}\circ s_{B(0,r_{1}),B_{1}}^{-1}.\]

\end{proof}

\subsection{The main interpolation lemma}

We now give the proof of Lemma \ref{l:interpolation-lemma}

Let $E_{j}=A_{j}(B_{j})$. Then
\[T_{B_{2}}\circ A_{2}^{-1}(E_{2})=B(0,1).\]

By eq. \eqn{ellipse-spacing},
\[E_{1}':=T_{B_{2}}\circ A_{2}^{-1}(E_{1})\subseteq B(0,\frac{1}{2})\]
and $E_{1}'$ is another ellipse. Let $B_{1}'$ be the smallest ball containing $E_{1}'$. Define a map $p$ on $\d B(0,2)\cup B(0,1)\cup B(0,1/2)$ by letting 
\begin{align}
p|_{B(0,2)} & =I\\
p|_{B(0,1/2)} & =A: = T_{B_{1}',B(0,\sigma)} \circ A_{2}^{-1}\circ A_{1}\circ T_{B(0,1/2),B_{1}}\\
p|_{B(0,1)} & =UV
\end{align}
where $A=UDV$ is the singular value decomposition of $A$. 

Applying Lemmas \ref{l:two-rotations} and \ref{l:ellipse-prop-2} to the domains $B(0,2)\backslash B(0,1)$ and $B(0,1)\backslash B(0,1/2)$ respectively, we get a bi-Lipschitz extension of $p$ to all of  $B(0,2)$, and
\[\Pi= (s_{T_{B_{2}}(B_{1}'),B(0,\sigma)}\circ T_{B_{2}}\circ  A_{2}^{-1})^{-1}\circ p \circ (T_{B_{2}}^{-1}\circ s_{B(0,1/2),T_{B_{2}}(B_{1})})^{-1}\]
is our desired bi-Lipschitz map from $B_{2}$ to $A(B_{2})$.

\bibliographystyle{amsalpha}
\bibliography{reference-1.bib}

\end{document}